\documentclass[10pt,a4paper]{article}
\usepackage{amsmath}
\usepackage{indentfirst}
\usepackage{geometry}
\usepackage{amsfonts}
\usepackage{amssymb}
\usepackage{amsthm}
\usepackage{color}
\usepackage{mathtools}
\newtheorem{Definition}{Definition}[section]
\newtheorem{Theorem}{Theorem}[section]

\newtheorem{Remark}[Theorem]{Remark}
\newtheorem{Lemma}[Theorem]{Lemma}
\newtheorem{Corollary}[Theorem]{Corollary}
\newtheorem{Proposition}[Theorem]{Proposition}
\numberwithin{equation}{section}
\newcommand{\bb}[1]{\mathbb{#1}}

\newcommand{\ds}{\displaystyle}
\newcommand{\cL}{\mathcal{L}}
\newcommand{\llan}{\left\langle}
\newcommand{\rran}{\right\rangle}
\newcommand{\sq}{\frac{1}{2}}

\title{Interior gradient and Hessian estimates for the Dirichlet problem of semi-linear degenerate elliptic systems: a probabilistic approach\footnotemark[1]}
\author{Jun Dai \footnotemark[2] \and Shanjian Tang\footnotemark[3] \and Bingjie Wu \footnotemark[2]}

\begin{document}
\maketitle
\begin{abstract}
In this paper, we give interior gradient and Hessian estimates for systems of semi-linear degenerate elliptic partial differential equations on bounded domains, using both tools of backward stochastic differential equations and quasi-derivatives.
\end{abstract}


\footnotetext[2]{School of Mathematical Sciences, Fudan University, Shanghai 200433, China (\textit{e-mail: 13110180052@fudan.edu.cn (Jun Dai), 13110840010@fudan.edu.cn (Bingjie Wu)}).}
\footnotetext[3]{Department of Finance and Control Sciences, School of
Mathematical Sciences, Fudan University, Shanghai 200433, China (\textit{e-mail: sjtang@fudan.edu.cn}).}
\footnotetext[1]{Partially supported by National Science Foundation of China (Grant No. 11631004)
and Science and Technology Commission of Shanghai Municipality (Grant No. 14XD1400400).   }
\section{Introduction}
Let $d, d_1$, and $k$ be given positive integers. Let $D$ be a bounded domain in a $d$-dimensional Euclidean space $\mathbb{R}^d$.
Consider the Dirichlet problem for a system of semi-linear degenerate elliptic partial differential equations (PDEs) of second order

\begin{equation}\label{eq:Par1}
 \begin{cases}
    \cL u(x)+f(x,u(x),\nabla u(x)\sigma(x))=0, \quad &x\in D;\\[8pt]
                 u(x)=g(x), \quad &x\in\partial D,
 \end{cases}
\end{equation}
where $u:=(u_1,\cdots,u_k)^*, \cL u:=(\cL u_1,\cdots,\cL u_k)^*,$
$$\cL u_m(x):=\sum_{i,j=1}^d a_{i,j}(x)\partial^2_{i,j}u_m(x)+\sum_{i=1}^d b_i(x)\partial_i u_m(x), \quad m=1,\cdots, k, $$
 with $a=\frac{1}{2}\sigma\sigma^*$. Here and in the following,  the asterisk $*$ in the superscript means the transpose.

Let $W$ be a $d_1$-dimensional Wiener process in the probability space $(\Omega,\mathcal{F},\mathbb{P})$ with $\{\mathcal{F}_t, t\geq 0\}$ being the augmented natural filtration. The probabilistic solution of (\ref{eq:Par1}) is given by Peng~\cite{PEN} as
\begin{equation}\label{eq:Sol1}
 u(x)=Y_0(x), \qquad x \in \overline{D},
\end{equation}
where $\{(Y_t(x),Z_t(x)), 0\le t\le \tau\}$ is the unique adapted solution to the backward stochastic differential equation (BSDE)
\begin{equation}\label{eq:BSDE1}
 \begin{cases}
    dY_t=-f(X_t(x),Y_t,Z_t)\, dt+Z_t\, dW_t, \qquad &t\in [0,\tau),\\[3mm]
    Y_{\tau}=g(X_{\tau}(x)),
 \end{cases}
\end{equation}
with $\{X_t(x), t\ge 0\}$ being the solution to the stochastic differential equation (SDE)
\begin{equation}\label{eq:SDE1}
 X_t=x+\int_0^t\sigma(X_s)dW_s+\int_0^tb(X_s)ds,
\end{equation}
and $\tau:=\tau(x):=\inf\{t>0, X_t(x)\notin D\}$ is the first exit time of $X_t(x)$ from $D$, under the assumption that the coefficients $b, \sigma, f$, and $g$ and the domain $D$ are all sufficiently smooth.

When the coefficients $(b, g, \sigma):\bb{R}^d\to\bb{R}^d\times \bb{R}^k\times \bb{R}^{d\times d_1}$ and $f: \bb{R}^d\times\bb{R}^k\times\bb{R}^{k\times d_1}\to\bb{R}^k$ are uniformly Lipschitz continuous in all of their arguments, one has to consider weak solutions for the associated PDEs.
One weak solution is the notion of viscosity solutions. The function $u$ defined by (\ref{eq:Sol1}) is shown by Darling and Pardoux~\cite{DAR} to be the unique continuous viscosity solution of \eqref{eq:Par1}, when $k=1$.
Another weak solution of PDEs is the notion of Sobolev solutions. Ouknine and Turpin \cite{OUK} gave a representation for the Sobolev solutions of degenerate parabolic PDEs through FBSDEs. They were inspired by the work of Bally and Matoussi \cite{BAL}, in which the Sobolev solutions of semi-linear SPDEs are described via BDSDEs. Later, Feng, Wang and Zhao \cite{FEN} studied the existence, uniqueness and the probabilistic representation of the Sobolev solutions of quasi-linear parabolic and elliptic PDEs in $\mathbb{R}^d$. We are interested in those conditions which yield  further regularity of $u$.

Using a deterministic approach, Caffarelli~\cite{CAF} obtained a priori $W^{2,p}$ estimates for the viscosity solutions of second order, uniformly elliptic, fully non-linear equations: $F(D^2u,x)=f(x)$ in a unit ball.
Freidlin \cite{FRE} obtained an early probabilistic result that the solution is smooth for degenerate linear elliptic equations (\ref{eq:Par1}) with $f(x)=0$, $x\in \overline{D}$ if the boundary data $g$ is sufficiently smooth. Peng \cite{PEN} showed that the classical solution of the nondegenerate quasi-linear elliptic PDE has a probabilistic interpretation $u(x)=Y_0(x)$, $x\in\overline D$. Darling and Pardoux~ \cite{DAR} proved that when $f$ is monotone in $y$, $u(\cdot)=Y_0(\cdot)$ is a bounded and continuous viscosity solution to the Dirichlet problem for a class of semi-linear elliptic PDEs. Later, Briand and Hu \cite{BRI} gives a stability result for BSDEs with random terminal time which is associated to a system of semi-linear elliptic PDEs by partially relaxing the monotonicity assumption on the coefficient.

Moreover, to obtain in a probabilistic way the gradient estimates of the solution of second order PDEs, the now well-known theory of stochastic flows plays a crucial role. For example, Pardoux and Peng \cite{PAR} used the tool of BSDEs to investigate the regularity properties of the solution of parabolic PDEs, and they proved that the solution of BSDEs $\{Y_s^{t,x},\ (s,t,x)\in [0,T]^2\times\mathbb{R}^d\}$ has a version whose trajectories belong to $C^{0,0,2}([0,T]^2\times \mathbb{R}^d)$. Tang \cite{TAN} extended their context to incorporate random coefficients, and proved that, when $f(t,x,y,z)$ is linear in $z$, the regularity of the solution of BSDEs can be derived from those of the coefficients of FBSDEs. These works take advantage of the Cauchy problem where the space variable takes values over the whole space. The methodology is difficult to be adapted to the Dirichlet problem of elliptic PDEs in a bounded domain: estimating the gradient of the function $u$ through directly differentiating the expression (\ref{eq:Sol1}) involves the differentiation of the exit time $\tau(x)$ with respect to $x$, while the function $\tau(x)$ is not necessarily differentiable with respect to $x$. To get around such a difficulty, Delarue \cite{DEL} established a priori H\"{o}lder estimate of Krylov and Safonov type for the viscosity solution of a degenerate quasi-linear elliptic PDE, where the H\"{o}lder bound does not depend on the regularity of $\sigma$ and $f$. He extended that of Krylov and Safonov \cite{KRY9}, by building a special type of SDEs with $\sigma$ depending on both $u$ and its gradient $\nabla u$.

Alternative powerful tool is that of quasi-derivative. It was first introduced by Krylov \cite{KRY1}, to find a different condition on coefficients such that $c$ is sufficiently large compared to first derivatives of $\sigma$ and $b$ with respect to $x$ under which $u$ is twice continuously differentiable in $\mathbb{R}^d$. This condition weakens the known conditions mentioned in \cite[page 257]{KRY6} where no quasi-derivative is used. Since then this technique has been applied to investigate the smoothness of solution of various elliptic and parabolic PDEs. Krylov \cite{KRY2,KRY4,KRY8} applied different quasi-derivative methods to study the interior regularity of harmonic functions of degenerate elliptic operators. Later Krylov \cite{KRY3} obtained $C^{1,1}$-regularity of the solution up to the boundary for the Dirichlet problem of degenerate Bellman equations under the boundary value assumption that $g\in C^4(\overline{D})$ (see \cite[Assumption 1.3, page 67]{KRY3} where $g$ should be required to lie in $C^4(\overline{D})$, though there it was only supposed to lie in $C^3(\overline{D})$; and see  Krylov's own exposition~\cite[page 2]{KRY8} for this point), which holds true for the degenerate linear elliptic equations (see \cite[Theorem 2.1, page 74]{KRY3}). The ideas are based on adding a $4$-dimensional process $y_t$ to the original $d$-dimensional processes $x_t$ (see \cite[page 83]{KRY3}) such that the augmented  process $z_t=(x_t,y_t)$ never leaves a surface in $\mathbb{R}^{d+4}$. In this way, he can get rid of the dependence of the first exit time on the initial point and use the techniques of \cite{KRY1} to obtain moment estimates of quasi-derivatives in the whole space. Recently, Zhou \cite{WEI1} introduced the notion of  the second quasi-derivative to estimate the derivatives up to the second order of $u$ inside the domain under the weaker boundary value assumption that $g\in C^{1,1}(\overline{D})$. Under the weaker boundary regularity assumption of $g\in C^{0,1}(\overline{D})$,  it has been illustrated (see, for instance, \cite[page 58-63]{KRY8}) that, the first-order derivatives of $u$ fail to be bounded up to the boundary in both PDE methods and quasi-derivative methods, even for the Laplacian equation (i.e., $\cL=\Delta$). He commented that for $g\in C^{1,1}(\overline{D})$, one can only expect to prove interior $C^{1,1}$-regularity (see \cite[page 3065]{WEI1}). His proof relied on a probabilistic interpretation of the linear degenerate elliptic PDEs. He introduced two local martingales with the help of quasi-derivatives and their auxiliary processes to formulate first and second derivatives of $u$, respectively (see \cite[Theorem 2.2]{WEI1}). Besides, instead of adding four more dimensions as \cite[page 83]{KRY3}, he constructed two families of local super-martingales to bound the moments of quasi-derivatives near the boundary and in the interior of the domain, respectively (see \cite[Lemmas 3.3 and 3.4]{WEI1}). All these existing works which employ the method of quasi-derivatives discussed either the linear second-order PDEs or the so-called Bellman equation (which is a fully nonlinear PDE) arising from optimal stochastic control problems.

In our context for a $k$-dimensional vector-valued nonlinear function $f$, we use both tools of BSDEs and quasi-derivatives to establish the gradient and Hessian estimates for the solution $u$  to the Dirichlet problem for a system of semi-linear degenerate second-order partial differential equations (\ref{eq:Par1}).

Our objective is to establish the counterpart of Zhou's estimates~\cite[Theorem 3.1]{WEI1} for a system of semi-linear elliptic PDEs, which is precisely stated in Theorem~\ref{thm:estimate} at the end of Section \ref{sec:ass} below.

 In contrast to Zhou \cite{WEI1}, we have new difficulties. In fact, for the gradient estimate, we need to calculate the difference $|Y_0^\delta(x+\delta\xi_0)-Y_0(x)|$ between the solutions of the perturbed BSDE (in (\ref{eq:FBSDE1})) and the unperturbed one and appeal to the BSDE estimates. As a consequence, new barrier functions are introduced near the boundary and in the interior of the domain (see Lemmas \ref{le:quasiestimate1} and \ref{le:quasiestimate2}), so as to bound higher moment estimates of quasi-derivatives, which leads to that the process $(\psi_{(\xi_t)}\psi^{-1})^p$ is considered in a better space. For the Hessian estimate, we estimate the second-order difference $|Y_0^\delta-2Y_0+Y_0^{-\delta}|$. To deal with the nonlinearity of $f$ in $\nabla u$, we use the technical skills developed in the estimates of BSDEs by Pardoux and Peng~\cite[Theorem 2.9]{PAR}. Finally, we emphasize that we consider a system of semi-linear degenerate elliptic PDEs rather than a single equation, where $u$, $f$ and $g$ take values in $\bb{R}^k$, although it must be said that our interior estimates are not sharper than those of \cite[Theorem 3.1]{WEI1} in some sense.

The paper is organized as follows: In Section \ref{sec:ass}, we set notations and list the standing assumptions. Then we introduce some standard estimates for the solution of random terminal BSDEs, and recall the concept of the quasi-derivative and some known basic results.  We end up with the statement of our main results. In Section \ref{sec:appl}, we build four barrier functions to get some moment estimates of the quasi-derivatives and derive generalized assertions at last. In Section \ref{sec:der}, we use the BSDE estimates to establish the interior gradient and Hessian estimates of $u$ in (\ref{eq:Par1}) under the aforementioned assumptions, and then show the existence and uniqueness of $u$ in (\ref{eq:Par1}).


\section{Preliminaries and Statement of the Main Results}\label{sec:ass}
Let $\mathcal{A}:=\{\alpha=(\alpha_1,\cdots,\alpha_d): \alpha_i, i=1,\cdots,d\ \text{are nonnegative integers}\}$ be the set of multi-indices. For any $\alpha\in\mathcal{A}$ and $x=(x_1,\cdots,x_d)\in\mathbb{R}^d$, denote
\begin{equation*}
  |\alpha|:=\sum\limits_{i=1}^n \alpha_i,\quad \partial^\alpha:=\partial_1^{\alpha_1}\partial_2^{\alpha_2}\cdots\partial_d^{\alpha_d}.
\end{equation*}

In a Euclidean space $\mathbb{E}$, denote by $\langle\cdot,\cdot\rangle$ the inner product, and the norm by $|\cdot|_{\mathbb{E}}$ or simply by $|\cdot|$ when no confusion is made. Let $\mathfrak{B}$ be the set of all skew-symmetric $d_1\times d_1$ matrices. Denote by $\|A\|$ the norm of a matrix $A$, which is defined to be the square root of the sum of all the squared components, i.e. $\|A\|^2$ is the trace of $AA^*$.

Denote by $C(\overline{D})$ the Banach spaces of continuous functions $g$ in $\overline{D}$ equipped with the norm
$$|g|_0=\sup\limits_{x\in \overline{D}}|g(x)|,$$
and by $C^m(\overline{D})$ with $m=1,2$ the Banach spaces of once (for $m=1$) or twice (for $m=2$) continuously differentiable functions $g$ in $\overline{D}$ equipped with the respective norm:
\begin{equation*}
 |g|_1:=|g|_0+|g_x|_0, \quad |g|_2:=|g|_1+|g_{xx}|_0,
\end{equation*}
where $g_x$ is the gradient vector of $g$, and $g_{xx}$ is the Hessian matrix of $g$. For $\beta\in (0,1]$, the H\"{o}lder space $C^{m,\beta}(\overline{D})$ is the Banach subspace of $C^m(\overline{D})$ consisting of all functions $g$ with the norm
\begin{equation*}
 |g|_{m,\beta}:=|g|_m+[g]_{m,\beta},\quad \text{where}\ [g]_{m,\beta}:=\sum\limits_{|\alpha|=m}\sup\limits_{x,y\in \overline{D}}\frac{|\partial^\alpha g(x)-\partial^\alpha g(y)|}{|x-y|^\beta}.
\end{equation*}

For a function $f\in C^{0,1}(\overline{D}\times\mathbb{R}^k\times\mathbb{R}^{k\times d_1})$, define
\begin{equation*}
  \|f(\cdot)\|_{0,1}:=|f(\cdot,0,0)|_0+[f]_{0,1,x},\quad \text{where}\ [f]_{0,1,x}:=\mathop{\sup\limits_{x,x'\in \overline{D}}}_{(y,z)\in\mathbb{R}^k\times\mathbb{R}^{k\times d_1}}\frac{|f(x,y,z)- f(x',y,z)|}{|x-x'|}.
\end{equation*}
For a function $f\in C^{1,1}(\overline{D}\times\mathbb{R}^k\times\mathbb{R}^{k\times d_1})$, define
\begin{equation*}
  [f]_{1,1}:=\sup\limits_{\Xi,\Xi'\in\overline{D}\times\mathbb{R}^k\times\mathbb{R}^{k\times d_1}}\frac{|\partial f(\Xi)-\partial f(\Xi')|}{|\Xi-\Xi'|}.
\end{equation*}

Denote by $H_{k,\rho}^2(\overline{D})$ the weighted Sobolev space, equipped with the norm:
\begin{equation*}
 |\varphi|_{k,\rho}:=\left\{\sum\limits_{0\leq|\alpha|\leq k}\int_D |\partial^\alpha \varphi(x)|^2\rho^{-1}(x)dx\right\}^{\frac{1}{2}},
\end{equation*}
where $\rho(x):=(1+|x|^2)^q,\ q\geq 2$ is a weight function.

 For $\{\mathcal{F}_t\}$-stopping time $\tau$ and some real number $\beta$, $\mathcal{M}_\beta(0, \tau; V)$ denotes the Hilbert space of all progressively measurable processes $X$ taking values in the Euclidean space $V$, such that
\begin{equation*}
  \|X\|_{\mathcal{M}_\beta}:=E\left[\int_0^\tau e^{2\beta s}|X_s|^2ds\right]^{\frac{1}{2}}<\infty.
\end{equation*}

Denote by $\mathbb{M}^m(D, \sigma, b)$ the set of real-valued $m$-times continuously differentiable functions given on $D$ such that for any $v\in\mathbb{M}^m(D, \sigma, b)$ the process $\{v(X_t), [0,\tau]\}$ is a local martingale relative to $\mathcal{F}_t$ for any $x\in D$. We write $\mathbb{M}^m$ for $\mathbb{M}^m(D, \sigma, b)$ whenever no confusion is made.

For $y,z\in \mathbb{R}^d$ and the $d_2$-dimensional column vector function $\phi$, set $\phi:=(\phi_1,\cdots, \phi_{d_2})^*$, and
\begin{equation*}
  \phi_{(y)}:=\left(\sum_{i=1}^d \partial_i \phi_m\cdot y_i\right)^*_{1\le m\le d_2}, \quad \phi_{(y)(z)}:=\left(\sum_{i,j=1}^d \partial^2_{ij} \phi_m\cdot y_i z_j\right)^*_{1\le m \le d_2}.
\end{equation*}

 Write $E_x$ for the expectation of a functional of the underlying process which takes value $x$ at the initial time $0$, and $N(K_1, K_2,\cdots)$ for a constant $N$ to indicate its dependence on $K_1, K_2, \cdots$ whenever necessary.

We introduce the following assumptions with  constants $p=1,2$ and $q=0,1$.
The assumptions $(H1)-(H3)$ are necessary for the well-posedness of solutions to SDEs.

\medskip
$(H1)$ $\sigma$ and $b$ are twice continuously differentiable in $\mathbb{R}^d$.

\medskip
$(H2)$ The domain $D\in C^4$ is bounded  in $\mathbb{R}^d$. There is a function $\psi\in C^4$ such that
(i) $\psi(x) >0 \  \hbox{\rm for } x\in D,$   (ii)  $\psi(x)=0$ and $|\psi_x(x)|\geq1$  for $x\in\partial D$,
and (iii) the following inequality holds true:
\begin{equation}\label{eq:psi}
 \cL\psi(x):=\sum_{i,j=1}^da_{ij}(x)\partial^2_{ij}\psi(x)+ \sum_{i=1}^db_i(x)\partial_i\psi(x)\leq-1 \quad  \hbox{ \rm for } x\in D.
\end{equation}
In what follows, we write $D_\lambda:=\{x\in D: \psi(x)>\lambda\}$ for $\lambda\in(0,1)$.

\medskip
$(H3)$ There exists a constant $K_0>0$, such that
\begin{equation*}
\sum_{i=1}^d\sum_{j=1}^{d_1}|\sigma_{ij}|_2+\sum_{i=1}^d|b_i|_2+|\psi|_4 \leq K_0.
\end{equation*}

The assumptions $(H4)-(H8)$ are necessary for the smoothness.

\medskip
$(H4)$ $f\in C^{0,1}(\overline{D}\times\mathbb{R}^k\times\mathbb{R}^{k\times d_1})$, and there exist constants $L$, $L_0>0$, such that
\begin{equation*}
|f(x,y,z)-f(x,\bar{y},\bar{z})|\leq L|y-\bar{y}|+L_0|z-\bar{z}|,  \quad  x\in \overline{D};  y, \bar{y}\in \mathbb{R}^k; z, \bar{z}\in\mathbb{R}^{k\times d_1}.
\end{equation*}

\medskip
$(H5)$ There exists a constant $\mu\in\mathbb{R}$ such that
\begin{equation*}
\langle f(x,y,z)-f(x,\bar{y},z),(y-\bar{y})\rangle\leq -\mu |y-\bar{y}|^2, \quad x\in \overline{D}; y, \bar{y}\in \mathbb{R}^k; z\in\mathbb{R}^{k\times d_1}.
\end{equation*}

\medskip
$(H6)_q$ $g\in C^{q,1}(\overline{D})$.

\medskip
$(H7)$ There exists constants $\beta$ and $\vartheta$ such that
\begin{equation*}
0<\mu<L,\ -\mu+2L_0^2<2\beta<0,\ \text{and}\ 2\vartheta=-2\mu+ L_0^2,
\end{equation*}
where $L$ and $L_0$ are the Lipschitz constants of $f$ with respect to $y$ and $z$ respectively in $(H4)$, and $\mu$ is the monotonicity constant of $f$ in $(H5)$.

\medskip
$(H8)$ $f\in C^{1,1}(\overline{D}\times\mathbb{R}^k\times\mathbb{R}^{k\times d_1})$, and for any $(x,y,z)\in\overline{D}\times\mathbb{R}^k\times\mathbb{R}^{k\times d_1}$, $f_y(x,y,z)\leq -\mu$.
%

The assumptions $(H9)$ and $(H10)_p$ are necessary for controlling the moments of quasi-derivatives.

\medskip
$(H9)$ The inequality $\langle a n,n\rangle>0$ holds for any unitary normal vector $n$ at $\partial D$.

\medskip
$(H10)_p$ There exist functions $(\rho, M): D\to \mathbb{R}^d\times \mathbb{R}$ and $Q(\cdot,\cdot): D\times\mathbb{R}^d\to\mathfrak{B}$, such that
(i)  $(\rho, M)$ is bounded in $D_\lambda$ for any $\lambda\in(0,1)$, and (ii)  $Q(\cdot,y)$ is  bounded in $D_\lambda$ for any $(\lambda, y) \in (0,1)\times \mathbb{R}^d$ and $Q(x,\cdot)$ is a linear function for any $x\in D$.
 Furthermore,  we have for $\beta$ satisfying $(H7)$
  \begin{eqnarray}
  &&\ds 2p(4p-1) \left\| \sigma_{(y)}(x)+\langle\rho(x),y\rangle\sigma(x)+\sigma(x)Q(x,y)\right\|^2+4p\left\langle y,b_{(y)}(x)+2\langle\rho(x),y\rangle b(x)\right\rangle\nonumber\\
  &\leq&\ds (-4p\beta-1)+2pM(x)\langle a(x)y,y\rangle, \quad \forall (x, y)\in D \times \mathbb{R}^d \hbox{ \rm with } |y|=1.\label{Assump2}
  \end{eqnarray}

\begin{Remark}\label{remark:ass}
 \rm{(i)} It is easy to see that $(H4)$ implies $(H5)$ for $L\le -\mu$. Thus, $(H5)$ gives some additional restriction only for $L>-\mu$ (see \cite[page 363]{YON}).

 \rm{(ii)} For $f\in C^1(\overline{D}\times\mathbb{R}^k\times\mathbb{R}^{k\times d_1})$, conditions $(H4)$ and $(H5)$ yield $\|f_z\|\leq L_0$.

\rm{(iii)} $(H4)$ implies that $f$ is continuous in $x$ and thus $|f(\cdot,0,0)|_0<\infty$, and moreover that $f$ is Lipschitz continuous in $(x,y,z)$. Consequently, $(H4)$ implies that $\|f(\cdot)\|_{0,1}<\infty$. In addition to the assumption that  all the first-order partial derivatives of $f$ are globally Lipschitz continuous in $(x,y,z)$, $(H8)$ implies that $[f]_{1,1}<\infty$.


 \rm{(iv)} $(H9)$ and $(H10)_p$ are conditions for guaranteeing that the moments of quasi-derivatives near the boundary or in the interior of the domain do not grow too fast. $(H9)$ implies $a$ is non-degenerate along the normal to the boundary. However, if $\sigma$ is a constant and $D$ is a bounded domain, $(H9)$ implies that $a$ is uniformly non-degenerate. $(H10)_p$ is weaker than the non-degenerate condition. Indeed, assume that $M=1,\ \rho=Q=0,\ p=\frac{1}{2}$ and $d=d_1=1$ for the sake of simplicity, then we will have $2\beta+1+2b'(x)+|\sigma'(x)|^2\leq a(x)$, where the sum of the terms on the left hand side of the inequality may be negative. As for the necessity, we have to admit that $(H10)_{\frac{1}{2}}$ is stronger than \cite[Assumption 3.2]{WEI1} for the linear case. But our $(H10)_p$ can be used for higher moment estimates of quasi-derivatives. In fact, if we take $f=cu$, $\sigma=x$ and $b=b_1x$, where $c$ and $b_1$ are constants, then as \cite[Remark 3.2]{WEI1}, the condition $1+2b_1\leq c$ is necessary for $u$ having Lipschitz continuous derivatives. However, when $p=\frac{1}{2}$, the conditions $(H4)$, $(H5)$, $(H7)$ and $(H10)_{\frac{1}{2}}$ imply $2\beta>-c$ and $1+2b_1<c$, which are stronger than the necessary condition in \cite[Assumption 3.2]{WEI1}.

\end{Remark}

Note that throughout the paper constants $K$ and $N$ may differ in different inequalities.

\subsection{BSDEs in Random Durations Revisited}\label{sec:bsde}
BSDEs with random terminal times have been studied by Peng \cite{PEN}, Darling and Pardoux \cite{DAR}, and Briand and Hu \cite{BRI}. See also Yong and Zhou \cite[page 360]{YON} for a relevant exposition and related references therein. In this subsection, we give some priori estimates for the solutions of BSDEs and represent the solutions of a system of second order semi-linear elliptic PDEs through BSDEs.

We consider the It\^{o} stochastic equation
\begin{equation}\label{eq:SDE2}
 X_t=x+\int_0^t\sigma(X_s)\, dW_s+\int_0^tb(X_s)\, ds, \quad t\ge 0.
\end{equation}
In view of $(H1)$, it has a unique solution $\{X_t, t\ge 0\}$ for any $x\in D$.
  We have the following four inequalities: $E\tau(x)<\infty$ (from $(H2)$ and Lemma \ref{le:stopping}), $|f(\cdot,0,0)|<\infty$ (from $(H4)$), $|g|_0<\infty$ (from $(H6)_0$) and $\vartheta<0$ (from $(H7)$), all of which yield  the following key assumption of \cite[Theorem 3.4]{DAR}: for some $\varrho\in(\vartheta,0)$,
\begin{equation}\label{eq:estin1_in_remark}
 E_{x}\left[ e^{2\varrho \tau}\left|g(X_\tau)\right|^2+\int_0^\tau e^{2\varrho s}|f(X_s,0,0)|^2ds \right]<\infty, \quad \forall x\in D.
\end{equation}
Therefore,  according to \cite[Theorem 3.4]{DAR}, the following BSDE
\begin{equation}\label{eq:BSDE2}
 \left\{
  \begin{aligned}
    dY_t=&-f(X_t,Y_t,Z_t)dt+Z_tdW_t, \quad t\in [0,\tau);\\
    Y_{\tau}=&g(X_{\tau})
  \end{aligned}
 \right.
\end{equation}
 has a unique pair $(Y_\cdot,Z_\cdot)\in \mathcal{M}_\vartheta(0,\tau; \mathbb{R}^k\times\mathbb{R}^{k\times d_1})$ if $(H2)$, $(H4)$, $(H5)$, $(H6)_0$ and $(H7)$ are all satisfied.

\begin{Lemma}\label{le:BSDE}\cite[Proposition 4.3 and Corollary 4.4.1]{DAR}
  Assume that $f$ is Lipschitz with respect to $(y,z)$ (also made in $(H4)$). Let the assumption $(H5)$ and the inequality \eqref{eq:estin1_in_remark} be satisfied for some $\varrho>\vartheta$. Then BSDE (\ref{eq:BSDE2}) admits a unique adapted solution $(Y_\cdot,Z_\cdot)\in \mathcal{M}_\vartheta(0,\tau; \mathbb{R}^k\times\mathbb{R}^{k\times d_1})$. Moreover, there exists a constant $K>0$ such that
 \begin{equation*}
   \|(Y_\cdot,Z_\cdot)\|^2_{\mathcal{M}_\vartheta [0,\tau]}\leq KE\left[\left|g(X_\tau)\right|^2e^{2\gamma\tau}\right]+KE\int_0^\tau e^{2\gamma s}\left|f(X_s,0,0)\right|^2ds<\infty,
 \end{equation*}
 and for any $p\geq 2$,
 \begin{equation}\label{eq:estin10}
   E\left[\sup\limits_{t\in[0,\tau]}|Y_t|^p\right]+E\int_0^\tau|Y_s|^pds+E\left[\left(\int_0^\tau \|Z_s\|^2ds\right)^\frac{p}{2}\right] \le K(|g|_0^p+|f(\cdot,0,0)|_0^p).
 \end{equation}
\end{Lemma}

\begin{Remark}
As was shown in \cite{BRI} for the one-dimensional case of BSDE \eqref{eq:BSDE2}, there is a unique solution $(Y_\cdot,Z_\cdot)\in \mathcal{M}_{-\mu}(0,\tau; \mathbb{R}^k\times\mathbb{R}^{k\times d_1})$ without the `structural' conditions  on the coefficient $f$ in $(H7)$,  which links the constant $\mu$ of monotonicity to the Lipschitz constant $L_0$ of $f$ in $z$.
\end{Remark}
\begin{Lemma}\label{le:BSDE3}\cite[Lemma 6.2 and Proposition 6.3]{DAR}
 Under the assumptions of Lemma \ref{le:BSDE}, if we set $u(x):=Y_0(x)$ for $x\in\overline{D}$, then we have $Y_t=u(X_t)$ for $0\leq t\leq \tau$ and $u$ is bounded and continuous in $\overline{D}$.
\end{Lemma}

\subsection{Introduction of Quasi-derivative}\label{sec:intro}
Conventionally, to obtain the gradient and Hessian estimates of $u$ in a probabilistic approach, we differentiate formula (\ref{eq:Sol1}) with respect to $x$. For the Dirichlet problem of elliptic equations in a domain, a crucial  trouble is that the first exit time $\tau=\tau(x)$ is not necessarily continuous (let alone the differentiability). To overcome this difficulty, we introduce the so-called first and second quasi-derivatives of $X_t$ with respect to $x$ along the vector $\xi_0$ and $\eta_0$ respectively.

 \begin{Definition}
 Let $x\in D$, $\xi_0\in\mathbb{R}^d$, and $(\xi_t, \xi_t^0)$ be adapted continuous processes defined on $[0,\tau]$ taking values in $\mathbb{R}^d\times \mathbb{R}$ such that $\xi_t|_{t=0}=\xi_0$. We call $\xi_t$ a first quasi-derivative of $X_t$ along the direction $\xi_0$ at point $x$ if the following process
 \begin{displaymath}
  v_{(\xi_t)}(X_t)+\xi_t^0 v(X_t), \quad 0\le t\le \tau
 \end{displaymath}
 is a local martingale for any $v\in\mathbb{M}$, and the associated process $\xi_t^0$ is called a first adjoint process of $\xi_t$.

 Additionally, let $\eta_0\in\mathbb{R}^d$, and $(\eta_t, \eta_t^0)$ be adapted continuous processes defined on $[0,\tau]$ taking values in $\mathbb{R}^d\times \mathbb{R}$ such that $\eta_t|_{t=0}=\eta_0$. We call $\eta_t$ a second quasi-derivative of $X_t$ associate with $\xi_t$ and $\xi_t^0$ along the direction of $\eta_0$ at point $x$ if the following process
 \begin{displaymath}
  v_{(\xi_t)(\xi_t)}(X_t)+v_{(\eta_t)}(X_t)+2\xi_t^0v_{(\xi_t)}(X_t)+\eta_t^0v(X_t), \quad 0\leq t\leq \tau
 \end{displaymath}
 is a local martingale for any $v\in\mathbb{M}^2$, and the associated process $\eta_t^0$ is called a second adjoint process of $\eta_t$.
 \end{Definition}

The notion of the first quasi-derivative can be found in \cite[Definition 2.2]{KRY4} and \cite[Definition 3.1.1]{KRY8}, while the notion of the second quasi-derivative can be found in \cite[Definition 2.3]{WEI1}. The following  examples of the first quasi-derivative $\xi_t$ and the second quasi-derivative $\eta_t$ can be found in \cite[Theorem 2.1]{WEI1}.

\begin{Lemma}\label{le:Qua}
 Let the scalar processes  $r_t$ and $\tilde{r}_t$, the $\mathbb{R}^{d_1}$-valued processes $\pi_t$ and $\tilde\pi_t$, and the $\mathfrak{B}$-valued processes $P_t$ and $\tilde P_t$ be all progressively measurable such that for any finite positive time $T$,
 \begin{equation}\label{ass:co1}
  \int_0^T\left(|(r_t, \pi_t, P_t)|^4+|(\tilde r_t, \tilde\pi_t, \tilde P_t)|^2\right)dt< \infty.
 \end{equation}
For $x \in D$ and $(\xi_0, \eta_0)\in \mathbb{R}^d\times \mathbb{R}^d$, denote by the processes $\xi_t$ and $\eta_t$  solutions of the following linear SDEs: for $t\in [0,\infty),$
 \begin{eqnarray}
  \xi_t &=&\ds \xi_0+\int_0^t[\sigma_{(\xi_s)}+r_s\sigma+\sigma P_s]\, dW_s+\int_0^t[b_{(\xi_s)}+2r_sb-\sigma\pi_s]\, ds,  \label{eq:Qua1}\\
  \eta_t &=&\ds \eta_0+\int_0^t\!\!\![\sigma_{(\eta_s)}+\tilde r_s\sigma+\sigma\tilde P_s+\sigma_{(\xi_s)(\xi_s)}+2r_s\sigma_{(\xi_s)}-r_s^2\sigma+2\sigma_{(\xi_s)}P_s+2r_s\sigma P_s+\sigma P_s^2]\, dW_s\nonumber\\
  & &\ds+\int_0^t\!\!\![b_{(\eta_s)}+2\tilde r_s b-\sigma\tilde\pi_s+b_{(\xi_s)(\xi_s)}+4r_s b_{(\xi_s)}-2\sigma_{(\xi_s)}\pi_s-2r_s\sigma\pi_s-2\sigma P_s\pi_s]\,ds,\nonumber\\
   \label{eq:Qua3}
 \end{eqnarray}
 where in $\sigma, b$ and their derivatives we have dropped the argument $X_s$.
 The processes $\xi_t^0$ and $\eta_t^0$ can be taken to be
 \begin{eqnarray}
   \xi_t^0 &=& \int_0^t \pi_sdW_s, \\
   \eta_t^0 &=& (\xi_t^0)^2-\langle\xi^0\rangle_t+\int_0^t \tilde\pi_sdW_s.
 \end{eqnarray}
 Then $\xi_t$ is a first quasi-derivative of $X_t$ along the direction of $\xi_0$ at $x$ and $\xi_t^0$ is a first adjoint process for $\xi_t$, and $\eta_t$ is a second quasi-derivative of $X_t$ associated with $\xi_t$ along the direction of $\eta_0$ at $x$ and $\eta_t^0$ is a second adjoint process for $\eta_t$.
\end{Lemma}
The proof of Lemma \ref{le:Qua} can be found in \cite[Lemma 3.1]{KRY4} and \cite[Theorem 2.1]{WEI1}. Indeed the auxiliary process $(r_t, \tilde r_t)$ relates with a time-change. The process $(\pi_t,\tilde\pi_t)$ relates with a measure-transformation via Girsanov's theorem, and the processes $(P_t,\tilde P_t)$ relates with a rotation of the driving Wiener process. Since all these transformations preserve the property of the local martingale and quasi-derivatives have additivity, we can easily arrive at the above results.

We also find that the quasi-derivatives $\xi_t$ and $\eta_t$ enjoy some freedom due to the presence of these auxiliary processes. Hence, sometimes we can turn the quasi-derivatives in such a way that they become tangent to the boundary when and where $X_t$ hit it (see examples in \cite[page 54-58]{KRY8}). In this case, it remains for us to estimate the moments of the quasi-derivatives, since the directional derivatives of $u$ along the quasi-derivatives on the boundary coincide with that of the boundary data $g$.


Let $\delta$ be a small positive constant. Consider the following forward-backward stochastic differential equation (FBSDE)
\begin{small}\begin{equation}\label{eq:FBSDE1}
 \left\{
  \begin{array}{rcl}
    dX_t^\delta&=&\displaystyle \left[\left(1+2\delta r_t+\delta^2\tilde r_t\right)b(X_t^\delta)-(1+2\delta r_t+\delta^2\tilde r_t)^\sq\sigma(X_t^\delta)\left(\delta\pi_t+\frac{1}{2}\delta^2\tilde\pi_t\right)e^{\delta P_t}e^{\frac{1}{2}\delta^2\tilde P_t}\right]dt\\[4mm]
    &&\displaystyle +(1+2\delta r_t+\delta^2\tilde r_t)^\sq\sigma(X_t^\delta)e^{\delta P_t}e^{\frac{1}{2}\delta^2\tilde P_t}dW_t,\quad t\in [0,\tau^\delta];\\[4mm]
   dY_t^\delta&=&\displaystyle\left[-f(X_t^\delta,Y_t^\delta,Z_t^\delta)\left(1+2\delta r_t+\delta^2\tilde r_t\right)-\tilde Z_t^\delta\left(\delta\pi_t+\frac{1}{2}\delta^2\tilde\pi_t\right)\right]dt+\tilde Z_t^\delta\, dW_t,\quad t\in [0,\tau^\delta],\\[4mm]
   \tilde Z_t^\delta&:=&\displaystyle Z_t^\delta(1+2\delta r_t+\delta^2\tilde r_t)^\sq e^{\delta P_t}e^{\frac{1}{2}\delta^2\tilde P_t}; \quad\quad
   X_0^\delta=\displaystyle x+\delta\xi_0+\frac{1}{2}\delta^2\eta_0, \quad
   Y_{\tau^\delta}^\delta=\displaystyle g(X^\delta_{\tau^\delta}),
  \end{array}
 \right.
\end{equation}\end{small}
where $\tau^\delta$ is the first exit time of $X_t^\delta$ from $D$.

\begin{Remark} (i) As $e^{\delta P_t}$ is an orthogonal matrix,
 $d\tilde W_t=\int_0^t e^{\delta P_s} d W_s$
is a Wiener process for any $\delta$.   (ii) Note that $(X^0, Y^0, Z^0)=(X, Y, Z)$ is the solution of (\ref{eq:SDE2}) and (\ref{eq:BSDE2}). Therefore, $(X^\delta, Y^\delta, Z^\delta)$ is a perturbation to $(X, Y, Z)$.
\end{Remark}

\begin{Lemma}\label{le:BSDE2}
 Let $(H1)$-$(H5)$, $(H6)_0$ and $(H7)$ be satisfied. Without loss of generality, we may assume the coefficients $(r_t, \pi_t, P_t)$ and $(\tilde r_t, \tilde\pi_t, \tilde P_t)$ be bounded if condition (\ref{ass:co1}) is satisfied. Then, there is a sufficiently small $\delta$ such that (see \cite[page 520]{KRY1})
 $$0\leq1+2\delta r_t+\delta^2\tilde r_t \leq 2,\quad |\delta\pi_t|+{1\over2}\delta^2|\tilde\pi_t|\leq 1, \quad \exp\left(\delta P_t+{1\over2}\delta^2\tilde P_t\right)\leq 2,$$
  and FBSDE (\ref{eq:FBSDE1}) has a unique solution $(X^\delta, Y^\delta, Z^\delta)$.

 Moreover, when $\{X_t(x), t\geq 0\}$ is a unique solution of (\ref{eq:SDE2}), and $\{(\xi_t(\xi_0), \eta_t(\eta_0)), t\geq 0\}$ is a unique solution of (\ref{eq:Qua1}) and (\ref{eq:Qua3}), for $p\geq 2$, $T\geq 1$, and $(x, \xi_0, \eta_0)\in D\times\mathbb{R}^d\times\mathbb{R}^d$, we have (see \cite[Theorems 3.1 and 3.2]{WEI2}),
 \begin{equation}\label{eq:estin11}
  \lim\limits_{\delta\to 0}E\left[\sup\limits_{0\leq t\leq \tau^\delta\wedge\tau\wedge T}\left|\frac{X_t^\delta(x+\delta\xi_0)-X_t(x)}{\delta}-\xi_t\right|^p\right]=0,
 \end{equation}
 and
 \begin{equation}\label{eq:estin12}
  \lim\limits_{\delta\to 0}E\left[\sup\limits_{0\leq t\leq \tau^\delta\wedge\tau^{-\delta}\wedge\tau\wedge T}\left|\frac{X_t^\delta(x+\delta\xi_0+\frac{1}{2}\delta^2\eta_0)-2X_t(x)+X_t^{-\delta}(x-\delta\xi_0+\frac{1}{2}\delta^2\eta_0)}{\delta^2}-\eta_t\right|^p\right]=0.
 \end{equation}
\end{Lemma}

\begin{Lemma}\label{le:BSDE4}
 Under the assumptions of Lemma \ref{le:BSDE2}, we have
 \begin{displaymath}
  u(X_t^\delta)=Y_t^\delta, \quad t\in [0,\tau^\delta].
 \end{displaymath}
\end{Lemma}
\begin{proof}
 The proof is similar to that of Lemma \ref{le:BSDE3}. Thanks to It\^{o}'s formula, we get $u^\epsilon (X_t^{\epsilon,\delta})=Y_t^{\epsilon,\delta}$ by regularizing the equations with smooth coefficients. Then, the desired result is a consequence of the stability of BSDEs provided in \cite[Theorem 2.4]{BRI} and the stability of weak solutions of degenerate elliptic PDEs in \cite[Theorem 4.6.1]{KRY6}.
\end{proof}

%
%
%
%

Our main result is stated in the following theorem.

\begin{Theorem}\label{thm:estimate}
 Let assumptions $(H1)$-$(H5)$, $(H7)$ and $(H9)$ be satisfied. Let $(X, Y)$ be the unique solution of (\ref{eq:SDE1}) and (\ref{eq:BSDE1}), and $u$ be defined by (\ref{eq:Sol1}). Then we have the following assertions.

 {\rm (i)} Under the assumptions $(H6)_0$ and $(H10)_1$, we have $u\in C^{0,1}_{loc}(D)\cap C(\overline{D})$, and for any $\xi_0\in \mathbb{R}^d$ and $a.e.\  x\in D$,
 \begin{equation}\label{eq:estimate1}
  |u_{(\xi_0)}(x)|\leq N\left(|\xi_0|+\frac{|\psi_{(\xi_0)}(x)|}{\psi^{\frac{3}{4}}(x)}\right)(|g|_{0,1}+\|f(\cdot)\|_{0,1})
 \end{equation}
 where $N=N(K_0, d, d_1, k, D, L, L_0, \mu)$.

 {\rm (ii)} Assume $(H6)_1$, $(H8)$ and $(H10)_2$ hold. Then $u\in C^{1,1}_{loc}(D)\cap C^{0,1}(\overline{D})$, and for any $\xi_0\in\mathbb{R}^d$ and $a.e.\  x\in D$,
 \begin{equation*}
  |u_{(\xi_0)(\xi_0)}(x)|\leq N\left(|\xi_0|^2+\frac{\psi_{(\xi_0)}^2(x)}{\psi^{\frac{7}{4}}(x)}\right)[|g|_{1,1}+\|f(\cdot)\|_{0,1}+[f]_{1,1}(1+|g|_1^2+\|f(\cdot)\|_{0,1}^2)] \label{eq:estimate2}
 \end{equation*}
 where $N=N(K_0, d, d_1, k, D, L, L_0, \mu)$. Furthermore, $u$ is the unique solution in
  $C_{loc}^{1,1}(D)\cap C^{0,1}(\overline{D})$ of the semi-linear system of PDEs
\begin{equation}\label{eq:Par2}
 \begin{cases}
    \cL u(x)+f(x,u(x),\nabla u(x)\sigma(x))=0, \quad &a.e.\  x\in D;\\[3mm]
                 u(x)=g(x), \quad & x\in\partial D.
 \end{cases}
\end{equation}
\end{Theorem}


\section{Moment Estimates of Quasi-derivatives}\label{sec:appl}
In this section, we construct barrier functions in the spirit of~\cite[Lemmas 3.3 and 3.4]{WEI1} to estimate quasi-derivatives, which are used in the gradient and Hessian estimates.

The following estimates on the first exit time can be found in \cite[Lemma 3.1]{WEI1} and \cite[Lemma 2.1]{KRY2}.
\begin{Lemma}\label{le:stopping}
 Let $(H2)$ be satisfied and $\tau(x)$ be the first exit time of $X_t$ from D. Then we have for $x\in D$,
 \begin{eqnarray*}
  &&\ds E\left[\tau(x)\right]\leq\psi(x)\leq|\psi|_0,\quad  E\left[\tau^2(x)\right]\leq2|\psi|_0\psi(x)\leq2|\psi|_0^2,\\
  &&\ds E\left[\tau^p(x)\right]\leq N(|\psi|_0^p), \quad\forall p>2.
  \end{eqnarray*}
\end{Lemma}

Given $\lambda$, $\delta_1\in (0,1)$ with $\delta_1<\lambda^2$, define
\begin{equation*}
 D_{\delta_1}^\lambda:=\{x\in D: \delta_1<\psi(x)<\lambda\},\quad D_{\lambda^2}:=\{x\in D: \psi(x)>\lambda^2\}.
\end{equation*}
For $x\in D_{\delta_1}^\lambda$, define $\tau_1:=\tau_{D_{\delta_1}^\lambda}(x)$ be the first exit time of $X_t$ from $D_{\delta_1}^\lambda$. For $x\in D_{\lambda^2}$, define $\tau_2:=\tau_{D_{\lambda^2}}(x)$ be the first exit time of $X_t$ from $D_{\lambda^2}$.

Krylov \cite[Section 3]{KRY3} introduced the method of dividing the whole domain $D$ into two parts to estimate the moments of quasi-derivatives separately. Since $\psi$ vanishes at the boundary, it is not convenient to construct coefficients of the quasi-derivatives, such as $r,\pi$ and $P$, uniformly in the whole domain. Zhou \cite{WEI1} constructed two families of local super-martingales to estimate moments of quasi-derivatives near the boundary and in the interior of the domain, separately. We still use his notions of quasi-derivatives $\xi_t$ and $\eta_t$.  In our more general  BSDE context (see next section for more details), as higher moment estimates of quasi-derivatives are necessary,  we could not use his original barrier functions $B_1$ and $B_2$ in \cite[Lemmas 3.3 and 3.4]{WEI1}, and instead we consider four new barrier functions in this section. See \cite[Remark 3.5]{WEI1} for the motivation of building barrier functions. Actually our main difficulty lies in the term $E_{x,\xi_0} [u_{(\xi_\tau)}(X_\tau)]$ in our gradient estimate of $u$. So we should try to construct $\xi_t$ such that either $\xi_\tau$ is tangent to $\partial D$ at $X_\tau$ almost surely, or $|u_{(\xi_\tau)}(X_\tau)|$ is bounded by a nonnegative local super-martingale $\{B(X_t,\xi_t),\ t\in[0,\tau]\}$. In our Hessian estimate of $u$, the same difficulty exists around both terms $E_{x,\xi_0} [u_{(\xi_\tau)(\xi_\tau)}(X_\tau)]$ and $E_{x,\eta_0} [u_{(\eta_\tau)}(X_\tau)]$. Similarly, we need to construct $\xi_t$ such that either $\xi_\tau$ is tangent to the boundary, or $|u_{(\xi_\tau)(\xi_\tau)}(X_\tau)|$ is bounded by another nonnegative local super-martingale $\{\tilde B(X_t,\xi_t),\ t\in[0,\tau]\}$. Here as mentioned in \cite[page 5]{WEI1},  $\eta_\tau$ is not necessarily tangent to $\partial D$ at $X_\tau$, for  $\eta_\tau$ can be represented as the sum of the tangential  and the normal components.

Define the three functions
 \begin{eqnarray*}
  \varphi(x)&:=&\lambda^2+\psi(x)-\frac{1}{4\lambda}\psi^2(x), \quad x\in D^\lambda_{\delta_1};\\
 B_1(x,y)&:=&\left(\lambda+\sqrt{\psi(x)}+\psi(x)\right)|y|^4+K_1\varphi^{\frac{7}{2}}(x)\frac{\psi_{(y)}^4(x)}{\psi^3(x)}, \quad  (x, y)\in D^\lambda_{\delta_1}\times \mathbb{R}^d,
 \end{eqnarray*}
 where $K_1\in[1,\infty)$ is a constant depending only on $K_0$; and
 \begin{equation*}
  B_2(y):=\lambda^{\frac{3}{4}}|y|^4, \quad  y\in \mathbb{R}^d.
 \end{equation*}

In this section, for simplicity of exposition, we shall omit the argument $X_t$ in the coefficients $\sigma$, $b$, $\psi$ and their derivatives whenever no confusion is made.


\begin{Lemma}\label{le:quasiestimate1}
 Let $(H3)$ and $(H9)$ be satisfied.
Define $X_t$ by \eqref{eq:SDE2} and the first quasi-derivative $\xi_t$ by (\ref{eq:Qua1}), where for $(\tilde x, y)\in D^\lambda_{\delta_1}\times \mathbb{R}^d$
  \begin{eqnarray}
   &&\ds A(\tilde x):=\sum_{i=1}^{d_1}\psi^2_{(\sigma_i)}(\tilde x),\quad \bar{\rho}(\tilde x,y):=-\frac{1}{A(\tilde x)}\sum_{i=1}^{d_1} \psi_{(\sigma_i)}(\tilde x)(\psi_{(\sigma_i)})_{(y)}(\tilde x),\nonumber\\
   &&\ds r(\tilde x,y):=\bar{\rho}(\tilde x,y)+\frac{\psi_{(y)}(\tilde x)}{\psi(\tilde x)},\quad  \tilde r(\tilde x,y):=\frac{\psi^2_{(y)}(\tilde x)}{\psi^2(\tilde x)},\nonumber\\
   &&\ds \pi_i(\tilde x,y):=\frac{4\psi_{(\sigma_i)}(\tilde x)\psi_{(y)}(\tilde x)}{\varphi(\tilde x)\psi(\tilde x)}, \quad i=1,\cdots,d_1,\nonumber\\
   &&\ds P_{ij}(\tilde x,y):=\frac{1}{A(\tilde x)}[\psi_{(\sigma_j)}(\tilde x)(\psi_{(\sigma_i)})_{(y)}(\tilde x)-\psi_{(\sigma_i)}(\tilde x)(\psi_{(\sigma_j)})_{(y)}(\tilde x)], \quad i,j=1,\cdots,d_1;\nonumber\\
   &&\ds r_t:=r(X_t,\xi_t),\ \pi_t:=(\pi_i(X_t,\xi_t))_{i=1,\cdots,d_1},\ P_t:=(P_{ij}(X_t,\xi_t))_{i,j=1,\cdots,d_1},\nonumber\\
   &&\ds \tilde r_t:=\tilde r(X_t,\xi_t),\ \tilde \pi_t=0,\ \tilde P_t=0, \quad \forall (x, \xi_0, t)\in D_{\delta_1}^\lambda\times \mathbb{R}^d\times[0,\tau_1]. \label{eq:coefficientinB1}
  \end{eqnarray}
Then for sufficiently small $\lambda$,  we have for $(x, \xi_0)\in D_{\delta_1}^\lambda\times \mathbb{R}^d$,

 {\rm (i)}\ the process $(B_1(X_t,\xi_t), B_1^{1\over 2}(X_t,\xi_t)), t\in[0,\tau_1]$ is a local super-martingale;

 {\rm (ii)}\ we have $\ds E_{x,\xi_0}\int_0^{\tau_1}\left(|\xi_t|^4+\frac{\psi_{(\xi_t)}^4(X_t)}{\psi^4(X_t)}\right)dt\leq N(K_0,\lambda)B_1(x,\xi_0)$ and
 $$\ds E_{x,\xi_0}\int_0^{\tau_1}(|r_t|^4+|\pi_t|^4+\|P_t\|^4+|\tilde r_t|^2)dt\leq N(K_0,\lambda)B_1(x,\xi_0);$$

 {\rm (iii)}\ we have $\ds E_{x,\xi_0}\left[\sup\limits_{0\leq t\leq\tau_1}\left(|\xi_t|^4+\frac{\psi_{(\xi_t)}^4(X_t)}{\psi^4(X_t)}\right)\right]\leq N(K_0,\lambda)B_1(x,\xi_0)$
  and
  $$\ds E_{x,\xi_0}\left[\sup\limits_{0\leq t\leq\tau_1}(|r_t|^4+|\pi_t|^4+\|P_t\|^4+|\tilde r_t|^2)\right]\leq N(K_0,\lambda)B_1(x,\xi_0).$$
\end{Lemma}

\begin{proof}

 First, in view of $(H9)$, there exists a constant $\delta'>0$, such that for $x\in\partial D$
 \begin{equation*}
  A:=\sum_{i=1}^{d_1}\psi^2_{(\sigma^i)}=2\langle a\psi_x,\psi_x\rangle=2|\psi_x|^2\langle an,n\rangle\geq 2\delta'.
 \end{equation*}
 Here we use the fact that $\psi_x$ has the same direction of $n$ and $|\psi_x|\geq 1$ near the boundary by continuity.
 Assume that $A\geq 1$ without loss of generality by replacing $\psi$ by $\psi/2\delta'$ if necessary.

 On the one hand, let
 \begin{equation*}
  \bar{\sigma}_t:=\sigma_{(\xi_t)}+r_t\sigma+\sigma P_t,\quad \bar{b}_t:=b_{(\xi_t)}+2r_tb-\sigma\pi_t, \quad t\geq 0.
 \end{equation*}
 Then we have
 \begin{equation*}
  d\xi_t=\bar{b}_tdt+\bar{\sigma}_tdW_t,\quad t\geq 0.
 \end{equation*}
 Using $(H3)$ and $\varphi\geq\lambda^2$, we reduce that
 \begin{equation*}
  |\sigma\pi_t|\leq K\left|\frac{\psi_{(\xi_t)}}{\psi}\right|, \quad t\in[0,\tau_1],
 \end{equation*}
 where $K$ is a constant depending on $K_0$ and $\lambda$. So, we have
 \begin{equation}\label{eq:quasi1}
  \|\bar{\sigma}_t\|\leq K\left(|\xi_t|+\frac{|\psi_{(\xi_t)}|}{\psi}\right),\quad |\bar{b}_t|\leq K\left(|\xi_t|+\frac{|\psi_{(\xi_t)}|}{\psi}\right),\quad t\in[0,\tau_1].
 \end{equation}
 Using It\^{o}'s formula, we have
 \begin{eqnarray*}
  &&\ds d\left[\left(\lambda+\sqrt{\psi}+\psi\right)|\xi_t|^4\right]\\
  &=&\ds \left(\lambda+\sqrt{\psi}+\psi\right)\left(4|\xi_t|^3\bar{b}_tdt+6|\xi_t|^2\|\bar{\sigma}_t\|^2dt+4|\xi_t|^3\bar{\sigma}_tdW_t\right)\\
  &&\ds +|\xi_t|^4\left[\left(1+2\sqrt{\psi}\right)\frac{\cL\psi}{2\sqrt{\psi}}-\frac{A}{8\psi^{\frac{3}{2}}}\right]dt+|\xi_t|^4\left(1+2\sqrt{\psi}\right)\frac{\psi_{(\sigma)}}
  {2\sqrt{\psi}}dW_t\\
  &&\ds +4|\xi_t|^3\bar{\sigma}_t\left(1+2\sqrt{\psi}\right)\frac{\psi_{(\sigma)}}{2\sqrt{\psi}}dt\\
  &=&\ds \Gamma_1(X_t,\xi_t)dt+\Lambda_1(X_t,\xi_t)dW_t, \quad t\in [0,\tau_1].
 \end{eqnarray*}
Set $\Gamma_1(X_t,\xi_t):=I_1+I_2+I_3+I_4$. Since $\lambda^2\leq\varphi\leq 2\lambda$ and $\psi\leq 2\varphi$, applying (\ref{eq:quasi1}) and Young's inequality, we have
 \begin{eqnarray*}
  I_1&=&\ds \lambda\left(4|\xi_t|^3\bar{b}_t+6|\xi_t|^2\|\bar{\sigma}_t\|^2\right)\\
  &\leq&\ds \lambda K\left[|\xi_t|^3\left(|\xi_t|+\frac{|\psi_{(\xi_t)}|}{\psi}\right)+|\xi_t|^2\left(|\xi_t|+\frac{|\psi_{(\xi_t)}|}{\psi}\right)^2\right]\\
  &\leq&\ds \left(3\lambda K\psi^{\frac{3}{2}}+\lambda K\frac{3}{4}+\frac{1}{32}\right)\frac{|\xi_t|^4}{\psi^{\frac{3}{2}}}+\left(\lambda K 2^{\frac{9}{2}}\varphi^2+32\cdot 2^{\frac{3}{2}}K^2\right)\varphi^{\frac{5}{2}}\frac{|\psi_{(\xi_t)}|^4}{\psi^4},
  \end{eqnarray*}

  \begin{eqnarray*}
  I_2&=&\ds\left(\sqrt{\psi}+\psi\right)\left(4|\xi_t|^3\bar{b}_t+6|\xi_t|^2\|\bar{\sigma}_t\|^2\right)\\
  &\leq&\ds 2K\sqrt{\psi}\left(3|\xi_t|^4+|\xi_t|^3\frac{|\psi_{(\xi_t)}|}{\psi}+2|\xi_t|^2\frac{|\psi_{(\xi_t)}|^2}{\psi^2}\right)\\
  &\leq&\ds \left(6K\psi^{\frac{5}{2}}+ \frac{3}{2}K\psi^{\frac{1}{2}}+\frac{1}{32}\right)\frac{|\xi_t|^4}{\psi^{\frac{3}{2}}}+\left(2^{\frac{11}{2}}K \varphi^\frac{5}{2}+ 2^{\frac{17}{2}}K^2\right)\varphi^{\frac{5}{2}}\frac{|\psi_{(\xi_t)}|^4}{\psi^4}.
  \end{eqnarray*}
Since $\cL\psi\leq -1$ and $A\geq 1$, we have
  \begin{equation*}
  I_3=|\xi_t|^4\left[\left(1+2\sqrt{\psi}\right)\frac{\cL\psi}{2\sqrt{\psi}}-\frac{A}{8\psi^{\frac{3}{2}}}\right]\leq -\frac{|\xi_t|^4}{8\psi^{\frac{3}{2}}}.
  \end{equation*}
 Applying (\ref{eq:quasi1}) and Young's inequality, we have
  \begin{eqnarray*}
  I_4&=&\ds 4|\xi_t|^3\bar{\sigma}_t\left(1+\sqrt{\psi}\right)\frac{\psi_{(\sigma)}}{2\sqrt{\psi}}\leq 2K\frac{|\xi_t|^3}{\sqrt{\psi}}\left(|\xi_t|+\frac{|\psi_{(\xi_t)}|}{\psi}\right)\\
  &\leq&\ds 2K\frac{|\xi_t|^4}{\sqrt{\psi}}+2K\frac{|\xi_t|^3|\psi_{(\xi_t)}|}{\psi^{\frac{3}{2}}}\leq 2K\psi\frac{|\xi_t|^4}{\psi^{\frac{3}{2}}}+\frac{3}{128}\frac{|\xi_t|^4}{\psi^{\frac{3}{2}}}+2^{17}K^4\frac{|\psi_{(\xi_t)}|^4}{\psi^{\frac{3}{2}}}.
 \end{eqnarray*}

 On the other hand, by definition of $r$ and $P$, we get
  \begin{equation*}
   \sum\limits_{i} \langle\psi_{xx}\sigma_i, P\sigma_i\rangle = {\rm{tr}}(\sigma\sigma^{*}\psi_{xx}P) = 0,
 \end{equation*}
 and
 \begin{equation*}
  (\psi_{(\sigma_j)})_{(\xi_t)}+r_t\psi_{(\sigma_j)}+\sum_i \psi_{(\sigma_i)}P_t^{ij}=\frac{\psi_{(\xi_t)}}{\psi}\psi_{(\sigma_j)},\quad t\in[0,\tau_1].
 \end{equation*}
 By It\^{o}'s formula, we have
 \begin{equation}\label{eq:estin3_in_B1}
  d\psi_{(\xi_t)}=\sum_i\frac{\psi_{(\xi_t)}}{\psi}\psi_{(\sigma_i)}dW_t^i+\left[(\cL\psi)_{(\xi_t)}+2r_t\cL\psi-\sum_i\psi_{(\sigma_i)}\pi_t^i\right]dt,\quad t\in[0,\tau_1].
 \end{equation}
 Since $\varphi\leq 2\lambda$, $\psi\leq 2\varphi$ and $\cL\psi\leq-1$, after removing the negative terms, we have
 \begin{small}\begin{eqnarray}
   \ds d\left(\varphi^{\frac{7}{2}}\frac{\psi^4_{(\xi_t)}}{\psi^3}\right)&=&\ds \frac{7}{2}\varphi^{\frac{5}{2}}\left(1-\frac{\psi}{2\lambda}\right)\cL\psi\frac{\psi^4_{(\xi_t)}}{\psi^3}dt
   -\frac{7}{2}\varphi^{\frac{5}{2}}\frac{\psi^2_{(\sigma)}}{4\lambda}\frac{\psi^4_{(\xi_t)}}{\psi^3}dt+
   \frac{35}{8}\varphi^{\frac{3}{2}}\left(1-\frac{\psi}{2\lambda}\right)^2\psi^2_{(\sigma)}\frac{\psi^4_{(\xi_t)}}{\psi^3}dt\nonumber\\
&&\ds+4\varphi^{\frac{7}{2}}\frac{\psi^3_{(\xi_t)}}{\psi^3}\left[(\cL\psi)_{(\xi_t)}+2\bar{\rho} \cL\psi\right]dt
+5\varphi^{\frac{7}{2}}\cL\psi\frac{\psi_{(\xi_t)}^4}{\psi^4}dt-16\varphi^{\frac{5}{2}}\frac{\psi^2_{(\sigma)}\psi^4_{(\xi_t)}}{\psi^4}dt\nonumber\\
&&\ds +\frac{7}{2}\varphi^{\frac{5}{2}}\left(1-\frac{\psi}{2\lambda}\right)\psi_{(\sigma)}\frac{\psi^4_{(\xi_t)}}{\psi^3}dW_t
+\varphi^{\frac{7}{2}}\frac{\psi^4_{(\xi_t)}\psi_{(\sigma)}}{\psi^4}dW_t
+\frac{7}{2}\left(1-\frac{\psi}{2\lambda}\right)\varphi^{\frac{5}{2}}\frac{\psi^2_{(\sigma)}\psi^4_{(\xi_t)}}{\psi^4}dt\nonumber\\
&\leq &\ds -\frac{15}{4}\varphi^{\frac{5}{2}}\frac{A\psi_{(\xi_t)}^4}{\psi^4}dt+4 K\varphi^{\frac{7}{2}}\frac{\psi_{(\xi_t)}^3}{\psi^3}|\xi_t|dt+\Lambda_2(X_t,\xi_t)dW_t,\label{eq:estin1_in_B1}
 \end{eqnarray}\end{small}
where
$$\Lambda_2(X_t,\xi_t):=\frac{7}{2}\varphi^{\frac{5}{2}}\left(1-\frac{\psi}{2\lambda}\right)\psi_{(\sigma)}\psi^4_{(\xi_t)}\psi^{-3}
+\varphi^{\frac{7}{2}}\psi^4_{(\xi_t)}\psi_{(\sigma)}\psi^{-4}.$$

Collecting all the above estimates and choosing $K_1$ such that $K_1\geq K$, and letting $\lambda$ be sufficiently small, we get
 \begin{equation}\label{eq:estimate_1}
  dB_1(X_t,\xi_t)\leq \left(-\frac{|\xi_t|^4}{64\psi^{\frac{3}{2}}}-\frac{1}{2}\varphi^{\frac{5}{2}}\frac{\psi^4_{(\xi_t)}}{\psi^4}\right)dt+dm_t\leq dm_t,\quad t\in[0,\tau_1],
 \end{equation}
 where $m_t:=\Lambda_1(X_t,\xi_t)dW_t+K_1\Lambda_2(X_t,\xi_t)dW_t, t\in [0,\tau_1]$ is a local martingale. It follows that the process $\{B_1(X_t,\xi_t),\ t\in [0, \tau_1]\}$ is a local super-martingale.

 Also, since $f(x)=\sqrt{x}$ is concave, the process $\{B_1^{1\over 2}(X_t,\xi_t),\ t\in [0, \tau_1]\}$ is a local super-martingale. Thus Assertion {\rm (i)} is proved.


By definition, we know that for $(x,\xi_0)\in D_{\delta_1}^\lambda\times\mathbb{R}^d$
\begin{equation*}
E_{x,\xi_0}\int_0^{\tau_1}\left(|(r_t, \pi_t, P_t)|^4+|\tilde r_t|^2\right)dt\leq NE_{x,\xi_0}\int_0^{\tau_1}\left(|\xi_t|^4+\frac{\psi^4_{(\xi_t)}}{\psi^4}\right)dt.
\end{equation*}

From (\ref{eq:estimate_1}), there exists a sufficiently small positive $\lambda_0$, such that
\begin{equation*}
 \lambda_0 E_{x,\xi_0}\int_0^{\tau_1}\left(|\xi_t|^4+\frac{\psi^4_{(\xi_t)}}{\psi^4}\right)dt\leq B_1(x,\xi_0)-E_{x,\xi_0}B_1(X_{\tau_1},\xi_{\tau_1})\leq B_1(x,\xi_0),
\end{equation*}
which yields Assertion {\rm (ii)}.


 Using It\^{o}'s formula,  from (\ref{eq:quasi1}), we have
 \begin{eqnarray*}
   d|\xi_t|^4&=& \ds 2|\xi_t|^2\left(2\langle\xi_t,\bar{b}_t\rangle+\|\bar{\sigma}_t\|^2\right)dt
   +|\langle\xi_t,\bar{\sigma}_t\rangle|^2dt+4|\xi_t|^2\langle\xi_t,\bar{\sigma}_td W_t\rangle\\
   &\leq&\ds N\left(|\xi_t|^4+\frac{\psi_{(\xi_t)}^4}{\psi^4}\right)dt+4|\xi_t|^2\langle\xi_t,\bar{\sigma}_td W_t\rangle,\quad t\in[0,\tau_1].
 \end{eqnarray*}
Using Assertion (ii) and the BDG inequality, for $\tau_n=\tau_1\wedge\inf\{t\geq 0:|\xi_t|\geq n\}$, we have
 \begin{eqnarray*}
   E_{\xi_0}\left[\sup\limits_{0\leq t\leq\tau_n} |\xi_t|^4\right]&\leq&\ds |\xi_0|^4+NE_{x,\xi_0}\int_0^{\tau_n}\left(|\xi_t|^4+\frac{\psi_{(\xi_t)}^4}{\psi^4}\right)dt\\
   & &\ds+4E_{x,\xi_0}\left[\sup\limits_{0\leq t\leq\tau_n}\left|\int_0^{\tau_n}|\xi_t|^2\langle\xi_t,\bar{\sigma}_td W_t\rangle\right|\right]\\
   &\le&\ds N B_1(x,\xi_0)+N E_{x,\xi_0}\left[\left(\int_0^{\tau_n}|\xi_t|^6\left(|\xi_t|^2+\frac{\psi_{(\xi_t)}^2}{\psi^2}\right)dt\right)^{\frac{1}{2}}\right].
 \end{eqnarray*}
 Since the last expectation is rewritten and estimated (using twice Cauchy inequality, and then Assertion (ii)) as  follows
  \begin{eqnarray*}
   &\leq&\ds E_{x,\xi_0}\left[\left(\int_0^{\tau_n}N^2\left(|\xi_t|^4+|\xi_t|^2\frac{\psi_{(\xi_t)}^2}{\psi^2}\right)dt\right)^{\frac{1}{2}}\sup_{0\leq t\leq\tau_n}|\xi_t|^2\right]\\
   &\leq&\ds {1\over2}  E_{\xi_0}\left[\sup\limits_{0\leq t\leq\tau_n} |\xi_t|^4\right] +{1\over 2} N^2 E_{x,\xi_0}\int_0^{\tau_n}\left(|\xi_t|^4+|\xi_t|^2\frac{\psi_{(\xi_t)}^2}{\psi^2}\right)dt\\
   &\le & {1\over2}  E_{\xi_0}\left[\sup\limits_{0\leq t\leq\tau_n} |\xi_t|^4\right] +{1\over 2} N^2 E_{x,\xi_0}\int_0^{\tau_n}\left({5\over 4}|\xi_t|^4+\frac{\psi_{(\xi_t)}^4}{\psi^4}\right)dt\\
  &\le & {1\over2}  E_{\xi_0}\left[\sup\limits_{0\leq t\leq\tau_n} |\xi_t|^4\right]+ N^2 N B_1(x,\xi_0),
 \end{eqnarray*}
we conclude the following
 \begin{equation}\label{eq:estin10_in_B1}
  E_{\xi_0}\left[\sup\limits_{0\leq t\leq\tau_n}|\xi_t|^4\right]\leq NB_1(x,\xi_0), \quad \forall (x,\xi_0)\in D_{\delta_1}^\lambda\times\mathbb{R}^d.
 \end{equation}
  In view of \eqref{eq:estin3_in_B1} and using It\^{o}'s formula to the term $\psi_{(\xi_t)}\psi^{-1}$, we find that the relevant local martingale is vanishing. Then using It\^{o}'s formula to the term $\psi_{(\xi_t)}^4\psi^{-4}$, we only need to consider the drift term. Hence, by Young's inequality, we have for $A\geq 1$
 \begin{eqnarray*}
   \ds d\left(\frac{\psi_{(\xi_t)}}{\psi}\right)^4 &=&4\frac{\psi_{(\xi_t)}^3}{\psi^3}\left[\frac{(\cL\psi)_{(\xi_t)}+2\bar{\rho}\cL\psi}{\psi}+\frac{\psi_{(\xi_t)}\cL\psi}{\psi^2}-\frac{4A\psi_{(\xi_t)}}{\varphi\psi^2}\right]dt\\
   &\leq&\displaystyle 4K\frac{|\psi_{(\xi_t)}|^3|\xi_t|}{\psi^4}dt+4\frac{\psi_{(\xi_t)}^4}{\psi^5}\cL\psi dt\\
   &\leq&\displaystyle 3\frac{\psi_{(\xi_t)}^4}{\psi^5}dt
                                   +NK^4\frac{|\xi_t|^4}{\psi}dt+4\frac{\psi_{(\xi_t)}^4}{\psi^5}\cL\psi dt,\quad t\in[0,\tau_1].
 \end{eqnarray*}
 In fact, by \eqref{eq:estimate_1}, we have $E\int_0^{\tau_1} |\xi_t|^4\psi^{-1}dt\leq NB_1(x,\xi_0)$. Hence, by $\cL\psi\leq-1$, we obtain
 \begin{equation}\label{eq:estin11_in_B1}
  E_{x,\xi_0}\left[\sup\limits_{0\leq t\leq\tau_n}\left(\frac{\psi_{(\xi_t)}}{\psi}\right)^4\right]\leq NB_1(x,\xi_0),\quad \forall (x,\xi_0)\in D_{\delta_1}^\lambda\times\mathbb{R}^d.
 \end{equation}

By \eqref{eq:estin10_in_B1} and \eqref{eq:estin11_in_B1}, letting $n\to\infty$, we conclude
 \begin{equation*}
  E_{x,\xi_0}\left[\sup\limits_{0\leq t\leq\tau_1}\left(|\xi_t|^4+\frac{\psi_{(\xi_t)}^4}{\psi^4}\right)\right]\leq NB_1(x,\xi_0),\quad \forall (x,\xi_0)\in D_{\delta_1}^\lambda\times\mathbb{R}^d.
 \end{equation*}
Also, by definition, we have for $(x,\xi_0)\in D_{\delta_1}^\lambda\times\mathbb{R}^d$
\begin{equation}\label{eq:estin2_in_B1}
 E_{x,\xi_0}\left[\sup\limits_{0\leq t\leq\tau_1}\left(|(r_t, \pi_t, P_t)|^4+|\tilde r_t|^2\right)\right]\leq NE_{x,\xi_0}\left[\sup\limits_{0\leq t\leq\tau_1}\left(|\xi_t|^4+\frac{\psi_{(\xi_t)}^4}{\psi^4}\right)\right].
\end{equation}
Thus Assertion {\rm (iii)} is proved.

\end{proof}

\begin{Remark}
 \rm{(i)}  In our Lemma~\ref{le:quasiestimate1},  to make the term $\varphi^{\frac{5}{2}}A\psi_{(\xi_t)}^4\psi^{-4}$ in \eqref{eq:estin1_in_B1} negative, $\pi_i(\tilde x,y)=4\psi_{(\sigma_i)}(\tilde x)\psi_{(y)}(\tilde x)(\varphi(\tilde x)\psi(\tilde x))^{-1}$ is the double of that of \cite[Lemma 3.3]{WEI1}.

 \rm{(ii)} The above estimate \eqref{eq:estin2_in_B1} is new to quasi-derivatives, and will be used to estimate the gradient and Hessian matrix of $u$.
\end{Remark}

\begin{Lemma}\label{le:quasiestimate2}
 Let $(H3)$ and $(H10)_1$ be satisfied.
 Define $X_t$ by \eqref{eq:SDE2} and the first quasi-derivative $\xi_t$ by (\ref{eq:Qua1}), where for $(\tilde x,y)\in D_{\lambda^2}\times\mathbb{R}^d$
 \begin{eqnarray}
   & &\ds r(\tilde x,y):=\langle\rho(\tilde x),y\rangle,\ \pi(\tilde x,y):=\frac{M(\tilde x)}{2}\sigma^{*}(\tilde x)y,\ P(\tilde x,y):= Q(\tilde x,y);\nonumber\\
   & &\ds r_t:=r(X_t,\xi_t),\ \pi_t:=\pi(X_t,\xi_t),\ P_t:=P(X_t,\xi_t),\quad \forall (x, \xi_0,t)\in D_{\lambda^2}\times\mathbb{R}^d\times[0,\tau_2].\nonumber\\ \label{eq:coefficientinB2}
 \end{eqnarray}
 Then for sufficiently small $\lambda$, we have for $(x, \xi_0)\in D_{\lambda^2}\times\mathbb{R}^d$,

 {\rm (i)}\ the process $(e^{4\beta t}B_2(\xi_t), e^{2\beta t}B_2^{1\over 2}(\xi_t)),\ t\in[0,\tau_2]$ is a local super-martingale;

 {\rm (ii)}\ $\ds E_{\xi_0}\int_0^{\tau_2}e^{4\beta t}|\xi_t|^4dt\leq N(K_0,\lambda)B_2(\xi_0)$;

 {\rm (iii)}\ $E_{x,\xi_0}\left[\sup\limits_{0\leq t\leq\tau_2}e^{4\beta t}(|\xi_t|^4+|r_t|^4+|\pi_t|^4+\|P_t\|^4)\right]\leq N(K_0,\lambda)B_2(\xi_0)$.
\end{Lemma}

\begin{proof}

 First of all, replacing $K_0$ by
 \begin{equation*}
  \max\left\{K_0,\sup\limits_{\tilde x\in D_{\lambda^2}}|\rho(\tilde x)|,\mathop{\sup\limits_{\tilde x\in D_{\lambda^2},}}_{y\in\mathbb{R}^d}\frac{\|Q(\tilde x,y)\|}{|y|},\sup\limits_{\tilde x\in D_{\lambda^2}}M(\tilde x)\right\}.
 \end{equation*}
 By It\^{o}'s formula, we have
 \begin{equation*}
  d|\xi_t|^4=\Gamma_2(X_t,\xi_t)dt+\sum_{m=1}^{d_1}\Lambda_2^m(X_t,\xi_t)dW_t^m, \quad t\in[0,\tau_2],
 \end{equation*}
 where
 \begin{equation*}
  \Lambda_2(X_t,\xi_t)=4|\xi_t|^2\langle\xi_t,\sigma_{(\xi_t)}+r_t\sigma+\sigma P_t\rangle.
 \end{equation*}
In view of $(H10)_1$, we have
 \begin{equation*}
  \Gamma_2(X_t,\xi_t)=|\xi_t|^2\left[4\langle\xi_t,b_{(\xi_t)}+2r_tb-\sigma\pi_t\rangle+6\|\sigma_{(\xi_t)}+r_t\sigma+\sigma P_t\|^2\right]\leq(-4\beta-2)|\xi_t|^4.
 \end{equation*}
 So
 \begin{equation}\label{eq:quasi2}
  d\left(e^{4\beta t}|\xi_t|^4\right)\leq -2e^{4\beta t}|\xi_t|^4dt+e^{4\beta t}\sum_{m=1}^{d_1}\Lambda_2^m(X_t,\xi_t)dW_t^m,\quad t\in[0,\tau_2].
 \end{equation}
Therefore, the process $\{e^{4\beta t}B_2(\xi_t), t\in[0,\tau_2]\}$ is a local super-martingale. In view of the concavity of the squared root function,  the process $\{e^{2\beta t}B_2^{1\over 2}(\xi_t), t\in [0,\tau_2]\}$ is a local super-martingale. Assertion (i) is proved.

In view of (\ref{eq:quasi2}), there exists a constant $N>0$ such that
 \begin{equation*}
  E_{\xi_0}\int_0^{\tau_2}e^{4\beta t}|\xi_t|^4dt\leq N\left(B_2(\xi_0)-E_{\xi_0}e^{4\beta\tau_2}B_2(\xi_{\tau_2})\right)\leq NB_2(\xi_0),\quad \forall \xi_0\in\mathbb{R}^d,
 \end{equation*}
 which implies Assertion (ii).

 Using Assertion (ii) and the BDG inequality, for $\tau_n:=\tau_2\wedge\inf\{t\geq 0: |\xi_t|\geq n\}$, we have
 \begin{eqnarray*}
   E_{\xi_0}\left[\sup\limits_{0\leq t\leq\tau_n} e^{4\beta t}|\xi_t|^4\right]&\leq&\displaystyle NB_2(\xi_0)+E_{\xi_0}\left[\sup\limits_{0\leq t\leq\tau_n}\left|\int_0^t e^{4\beta t}\Lambda_2(X_t,\xi_t)dW_t\right|\right]\\
   &\leq&\displaystyle NB_2(\xi_0)+NE_{\xi_0}\left[\left(\int_0^{\tau_n}e^{8\beta t}|\xi_t|^8dt\right)^{\frac{1}{2}}\right].
 \end{eqnarray*}
 Since the last expectation is written and estimated (using Cauchy inequality, and then Assertion (ii)) as follows
  \begin{eqnarray*}
   &\leq&\displaystyle E_{\xi_0}\left[\left(\sup\limits_{0\leq t\leq\tau_n} e^{2\beta t}|\xi_t|^2\right)\left(\int_0^{\tau_n}N^2 e^{4\beta t}|\xi_t|^4dt\right)^{\frac{1}{2}}\right]\\
   &\leq&\displaystyle {1 \over 2} E_{\xi_0}\left[\sup\limits_{0\leq t\leq\tau_n} e^{4\beta t}|\xi_t|^4\right]+\frac{N^2}{2}E_{\xi_0}\int_0^{\tau_n}e^{4\beta t}|\xi_t|^4dt\\
   &\leq& \ds {1 \over 2} E_{\xi_0}\left[\sup\limits_{0\leq t\leq\tau_n} e^{4\beta t}|\xi_t|^4\right]+ N^2NB_2(\xi_0),
  \end{eqnarray*}
  we conclude the following
  \begin{equation*}
    E_{\xi_0}\left[\sup\limits_{0\leq t\leq\tau_n} e^{4\beta t}|\xi_t|^4\right]\leq NB_2(\xi_0),\quad \forall \xi_0\in\mathbb{R}^d.
  \end{equation*}
 Thus, letting $n\to\infty$, we have
 \begin{equation*}
  E_{\xi_0}\left[\sup\limits_{0\leq t\leq\tau_2} e^{4\beta t}|\xi_t|^4\right]\leq NB_2(\xi_0),\quad \forall \xi_0\in\mathbb{R}^d.
 \end{equation*}
By definition, we have for $t\in[0,\tau_2]$, $$|(r_t, \pi_t, P_t)|^4\leq K_0|\xi_t|^4.$$
Hence, we have for $(x,\xi_0)\in D_{\lambda^2}\times\mathbb{R}^d$
\begin{equation*}
 E_{x,\xi_0}\left[\sup\limits_{0\leq t\leq\tau_2} e^{4\beta t}|(r_t, \pi_t, P_t)|^4\right]\leq K_0 E_{\xi_0}\left[\sup\limits_{0\leq t\leq\tau_2} e^{4\beta t}|\xi_t|^4\right]\leq NB_2(\xi_0),
\end{equation*}
which proves Assertion (iii).
\end{proof}

 As shown in both Lemmas \ref{le:quasiestimate1} and \ref{le:quasiestimate2}, both barrier functions $B_1$ and $B_2$ play a crucial role  in the fourth-order moment estimates of first quasi-derivatives $\xi_t$, which are used to estimate the gradient of $u$.

To estimate the Hessian of $u$, the second quasi-derivative is introduced, and the eighth-order  moment estimates of first quasi-derivatives  have to be considered due to standard BSDE estimates for second-order difference.
Define both functions
  \begin{equation*}
    B_3(x,y):=\left(\lambda+\sqrt{\psi(x)}+\psi(x)\right)|y|^8+K_1\varphi^{\frac{15}{2}}(x)\frac{\psi_{(y)}^8(x)}{\psi^7(x)},\quad (x, y)\in D_{\delta_1}^\lambda\times \mathbb{R}^d,
  \end{equation*}
where $K_1\in[1,\infty)$ is a constant depending only on $K_0$;
and
 \begin{equation*}
  B_4(y):=\lambda^{\frac{3}{4}}|y|^8, \quad  y\in\mathbb{R}^d.
 \end{equation*}
With the help of both barrier functions $B_3$ and $B_4$, we can extend both Lemmas~\ref{le:quasiestimate1} and \ref{le:quasiestimate2} to estimate the eighth-order moment of $\xi_t$ and the fourth-order  moment of  $\eta_t$.

\begin{Lemma}\label{le:quasiestimate3}
 Let the assumptions of Lemma \ref{le:quasiestimate1} be satisfied.  Define $X_t$ by \eqref{eq:SDE2}, the first quasi-derivative $\xi_t$ by (\ref{eq:Qua1}), and the second quasi-derivatives $\eta_t$ by (\ref{eq:Qua3}). For $(x,\xi_0,t)\in D_{\delta_1}^\lambda\times\mathbb{R}^d\times[0,\tau_1]$, let the coefficients $r_t,\ \tilde r_t,\ \tilde\pi_t,\ P_t,\ \tilde P_t$ be defined as \eqref{eq:coefficientinB1} in Lemma \ref{le:quasiestimate1} and define $\pi_t:=8[\psi_{(\sigma)}\psi_{(\xi_t)}(\varphi\psi)^{-1}](X_t)$. Then for sufficiently small $\lambda$, we have for $(x, \xi_0)\in D_{\delta_1}^\lambda\times\mathbb{R}^d$ and $\eta_0=0$

 {\rm (i)} the process $(B_3(X_t,\xi_t), B_3^{1\over 2}(X_t,\xi_t)), t\in[0,\tau_1]$ is a local super-martingale;

 {\rm (ii)} $\ds E_{x,\xi_0}\left[\sup\limits_{0\leq t\leq\tau_1}\left(|\xi_t|^8+\frac{\psi_{(\xi_t)}^8(X_t)}{\psi^8(X_t)}\right)+\int_0^{\tau_1}|\xi_t|^8+\frac{\psi_{(\xi_t)}^8(X_t)}{\psi^8(X_t)}dt\right]\leq N(K_0,\lambda)B_3(x,\xi_0)$;

 {\rm (iii)} $E_0\left[\sup\limits_{0\leq t\leq\tau_1}|\eta_t|^4\right]\leq N(K_0,\lambda)B_3(x,\xi_0)$.
\end{Lemma}
\begin{proof}
Repeating the arguments between \eqref{eq:quasi1} and \eqref{eq:estin1_in_B1}, we have the analogue of \eqref{eq:estimate_1}
\begin{equation*}
 dB_3(X_t,\xi_t)\leq \left(-\frac{1}{64}\frac{|\xi_t|^8}{\psi^{\frac{3}{2}}}-\frac{1}{2}\varphi^{\frac{13}{2}}\frac{\psi_{(\xi_t)}^8}{\psi^8}\right)dt+dm_t\leq dm_t, \quad t\in[0,\tau_1],
\end{equation*}
where $\{m_t, t\in[0,\tau_1]\}$ is a local martingale. Then, following the arguments next to formula (\ref{eq:estimate_1}), we can prove Assertion (i).

Analogous to the proof of  Assertions (ii) and (iii) of Lemma \ref{le:quasiestimate1}, we have Assertion (ii).

Now we estimate the moments of the second quasi-derivative $\eta_t$.
By \eqref{eq:Qua3}, we have
\begin{equation*}
 d\eta_t=\left[\sigma_{(\eta_t)}+G(X_t,\xi_t)\right]dW_t+\left[b_{(\eta_t)}+H(X_t,\xi_t)\right]dt, \quad t\in[0,\infty),
\end{equation*}
with
\begin{eqnarray*}
 G(X_t,\xi_t)&=&\ds \sigma_{(\xi_t)(\xi_t)}+2r_t\sigma_{(\xi_t)}+\left(2\sigma_{(\xi_t)}+2r_t\sigma+\sigma P_t\right)P_t+(\tilde r_t-r_t^2),\\
 H(X_t,\xi_t)&=&\ds b_{(\xi_t)(\xi_t)}+4r_tb_{(\xi_t)}+2\tilde r_t b-2\sigma_{(\xi_t)}\pi_t-2r_t\sigma\pi_t-2\sigma P_t\pi_t,\quad t\geq 0.
\end{eqnarray*}
Then, we have the estimates
\begin{equation*}
 ||G||\leq N|\xi_t|\left(|\xi_t|+\frac{|\psi_{(\xi_t)}|}{\psi}\right),\quad |H|\leq N\left(|\xi_t|^2+\frac{\psi_{(\xi_t)}^2}{\psi^2}\right),\quad t\in [0,\tau_1].
\end{equation*}
Hence, It\^{o}'s formula implies
\begin{eqnarray*}
 d\left(e^{2\varphi}|\eta_t|^4\right)&\leq& e^{2\varphi}\left[2\left(1-\frac{\psi}{2\lambda}\right)\cL\psi+\left(1-\frac{\psi}{2\lambda}\right)^2 A+N-\frac{A}{2\lambda}\right]|\eta_t|^4dt\\
 &&+Ne^{2\varphi}\left(|\xi_t|^8+\frac{\psi_{(\xi_t)}^8}{\psi^8}\right)dt+2e^{2\varphi}|\eta_t|^4(1-\frac{\psi}{2\lambda})\psi_{(\sigma)}dW_t\\
 &&+4e^{2\varphi}|\eta_t|^2\langle \eta_t,[\sigma_{(\eta_t)}+G(X_t,\xi_t)]dW_t\rangle, \quad t\in[0,\tau_1].
\end{eqnarray*}
For sufficiently small $\lambda$, there exists a positive constant $\lambda_0$, such that for $(x, \xi_0)\in D_{\delta_1}^\lambda\times\mathbb{R}^d$ and $\eta_0=0$
\begin{equation}\label{eq:estin1_in_B3}
 E_{x,0}\left[e^{2\varphi}|\eta_t|^4\right]+\lambda_0E_{x,0}\int_0^te^{2\varphi}|\eta_s|^4ds\leq NE_{x,\xi_0}\int_0^t e^{2\varphi}\left(|\xi_t|^8+\frac{\psi_{(\xi_t)}^8}{\psi^8}\right)ds.
\end{equation}
Next, using Assertion (ii), formula \eqref{eq:estin1_in_B3}, BDG inequality and Cauchy inequality, for $\tau_n:=\tau_1\wedge\{t\geq 0: e^{\varphi}|\eta_t|^2\geq n\}$ and $\eta_0=0$, we have
\begin{eqnarray}
&&\ds E_{x,0}\left[\sup\limits_{0\leq t\leq\tau_n}\left|\int_0^t 2 e^{2\varphi}\left(1-\frac{\psi}{2\lambda}\right)\psi_{(\sigma)}|\eta_s|^4dW_s\right|\right]\nonumber\\
 &\leq&\ds {1 \over 3} E_{x,0}\left[\sup\limits_{0\leq t\leq\tau_n} e^{2\varphi}|\eta_t|^4\right]+3N^2E_{x,0}\int_0^{\tau_n}e^{2\varphi}|\eta_t|^4dt,\label{eq:estin2_in_B3}
\end{eqnarray}
and
\begin{eqnarray}
 &&\ds E_{x,\xi_0,0}\left[\sup\limits_{0\leq t\leq\tau_n}\left|\int_0^t4e^{2\varphi}|\eta_t|^2\langle \eta_t,[\sigma_{(\eta_t)}+G(X_t,\xi_t)]dW_t\rangle\right|\right]\nonumber\\
 &\leq&\ds E_{x,\xi_0,0}\left\{\sup\limits_{0\leq t\leq\tau_n} (e^{\varphi}|\eta_t|^2)\left[\int_0^{\tau_n}N^2e^{2\varphi}|\eta_t|^4+N^3e^{2\varphi}|\eta_t|^2|\xi_t|^2\left(|\xi_t|^2+\frac{|\psi_{(\xi_t)}|^2}{\psi^2}\right)dt\right]^{1\over2}\right\}\nonumber\\
 &\leq&\ds {1\over 3} E_{x,0}\left[\sup\limits_{0\leq t\leq\tau_n} e^{2\varphi}|\eta_t|^4\right]+N^2E_{x,0}\int_0^{\tau_n}e^{2\varphi}|\eta_t|^4dt\nonumber\\
 &&\ds+2N^4E_{x,\xi_0}\int_0^{\tau_n}e^{2\varphi}\left(|\xi_t|^8+\frac{|\psi_{(\xi_t)}|^8}{\psi^8}\right)dt.\nonumber\\
 &\leq&\ds {1\over 3} E_{x,0}\left[\sup\limits_{0\leq t\leq\tau_n} e^{2\varphi}|\eta_t|^4\right]+ (N^2N+2N^4)E_{x,\xi_0}\int_0^{\tau_n}e^{2\varphi}\left(|\xi_t|^8+\frac{|\psi_{(\xi_t)}|^8}{\psi^8}\right)dt\nonumber\\
 &\leq&\ds {1\over 3} E_{x,0}\left[\sup\limits_{0\leq t\leq\tau_n} e^{2\varphi}|\eta_t|^4\right]+ (N^2N+2N^4)NB_3(x,\xi_0).
 \label{eq:estin3_in_B3}
\end{eqnarray}
Then by Assertion \rm{(ii)}, formulas \eqref{eq:estin2_in_B3} and \eqref{eq:estin3_in_B3}, we have
\begin{eqnarray*}
 E_{x,0}\left[\sup\limits_{0\leq t\leq\tau_n} e^{2\varphi}|\eta_t|^4\right]&\leq&\displaystyle NB_3(x,\xi_0)+E_{x,0}\left[\sup\limits_{0\leq t\leq\tau_n}\left|\int_0^t 2 e^{2\varphi}\left(1-\frac{\psi}{2\lambda}\right)\psi_{(\sigma)}|\eta_s|^4dW_s\right|\right]\\
 &&\displaystyle +E_{x,\xi_0,0}\left[\sup\limits_{0\leq t\leq\tau_n}\left|\int_0^t4e^{2\varphi}|\eta_t|^2\langle \eta_t,[\sigma_{(\eta_t)}+G(X_t,\xi_t)]dW_t\rangle\right|\right]\\
 &\leq &\ds NB_3(x,\xi_0),\quad \forall (x, \xi_0)\in D_{\delta_1}^\lambda\times\mathbb{R}^d.
\end{eqnarray*}
Thus, letting $n\to\infty$, Assertion {\rm (iii)} is proved.
\end{proof}

\begin{Lemma}\label{le:quasiestimate4}
 Let $(H3)$ and $(H10)_2$ be satisfied.
Define $X_t$ by \eqref{eq:SDE2}, the first quasi-derivative $\xi_t$ by \eqref{eq:Qua1}, and the second quasi-derivatives $\eta_t$ by (\ref{eq:Qua3}). For $(\tilde x,y,t)\in D_{\lambda^2}\times\mathbb{R}^d\times[0,\infty)$, define functions $r(\tilde x,y),\ \pi(\tilde x,y),\ P(\tilde x,y)$ by \eqref{eq:coefficientinB2}. For $(x,\xi_0,t)\in D_{\lambda^2}\times\mathbb{R}^d\times[0,\tau_2]$ and $\eta_0=0$, let the  coefficients $r_t$, $\pi_t$ and $P_t$ be defined by \eqref{eq:coefficientinB2} in Lemma \ref{le:quasiestimate2}, and define
\begin{equation}\label{eq:coefficientinB4}
  \tilde r_t:=r(X_t,\eta_t),\quad
  \tilde\pi_t:=\pi(X_t,\eta_t),\quad
  \tilde P_t:=P(X_t,\eta_t).
\end{equation}
Then for sufficiently small $\lambda$, we have for $(x,\xi_0)\in D_{\lambda^2}\times\mathbb{R}^d$ and $\eta_0=0$,

 {\rm (i)} the process $(e^{8\beta t}B_4(\xi_t), e^{4\beta t}B_4^{1\over 2}(\xi_t)),\ t\in[0,\tau_2]$ is a local super-martingale;

 {\rm (ii)} $\ds E_{\xi_0}\left[\sup\limits_{0\leq t\leq\tau_2}e^{8\beta t}|\xi_t|^8\right]+E_{\xi_0}\int_0^{\tau_2}e^{8\beta t}|\xi_t|^8dt\leq N(K_0,\lambda)B_4(\xi_0)$;

 {\rm (iii)} $E_{x,0}\left[\sup\limits_{0\leq t\leq\tau_2}e^{8\beta t}(|\eta_t|^4+|\tilde r_t|^4)\right]\leq N(K_0,\lambda)B_4(\xi_0)$.
\end{Lemma}

\begin{proof}

By It\^{o}'s formula, we have
 \begin{equation*}
  d|\xi_t|^8= \Gamma_3(X_t,\xi_t)dt+\Lambda_3(X_t,\xi_t)dW_t,\quad t\in[0,\tau_2],
 \end{equation*}
 where by $(H10)_2$, we have
 \begin{equation*}
  \Gamma_3(X_t,\xi_t)=8|\xi_t|^6\left\langle\xi_t,\left(b_{(\xi_t)}+2r_tb-\sigma\pi_t\right)\right\rangle+28|\xi_t|^6\|\sigma_{(\xi_t)}+r_t\sigma+\sigma P_t\|^2\leq (-8\beta-4)|\xi_t|^2.
 \end{equation*}
 Then repeating the arguments next to formula (\ref{eq:quasi2}), we can prove Assertion {\rm (i)}.

Similarly, Assertion {\rm (ii)} can be proved in the same way as Lemma \ref{le:quasiestimate2}.

Now, we estimate the moments of the second quasi-derivative $\eta_t$. By \eqref{eq:Qua3}, we have
\begin{eqnarray*}
  d\eta_t&=&\ds \left[\sigma_{(\eta_t)}+\tilde r_t\sigma+\sigma\tilde P_t+\sigma_{(\xi_t)(\xi_t)}+2r_t\sigma_{(\xi_t)}-r_t^2\sigma+2\sigma_{(\xi_t)}P_t+2r_t\sigma P_t+\sigma P_t^2\right]dW_t\\
  &&\ds+\left[b_{(\eta_t)}+2\tilde r_t b-\sigma\tilde\pi_t+b_{(\xi_t)(\xi_t)}
  +4r_tb_{(\xi_t)}-2\sigma_{(\xi_t)}\pi_t-2r_t\sigma\pi_t-2\sigma P_t\pi_t\right]dt\\
  &=&\ds \left[\tilde\sigma_t+G_t \right]dW_t+\left[\tilde b_t+H_t\right]dt,\quad t\geq 0,
\end{eqnarray*}
where
\begin{equation*}
 \tilde\sigma_t=\sigma_{(\eta_t)}+\tilde r_t\sigma+\sigma\tilde P_t, \quad \tilde b_t=b_{(\eta_t)}+2\tilde r_t b-\sigma\tilde\pi_t, t\geq 0,
\end{equation*}
\begin{eqnarray*}
  & &G_t=\sigma_{(\xi_t)(\xi_t)}+2r_t\sigma_{(\xi_t)}-r_t^2\sigma+(2\sigma_{(\xi_t)}+2r_t\sigma+\sigma P_t)P_t,\\
  & &H_t=b_{(\xi_t)(\xi_t)}+4r_t b_{(\xi_t)}-2(\sigma_{(\xi_t)}+r_t\sigma-\sigma P_t)\pi_t,\quad t\geq 0.
\end{eqnarray*}
Assuming that
\begin{equation*}
 \sup\limits_{\tilde x\in D_{\lambda^2}}|\rho(\tilde x)|\leq K_0,\quad \sup\limits_{\tilde x\in D_{\lambda^2}}\|Q(\tilde x,y)\|\leq K_0|y|,\quad \sup\limits_{\tilde x\in D_{\lambda^2}}M(\tilde x)\leq K_0,
\end{equation*}
by $(H3)$, it is not hard to see that
\begin{equation}\label{eq:quasi4}
 \|G_t\|\leq N|\xi_t|^2,\quad |H_t|\leq N|\xi_t|^2,\quad t\in [0,\tau_2].
\end{equation}
So, using $(H10)_2$ and formula (\ref{eq:quasi4}), we have
\begin{eqnarray*}
 d\left(e^{8\beta t}|\eta_t|^4\right)&=&\ds |\eta_t|^2\left[8\beta|\eta_t|^2+4\langle\eta_t,\tilde b_t+H_t\rangle+6\|\tilde\sigma_t+G_t\|^2\right]e^{8\beta t}dt\\
 & &\ds+4|\eta_t|^2e^{8\beta t}\langle\eta_t,(\tilde\sigma_t+G_t)dW_t\rangle\\
 &\leq&\ds \left[(4\beta-1)|\eta_t|^4+N|\xi_t|^8\right]e^{8\beta t}dt+4|\eta_t|^2e^{8\beta t}\langle\eta_t,(\tilde\sigma_t+G_t)dW_t\rangle, \quad t\in[0,\tau_2],
\end{eqnarray*}
where $4\beta-1\leq -1$. Then following the arguments next to formula \eqref{eq:estin1_in_B3}, we prove Assertion {\rm (iii)}.
\end{proof}


We have  more  moment estimates for quasi-derivatives.

\begin{Corollary}
 In addition to $(H3)$ and $(H9)$, let $(H10)_p$ be satisfied for some positive $p$. Define both functions
 \begin{equation*}
   B_{2p-1}(x,y):=\left(\lambda+\sqrt{\psi(x)}+\psi(x)\right)|y|^{4p}+K_1\varphi^{\frac{8p-1}{2}}(x)\frac{\psi_{(y)}^{4p}(x)}{\psi^{4p-1}(x)},\quad (x,y)\in D_{\delta_1}^\lambda\times\mathbb{R}^d,
 \end{equation*}
 where $K_1\in[1,\infty)$ is a constant depending only on $K_0$; and
 \begin{equation*}
  B_{2p}(y):=\lambda^{\frac{3}{4}}|y|^{4p},\quad  y\in\mathbb{R}^d.
 \end{equation*}
 Define $X_t$ by \eqref{eq:SDE2}, the first quasi-derivative $\xi_t$ by \eqref{eq:Qua1}, and the second quasi-derivative $\eta_t$ by \eqref{eq:Qua3}. For $(x, \xi_0,t)\in D_{\delta_1}^\lambda\times\mathbb{R}^d\times[0,\tau_1]$, define
 $$\pi_t:=4p[\psi_{(\sigma)}\psi_{(\xi_t)}(\varphi\psi)^{-1}](X_t)$$ and choose other coefficients $r_t,\ \tilde r_t,\ \tilde\pi_t,\ P_t,\ \tilde P_t$ defined as \eqref{eq:coefficientinB1} in Lemma \ref{le:quasiestimate1}. For $(x, \xi_0,t)\in D_{\lambda^2}\times\mathbb{R}^d\times[0,\tau_2]$, choose the coefficients $r_t,\ \tilde r_t,\ \pi_t,\ \tilde\pi_t,\ P_t,\ \tilde P_t$ defined as \eqref{eq:coefficientinB2} and \eqref{eq:coefficientinB4} in Lemmas \ref{le:quasiestimate2} and \ref{le:quasiestimate4}. Then for sufficiently small $\lambda$, we have

 {\rm (i)} for $(x, \xi_0)\in D_{\delta_1}^\lambda\times\mathbb{R}^d$ and $\eta_0=0$, the process $(B_{2p-1}(X_t,\xi_t), B_{2p-1}^{1\over 2}(X_t,\xi_t)), t\in[0,\tau_1]$ is a local super-martingale;
 \begin{eqnarray}
  &&\ds E_{x,\xi_0,0}\left[\sup\limits_{0\leq t\leq\tau_1}\left(|\xi_t|^{4p}+\frac{\psi_{(\xi_t)}^{4p}(X_t)}{\psi^{4p}(X_t)}+|\eta_t|^{2p}\right)\right]
 +E_{x,\xi_0}\int_0^{\tau_1}\left(|\xi_t|^{4p}+\frac{\psi_{(\xi_t)}^{4p}(X_t)}{\psi^{4p}(X_t)}\right)dt\nonumber\\
  &\leq&\ds N(K_0,\lambda) B_{2p-1}(x,\xi_0);
 \end{eqnarray}

 {\rm (ii)} for $(x, \xi_0)\in D_{\lambda^2}\times\mathbb{R}^d$ and $\eta_0=0$, the process $(e^{4p\beta t}B_{2p}(\xi_t), e^{2p\beta t}B_{2p}^{1\over 2}(\xi_t)), t\in[0,\tau_2]$ is a local super-martingale;
 \begin{equation}
  E_{x,\xi_0,0}\left[\sup\limits_{0\leq t\leq\tau_2}e^{4p\beta t}\left(|\xi_t|^{4p}+|\eta_t|^{2p}\right)\right]+E_{x,\xi_0}\int_0^{\tau_2} e^{4p\beta t}|\xi_t|^{4p}dt\leq N(K_0,\lambda) B_{2p}(\xi_0).
 \end{equation}
\end{Corollary}

\section{Interior Gradient and Hessian Estimates}\label{sec:der}

In this section, we prove Theorem \ref{thm:estimate}. We begin with a standard BSDE estimate and estimate the derivative in two regions near the boundary and in the interior of the domain.

\begin{Remark}\label{re:sec4}
 Without loss of generality, we may assume that $u\in C^1(\overline{D})$ when investigating the gradient estimate of $u$, and $u\in C^2(\overline{D})$ when investigating the Hessian estimate of $u$.
In fact, to find the gradient estimate of the solution such as $\lim\limits_{\delta\to 0}|Y_0^\delta-Y_0|\delta^{-1}\leq N$, we can choose smooth coefficients $(\sigma^\epsilon, f^\epsilon, g^\epsilon)$ such that they converge to $(\sigma, f, g)$. Here we need to notice that $(b, \sigma)$ is smooth enough under assumptions. Assume $(\sigma^\epsilon)^2\geq (\sigma)^2+\epsilon I$, then the nondegenerate elliptic equation with $(b, \sigma^\epsilon, f^\epsilon, g^\epsilon)$ has a classical solution $u^\epsilon(x)=Y_0^\epsilon$. Moreover, due to the solution of a BSDE has continuous dependence on parameter (see \cite[Theorem 2.4]{BRI}), we have
 $$
 \lim\limits_{\delta\to 0}\lim\limits_{\epsilon\to 0} |Y_0^\delta-Y_0^{\epsilon,\delta}|\delta^{-1}=0 \quad \hbox{ \rm and} \quad \lim\limits_{\delta\to 0}\lim\limits_{\epsilon\to 0} |Y_0-Y_0^\epsilon|\delta^{-1}=0.
 $$
  Therefore, our problem is reduced to the estimate of
  $\lim\limits_{\delta\to 0}|Y_0^{\epsilon,\delta}-Y_0^\epsilon|\delta^{-1}$, and thus we can assume $u\in C^1$ to derive the gradient estimate. By the way, we can also assume that $g\in C^1$ and $f\in C^1$ with bounded partial derivatives in $(x,y,z)$ when estimating the gradient of $u$. Noting that $ \cL^\epsilon \psi\leq \cL\psi+\epsilon\psi_{xx}/2\leq -1/2$, we have $ E\tau^\epsilon(x)\leq N$.

  Similarly, we may assume $u, g\in C^2$ and $f\in C^2$ with bounded first and second order partial derivatives in $(x,y,z)$ when investigating Hessian estimates.
\end{Remark}
\subsection{Interior Gradient Estimate}\label{sec:first}
Define the first quasi-derivative $\xi_t$ by \eqref{eq:Qua1}, $X_t$ by \eqref{eq:SDE2} and $X_t^\delta$ by \eqref{eq:FBSDE1}. Let $\tau$ and $\tau^\delta$ be the first exit time from $D$ of $X_t$ and $X_t^\delta$, respectively; and  $\tau_1$ and $\tau_2$ be the first exit time of $X_t$ from $D_{\delta_1}^\lambda$ and  $D_{\lambda^2}$, respectively. Set
$$\gamma^{\delta,n}_T:=\tau\wedge\tau^\delta\wedge k_n\wedge T,$$
where $k_n:=\inf\{t\geq 0;|\xi_t|\geq n\}$ and $T\in[1,\infty)$. Set
$$\gamma_0=\tau^\delta\wedge T, \quad \gamma_{1,T}^{\delta,n}:=\tau_1\wedge\gamma^{\delta,n}_T,  \hbox{ \rm and } \quad \gamma_{2,T}^{\delta,n}=\tau_2\wedge\gamma^{\delta,n}_T.$$
For simplicity of notation, write
$$\gamma:=\gamma^{\delta,n}_T, \quad \gamma_1:=\gamma_{1,T}^{\delta,n}, \quad\hbox{ \rm and } \quad \gamma_2:=\gamma_{2,T}^{\delta,n}$$ whenever no confusion is made.
Also, write for $(x,\xi_0,t)\in D\times\mathbb{R}^d\times[0,\gamma]$,
$$E_{x,\xi_0}[X_t^\delta-X_t]:=E[X_t^{\delta}(x+\delta\xi_0)-X_t(x)]. $$

The main result of this section is stated as follows.

\begin{Theorem}\label{thm:firstorder}
Define $u$ by (\ref{eq:Sol1}). Let $X$ be the unique solution of SDE (\ref{eq:SDE1}), $(Y, Z)$ be the unique solution of BSDE (\ref{eq:BSDE1}) and $(X^\delta, Y^\delta, Z^\delta)$ be the unique solution of FBSDE (\ref{eq:FBSDE1}) with $(\eta_0,\tilde r_t,\tilde\pi_t,\tilde P_t)$ vanishing. Let assumptions $(H1)$-$(H5)$, $(H6)_0$, $(H7)$, $(H9)$ and $(H10)_1$ be satisfied. Then $u\in C^{0,1}_{loc}(D)\cap C(\overline{D})$,
 \begin{eqnarray}
  & &\ds \mathop{\varlimsup\limits_{T\to\infty}}_{n\to\infty}\varlimsup\limits_{\delta\to 0} E_{x,\xi_0}\left[\sup\limits_{t\leq{\gamma_{1,T}^{\delta,n}}\wedge{\gamma_{2,T}^{\delta,n}}}\left|\frac{Y_t^\delta-Y_t}{\delta}\right|^2\right]\nonumber\\
  &\leq&\ds N\left(|\xi_0|^2+\frac{|\psi_{(\xi_0)}(x)|^2}{\psi^{\frac{3}{2}}(x)}\right)(|g|^2_{0,1}+\|f(\cdot)\|_{0,1}^2),\  \forall (x,\xi_0)\in D\times\mathbb{R}^d. \label{eq:estin13}
 \end{eqnarray}
In particular, for any $\xi_0\in\mathbb{R}^d$ and $a.e.\ x\in D,$
 \begin{equation}
  |u_{(\xi_0)}(x)|\leq N\left(|\xi_0|+\frac{|\psi_{(\xi_0)}(x)|}{\psi^{\frac{3}{4}}(x)}\right)(|g|_{0,1}+\|f(\cdot)\|_{0,1})
 \end{equation}
 where $N=N(K_0, d, d_1, k, D, \mu, L, L_0)$.
\end{Theorem}

The proof of Theorem \ref{thm:firstorder} consists of the following sequel of propositions. The following proposition is an analogue to Lemma \ref{le:BSDE}.
\begin{Proposition}\label{pro:estin9}
Under the assumptions of Theorem \ref{thm:firstorder} and the condition (\ref{ass:co1}), there exists a constant $N$ such that for sufficiently small $\delta$

\rm{(i)} $\ds E\left[\sup\limits_{0\leq t\leq\gamma_0}|Y_t^\delta|^2\right]+E\int_0^{\gamma_0}\left(|Y_t^\delta|^2+\|\tilde Z_t^\delta\|^2\right)ds\leq N(|g|_0^2+|f(\cdot,0,0)|^2_0)$;

\rm{(ii)} $\ds E\left[\sup\limits_{0\leq t\leq\gamma_0}|Y_t^\delta|^{2p}\right]+E\left[\left(\int_0^{\gamma_0}\|\tilde Z_t^\delta\|^2ds\right)^p\right]\leq N(|g|_0^{2p}+|f(\cdot,0,0)|_0^{2p}), \quad p\geq 2$.
\end{Proposition}
\begin{proof}
Using It\^{o}'s formula, we have
\begin{eqnarray*}
  d|Y_t^\delta|^2 &\geq&\displaystyle \left\{\left(1+2\delta r_t\right)\left[2\mu-2\left(1+\epsilon\right)L_0^2-\epsilon'\right]-\left(1+\epsilon''\right)|\delta\pi_t|^2 \right\}|Y_t^\delta|^2 dt-\frac{1}{2(1+\epsilon)}\|Z_t^\delta\|^2 dt\\
  &&\displaystyle +\left[1-\frac{1}{1+\epsilon''}\right]\|\tilde Z_t^\delta\|^2 dt-(1+2\delta r_t)\frac{1}{\epsilon'}|f(X_t^\delta,0,0)|^2dt+2Y_t^\delta \tilde Z_t^\delta dW_t.
\end{eqnarray*}
For $x\in D$, if we choose $\delta$ is small enough, we have
$$|\delta r_t|< \frac{1}{2}, \quad |\delta\pi_t|<\epsilon', \quad \hbox{ \rm  and} \quad \|Z_t^\delta\|\leq N\|\tilde Z_t^\delta\|,$$
and from Lemma \ref{le:BSDE4} and the BDG inequality, we have
\begin{eqnarray*}
  &&\ds E\left[\sup\limits_{0\le t\leq\gamma_0}|Y_t^\delta|^2\right]+E\int_0^{\gamma_0}|Y_t^\delta|^2ds+E\int_0^{\gamma_0}\|\tilde Z_t^\delta\|^2ds\\
  &\leq&\ds N\left(E\left[|u(X_{\gamma_0}^\delta)|^2\right]+E\int_0^{\gamma_0}|f(X_s^\delta,0,0)|^2ds\right),
\end{eqnarray*}
which yields Assertion \rm{(i)}. Assertion \rm{(ii)} can be proved in the same way.
\end{proof}

The following lemma is about estimating the directional derivatives of the solutions along first quasi-derivatives on the boundary. When $u$ taking values in $\mathbb{R}^k$, we can still use the methods provided in \cite{WEI2} where concerning the case of dimension one. We refer the reader to arguments (4.3)-(4.14) in \cite{WEI2} for details, but we still prove it for the sake of completeness.

\begin{Lemma}\label{le:estin1_in_sec4}
Let the assumptions of Theorem \ref{thm:firstorder} be satisfied. Assume $u\in C^1(\overline{D})$. Then, there exists a constant $N$ such that

\rm{(i)} for $(x,\xi_0)\in D_{\delta_1}^\lambda\times\mathbb{R}^d$,
\begin{equation}\label{eq:estin1_in_lemma1}
\mathop{\varlimsup\limits_{T\to\infty}}_{n\to\infty}\varlimsup\limits_{\delta\to 0} E_{x,\xi_0}\left[|u_{(\xi_{\gamma_{1,T}^{\delta,n}})}(X_{\gamma_{1,T}^{\delta,n}})|^2e^{2\beta\gamma_{1,T}^{\delta,n}}\right]
\leq NB_1^{1\over 2}(x,\xi_0)+ \mathop{\sup\limits_{y\in\partial D_{\delta_1}^\lambda}}_{0\neq\zeta\in\mathbb{R}^d}\frac{|u_{(\zeta)}(y)|^2}{B_1^\sq(y,\zeta)}B_1^{1\over 2}(x,\xi_0);
\end{equation}

\rm{(ii)} for $(x,\xi_0)\in D_{\lambda^2}\times\mathbb{R}^d$,
\begin{equation}
\mathop{\varlimsup\limits_{T\to\infty}}_{n\to\infty}\varlimsup\limits_{\delta\to 0} E_{x,\xi_0}\left[|u_{(\xi_{\gamma_{2,T}^{\delta,n}})}(X_{\gamma_{2,T}^{\delta,n}})|^2e^{2\beta\gamma_{2,T}^{\delta,n}}\right]
\leq NB_2^{1\over 2}(\xi_0)+ \mathop{\sup\limits_{y\in\partial D_{\lambda^2}}}_{0\neq\zeta\in\mathbb{R}^d}\frac{|u_{(\zeta)}(y)|^2}{B_2^\sq(\zeta)}B_2^{1\over 2}(\xi_0).
\end{equation}
\end{Lemma}

\begin{proof}
 First, for $(x,\xi_0)\in D_{\delta_1}^\lambda\times\mathbb{R}^d$,
\begin{eqnarray*}
 &&\ds E_{x,\xi_0}\left[\left|u_{(\xi_{\gamma_1})}(X_{\gamma_1})\right|^2e^{2\beta\gamma_1}\right]\\
 &\leq&\ds E_{x,\xi_0}\left[\frac{|u_{(\xi_{\gamma_1})}(X_{\gamma_1})|^2}{B_1^\sq(X_{\gamma_1},\xi_{\gamma_1})}
 B_1^\sq(X_{\gamma_1},\xi_{\gamma_1})\right]\\
 &\leq&\ds E_{x,\xi_0}\left[\left(\frac{|u_{(\xi_{\gamma_1})}(X_{\gamma_1})|^2}{B_1^\sq(X_{\gamma_1},\xi_{\gamma_1})}
 -\frac{|u_{(\xi_{\gamma_1})}(X_{\tau_1})|^2}{B_1^\sq(X_{\tau_1},\xi_{\gamma_1})}\right)B_1^\sq(X_{\gamma_1},\xi_{\gamma_1})\right]\\
&& +E_{x,\xi_0}\left[\frac{|u_{(\xi_{\gamma_1})}(X_{\tau_1})|^2}{B_1^\sq(X_{\tau_1},\xi_{\gamma_1})}B_1^\sq(X_{\gamma_1},\xi_{\gamma_1})\right]\\
 &=&\displaystyle J_1(\delta,n,T)+J_2(\delta,n,T).
\end{eqnarray*}

Note that the function $(\tilde x, y)\to |u_{(y)}(\tilde x)|^2B_1^{-\sq}(\tilde x,y)$ is continuous from $D_{\delta_1}^\lambda\times S_1$ to $\mathbb{R}$ with $S_1$ being the unitary ball of $\mathbb{R}^d$. Weierstrass approximation theorem asserts that there is a polynomial $W(\tilde x,y),  (\tilde x,y)\in D_{\delta_1}^\lambda\times S_1$ such that
$$\max\limits_{\tilde x\in D_{\delta_1}^\lambda, y\in S_1}\left|\frac{|u_{(y)}(\tilde x)|^2}{B_1^\sq(\tilde x,y)}-W(\tilde x,y)\right|\leq 1.$$

From Assertion \rm{(i)} of Lemma \ref{le:quasiestimate1}, it follows that
 \begin{eqnarray*}
  J_1(\delta,n,T)&\leq&\ds E_{x,\xi_0}\left[|W(X_{\gamma_1},\xi_{\gamma_1})-W(X_{\tau_1},\xi_{\gamma_1})|B_1^\sq(X_{\gamma_1},\xi_{\gamma_1})\right]
  +2E_{x,\xi_0}B_1^\sq(X_{\gamma_1},\xi_{\gamma_1})\\
&\leq&\ds NE_{x,\xi_0}\left[|X_{\gamma_1}-X_{\tau_1}|B_1^\sq(X_{\gamma_1},\xi_{\gamma_1})\right]+2B_1^{1\over 2}(x,\xi_0)\\
&\leq&\ds \frac{N^2}{4}E_{x,\xi_0}\left[|X_{\gamma_1}-X_{\tau_1}|^2\right]B_1^{1\over 2}(x,\xi_0)+\frac{E_{x,\xi_0}B_1(X_{\gamma_1},\xi_{\gamma_1})}{B_1^{1\over 2}(x,\xi_0)}
+2B_1^{1\over 2}(x,\xi_0)\\
&\leq&\ds \frac{N^2}{4}E_{x,\xi_0}\left[|X_{\gamma_1}-X_{\tau_1}|^2\right]B_1^{1\over 2}(x,\xi_0)+3B_1^{1\over 2}(x,\xi_0)\\
&\leq&\displaystyle \frac{N^2}{4}B_1^{1\over 2}(x,\xi_0)\left(E_{x,\xi_0}|\tau_1-\tau_1\wedge\tau^\delta|+E_{x,\xi_0}|\tau_1-\tau_1\wedge T| +E_{x,\xi_0}|\tau_1-\tau_1\wedge k_n|\right)\\
    &&\displaystyle+3B_1^{1\over 2}(x,\xi_0).
 \end{eqnarray*}

Due to \cite[Theorem3.3]{WEI2}, we have
\begin{equation*}
 \varlimsup\limits_{\delta\to 0}E_{x,\xi_0}(\tau_1-\tau_1\wedge\tau^\delta)=0.
\end{equation*}
Thus,
\begin{equation*}
 \mathop{\varlimsup\limits_{T\to\infty}}_{n\to\infty}\varlimsup\limits_{\delta\to 0} J_1(\delta,n,T)\leq 3B_1^{1\over 2}(x,\xi_0).
\end{equation*}
From Assertion \rm{(i)} in Lemma \ref{le:quasiestimate1}, we have
\begin{equation*}
 J_2(\delta,n,T)\leq \mathop{\sup\limits_{y\in\partial D_{\delta_1}^\lambda}}_{0\neq\zeta\in\mathbb{R}^d}\frac{|u_{(\zeta)}(y)|^2}{B_1^\sq(y,\zeta)}B_1^{1\over 2}(x,\xi_0)\ .
\end{equation*}
So, we have Assertion \rm{(i)}.

Second, Assertion \rm{(ii)} can be proved for $(x,\xi_0)\in D_{\lambda^2}\times\mathbb{R}^d$ in the same way.
\end{proof}

Combined with Lemmas \ref{le:BSDE4} and \ref{le:estin1_in_sec4}, we have the following immediate consequence.

\begin{Proposition}
Let the assumptions of Theorem \ref{thm:firstorder} be satisfied. Assume $u\in C^1(\overline{D})$. Then there exists a constant $N$ such that

\rm{(i)} for $(x,\xi_0)\in D_{\delta_1}^\lambda\times\mathbb{R}^d$,
  \begin{equation}\label{eq:estin1_in_pro1}
  \mathop{\varlimsup\limits_{T\to\infty}}_{n\to\infty}\varlimsup\limits_{\delta\to 0} \frac{1}{\delta^2} E_{x,\xi_0}\left[|Y_{\gamma_{1,T}^{\delta,n}}^\delta-Y_{\gamma_{1,T}^{\delta,n}}|^2e^{2\beta\gamma_{1,T}^{\delta,n}}\right]
 \leq NB_1^{1\over 2}(x,\xi_0)+ \mathop{\sup\limits_{y\in\partial D_{\delta_1}^\lambda}}_{0\neq\zeta\in\mathbb{R}^d}\frac{|u_{(\zeta)}(y)|^2}{B_1^\sq(y,\zeta)}B_1^{1\over 2}(x,\xi_0)\ ;
 \end{equation}

\rm{(ii)} for $(x,\xi_0)\in D_{\lambda^2}\times\mathbb{R}^d$,
\begin{equation}\label{eq:estin2_in_pro1}
 \mathop{\varlimsup\limits_{T\to\infty}}_{n\to\infty}\varlimsup\limits_{\delta\to 0} \frac{1}{\delta^2} E_{x,\xi_0}\left[|Y_{\gamma_{2,T}^{\delta,n}}^\delta-Y_{\gamma_{2,T}^{\delta,n}}|^2e^{2\beta\gamma_{2,T}^{\delta,n}}\right]
 \leq NB_2^{1\over 2}(\xi_0)+ \mathop{\sup\limits_{y\in\partial D_{\lambda^2}}}_{0\neq\zeta\in\mathbb{R}^d}\frac{|u_{(\zeta)}(y)|^2}{B_2^\sq(\zeta)}B_2^{1\over 2}(\xi_0).
\end{equation}
\end{Proposition}

\begin{proof}
First, for $(x, \xi_0)\in D_{\delta_1}^\lambda\times\mathbb{R}^d$, we have
 \begin{eqnarray*}
    &&\ds E_{x,\xi_0}\left[|Y_{\gamma_1}^\delta-Y_{\gamma_1}|^2e^{2\beta\gamma_1}\right]\\
  &=&\ds E_{x,\xi_0}\left[|u(X_{\gamma_1}^\delta)-u(X_{\gamma_1})|^2e^{2\beta\gamma_1}\right]\\
  &\leq&\ds E_{x,\xi_0} \left[\delta^2\left|\frac{u(X_{\gamma_1}^\delta)-u(X_{\gamma_1})}{\delta}-u_{(\xi_{\gamma_1})}
  (X_{\gamma_1})\right|^2e^{2\beta\gamma_1}\right]+E_{x,\xi_0}\left[ \delta^2\left|u_{(\xi_{\gamma_1})}(X_{\gamma_1})\right|^2e^{2\beta\gamma_1}\right].
 \end{eqnarray*}
And, due to Mean Value Theorem, we get
\begin{small}\begin{eqnarray*}
 &&\ds E_{x,\xi_0}\left[\left|\frac{u(X_{\gamma_1}^\delta)-u(X_{\gamma_1})}{\delta}-u_{(\xi_{\gamma_1})}(X_{\gamma_1})\right|^2e^{2\beta\gamma_1}\right]\\
&\leq&\ds E_{x,\xi_0}\left[\left|{\widehat u}_x\left(\frac{X_{\gamma_1}^\delta-X_{\gamma_1}}{\delta}-\xi_{\gamma_1}\right)\right|^2e^{2\beta\gamma_1}\right]
+E_{x,\xi_0}\left[\left|{\widehat u}_x-u_x(X_{\gamma_1})\right|^2|\xi_{\gamma_1}|^2e^{2\beta\gamma_1}\right],
\end{eqnarray*}\end{small}
 where
 $${\widehat u}_x:=\int_0^1 u_x(X_{\gamma_1}+\lambda(X_{\gamma_1}^\delta-X_{\gamma_1}))d\lambda.$$
 As $u_x$ is continuous and $\ds\lim\limits_{\delta\to 0}E_{x,\xi_0}[|X_{\gamma_1}^\delta-X_{\gamma_1}|^2]=0$, by the dominated convergence theorem and \eqref{eq:estin11} with $p=2$, we have
\begin{equation*}
 \lim\limits_{\delta\to 0}E_{x,\xi_0}\left[\left|\frac{u(X_{\gamma_1}^\delta)-u(X_{\gamma_1})}{\delta}-u_{(\xi_{\gamma_1})}(X_{\gamma_1})\right|^2e^{2\beta\gamma_1}\right]=0.
\end{equation*}
Then, using formula (\ref{eq:estin1_in_lemma1}), we prove Assertion \rm{(i)}.

Second, repeating the arguments for $(x,\xi_0)\in D_{\lambda^2}\times\mathbb{R}^d$, Assertion \rm{(ii)} is proved.
\end{proof}

Note that our definition of $B_1$ and $B_2$ preceding Lemmas \ref{le:quasiestimate1} is different from those of \cite{WEI1}.  However, we can easily check out the same relations as \cite{WEI1} between both barrier functions $B_1$ and $B_2$.

\begin{Lemma}\label{le:barriers}
For sufficiently small $\lambda$, we have
 \begin{equation}\label{eq:estin24}
 \begin{cases}
   B_1(x,y)\geq 4B_2(y),\quad \forall (x,y)\in\{x: \psi(x)=\lambda\}\times\mathbb{R}^d;\\[8pt]
   4B_1(x,y)\leq B_2(y), \quad \forall (x,y)\in\{x: \psi(x)=\lambda^2\}\times\mathbb{R}^d.
  \end{cases}
 \end{equation}
\end{Lemma}

The following lemma helps to estimate the unknown terms on the right hand sides of formulas (\ref{eq:estin1_in_pro1}) and (\ref{eq:estin2_in_pro1}).


\begin{Lemma}\label{le:estin2_in_sec4}
 Assume that $g\in C^1(\overline{D})$, $f\in C^1(\overline{D}\times\mathbb{R}^k\times\mathbb{R}^{k\times d_1})$ has bounded partial derivatives in $x$, and $u\in C^1(\overline{D})$. Moreover, assume that
 \begin{equation}\label{eq:estin1_in_lemma2}
 \frac{|u_{(\xi_0)}(x)|^2}{B_1^{1\over 2}(x,\xi_0)}\leq \mathop{\sup\limits_{y\in\partial D_{\delta_1}^\lambda}}_{0\neq\zeta\in\mathbb{R}^d}\frac{|u_{(\zeta)}(y)|^2}{B_1^\sq(y,\zeta)}+N(|g|^2_0+\|f(\cdot)\|_{0,1}^2), \quad \forall (x,\xi_0)\in D_{\delta_1}^\lambda\times\mathbb{R}^d\setminus\{0\},
 \end{equation}
 and
 \begin{equation}\label{eq:estin2_in_lemma2}
 \frac{|u_{(\xi_0)}(x)|^2}{B_2^{1\over 2}(\xi_0)}\leq \mathop{\sup\limits_{y\in\partial D_{\lambda^2}}}_{0\neq\zeta\in\mathbb{R}^d}\frac{|u_{(\zeta)}(y)|^2}{B_2^\sq(\zeta)}+N(|g|^2_0+\|f(\cdot)\|_{0,1}^2), \quad \forall (x,\xi_0)\in D_{\lambda^2}\times\mathbb{R}^d\setminus\{0\},
\end{equation}
where $N=N(K_0, d, d_1, k, D, \mu, L, L_0)$, and $B_1$ and $B_2$ are defined preceding Lemmas \ref{le:quasiestimate1}. Then we have
\begin{equation}
  |u_{(\xi_0)}(x)|^2\leq  N\left(|\xi_0|^2+\frac{|\psi_{(\xi_0)}(x)|^2}{\psi^{\frac{3}{2}}(x)}\right)(|g|^2_1+\|f(\cdot)\|_{0,1}^2),\quad \forall (x,\xi_0)\in D\times\mathbb{R}^d.
\end{equation}
In particular, $N$ does not depend on $|u_x|_0$.
\end{Lemma}

\begin{proof} The proof is same to the arguments next to \cite[relation (3.24)]{WEI1}.

 Set $N_1:=N(|g|^2_0+\|f(\cdot)\|_{0,1}^2)$. In view of formulas (\ref{eq:estin24}), (\ref{eq:estin1_in_lemma2}) and (\ref{eq:estin2_in_lemma2}), we have for $(x,\xi_0)\in \{x\in D: \psi(x)=\lambda\}\times\mathbb{R}^d\setminus\{0\}$
\begin{eqnarray*}
 \ds\frac{|u_{(\xi_0)}(x)|^2}{B_1^{1\over 2}(x,\xi_0)}&\leq&\displaystyle \frac{|u_{(\xi_0)}(x)|^2}{2B_2^{1\over 2}(\xi_0)}
 \leq\displaystyle \mathop{\sup\limits_{\psi(x)=\lambda^2}}_{ 0\neq\xi_0\in\mathbb{R}^d}\frac{|u_{(\xi_0)}(x)|^2}{2B_2^{1\over 2}(\xi_0)}+{1\over 2}N_1\\
 &\leq&\displaystyle \mathop{\sup\limits_{\psi(x)=\lambda^2}}_{ 0\neq\xi_0\in\mathbb{R}^d}\frac{|u_{(\xi_0)}(x)|^2}{4B_1^{1\over 2}(x,\xi_0)}+\frac{N_1}{2}\\
 &\leq&\displaystyle \mathop{\sup\limits_{\psi(x)=\lambda}}_{0\neq\xi_0\in\mathbb{R}^d}\frac{|u_{(\xi_0)}(x)|^2}{4B_1^{1\over 2}(x,\xi_0)}
 +\mathop{\sup\limits_{\psi(x)=\delta_1}}_{0\neq\xi_0\in\mathbb{R}^d}\frac{|u_{(\xi_0)}(x)|^2}{4B_1^{1\over 2}(x,\xi_0)}+\frac{3N_1}{4},
\end{eqnarray*}
which implies that
\begin{equation}\label{eq:estin3}
 \mathop{\sup\limits_{\psi(x)=\lambda}}_{ 0\neq\xi_0\in\mathbb{R}^d} \frac{|u_{(\xi_0)}(x)|^2}{B_1^{1\over 2}(x,\xi_0)}\leq \mathop{\sup\limits_{\psi(x)=\delta_1}}_{0\neq\xi_0\in\mathbb{R}^d}\frac{|u_{(\xi_0)}(x)|^2}{3B_1^{1\over 2}(x,\xi_0)}+N_1.
\end{equation}
Meanwhile, in view of formulas (\ref{eq:estin24}), (\ref{eq:estin1_in_lemma2}) and (\ref{eq:estin3}), we have for $(x,\xi_0)\in\{x: \psi(x)=\lambda^2\}\times\mathbb{R}^d\setminus\{0\}$,
\begin{eqnarray*}
 \ds \frac{|u_{(\xi_0)}(x)|^2}{B_2^{1\over 2}(\xi_0)}&\leq&\displaystyle \frac{|u_{(\xi_0)}(x)|^2}{2B_1^{1\over 2}(x,\xi_0)}
 \leq \mathop{\sup\limits_{\psi(x)=\lambda}}_{ 0\neq\xi_0\in\mathbb{R}^d}\frac{|u_{(\xi_0)}(x)|^2}{2B_1^{1\over 2}(x,\xi_0)}
 +\mathop{\sup\limits_{\psi(x)=\delta_1}}_{0\neq\xi_0\in\mathbb{R}^d}\frac{|u_{(\xi_0)}(x)|^2}{2B_1^{1\over 2}(x,\xi_0)}+{1\over2}N_1\\
 &\leq&\ds \mathop{\sup\limits_{\psi(x)=\delta_1}}_{0\neq\xi_0\in\mathbb{R}^d}\frac{2\, |u_{(\xi_0)}(x)|^2}{3B_1^{1\over 2}(x,\xi_0)}+N_1.
\end{eqnarray*}
Therefore, taking the supremum, we have
\begin{equation}\label{eq:estin4}
  \mathop{\sup\limits_{\psi(x)=\lambda^2}}_{0\neq\xi_0\in\mathbb{R}^d} \frac{|u_{(\xi_0)}(x)|^2}{B_2^{1\over 2}(\xi_0)}\leq \mathop{\sup\limits_{\psi(x)=\delta_1}}_{0\neq\xi_0\in\mathbb{R}^d}\frac{2\, |u_{(\xi_0)}(x)|^2}{3B_1^{1\over 2}(x,\xi_0)}+N_1.
\end{equation}
Combining (\ref{eq:estin1_in_lemma2}) and (\ref{eq:estin3}), we get, for $(x,\xi_0)\in D_{\delta_1}^\lambda\times\mathbb{R}^d\setminus\{0\}$,
\begin{equation}\label{eq:estin5}
  \frac{|u_{(\xi_0)}(x)|^2}{B_1^{1\over 2}(x,\xi_0)}\leq  \mathop{\sup\limits_{\psi(y)=\delta_1}}_{0\neq\xi_0\in\mathbb{R}^d}\frac{4\, |u_{(\xi_0)}(y)|^2}{3B_1^\sq(y,\xi_0)}+2N_1.
\end{equation}
Combining (\ref{eq:estin2_in_lemma2}) and (\ref{eq:estin4}), we get, for $(x,\xi_0)\in D_\lambda^2\times\mathbb{R}^d\setminus\{0\}$,
\begin{equation}\label{eq:estin6}
  \frac{|u_{(\xi_0)}(x)|^2}{B_2^{1\over 2}(\xi_0)}\leq  \mathop{\sup\limits_{\psi(y)=\delta_1}}_{0\neq\xi_0\in\mathbb{R}^d}\frac{2\, |u_{(\xi_0)}(y)|^2}{3B_1^\sq(y,\xi_0)}+2N_1.
\end{equation}
Thus, it remains to estimate
\begin{equation*}
 \varlimsup\limits_{\delta_1\to 0}\mathop{\sup\limits_{\psi(y)=\delta_1}}_{0\neq\xi_0\in\mathbb{R}^d}\frac{|u_{(\xi_0)}(y)|^2}{B_1^\sq(y,\xi_0)}.
\end{equation*}
Notice that for each $\delta_1$, there exist $y(\delta_1)\in\{x:\ \psi(x)=\delta_1\}$ and $\xi_0(\delta_1)\in\{\xi_0:|\xi_0|=1\}$, such that
\begin{equation*}
 \mathop{\sup\limits_{\psi(y)=\delta_1}}_{0\neq\xi_0\in\mathbb{R}^d}\frac{|u_{(\xi_0)}(y)|^2}{B_1^\sq(y,\xi_0)}=\frac{|u_{(\xi_0(\delta_1))}(y(\delta_1))|^2}{B_1^\sq(y(\delta_1),\xi_0(\delta_1))}.
\end{equation*}
A subsequence of $(y(\delta_1),\xi_0(\delta_1))$ converges to some $(z,\zeta)$, as $\delta_1\to 0$, such that $z\in\partial D$ and $|\zeta|=1$.

If $\psi_{(\zeta)}(z)\neq 0$, we have
\begin{equation}\label{eq:estin7}
 \varlimsup\limits_{\delta_1\to 0}\mathop{\sup\limits_{\psi(y)=\delta_1}}_{0\neq\xi_0\in\mathbb{R}^d}\frac{|u_{(\xi_0)}(y)|^2}{B_1^\sq(y,\xi_0)}= \varlimsup\limits_{\delta_1\to 0}\frac{|u_{(\xi_0(\delta_1))}(y(\delta_1))|^2}{B_1^\sq(y(\delta_1),\xi_0(\delta_1))}=0.
\end{equation}
If $\psi_{(\zeta)}(z)= 0$, we have
\begin{equation}\label{eq:estin8}
 \varlimsup\limits_{\delta_1\to 0}\mathop{\sup\limits_{\psi(y)=\delta_1}}_{0\neq\xi_0\in\mathbb{R}^d}\frac{|u_{(\xi_0)}(y)|^2}{B_1^\sq(y,\xi_0)}= \varlimsup\limits_{\delta_1\to 0}\frac{|u_{(\xi_0(\delta_1))}(y(\delta_1))|^2}{B_1^\sq(y(\delta_1),\xi_0(\delta_1))}=\frac{|g_{(\zeta)}(z)|^2}{\sqrt{\lambda}}\leq N|g|^2_1.
\end{equation}
From (\ref{eq:estin5}),(\ref{eq:estin6}),(\ref{eq:estin7}) and (\ref{eq:estin8}), we have
\begin{equation*}
\begin{cases}
  \ds\frac{|u_{(\xi_0)}(x)|^2}{B_1^{1\over 2}(x,\xi_0)}\leq  N(|g|_1^2+\|f(\cdot)\|_{0,1}^2), \quad&  \forall(x,\xi_0)\in D^\lambda\times\mathbb{R}^d\setminus\{0\};\\ \ds\frac{|u_{(\xi_0)}(x)|^2}{B_2^{1\over 2}(\xi_0)}\leq N(|g|_1^2+\|f(\cdot)\|_{0,1}^2), \quad&  \forall (x,\xi_0)\in D_{\lambda^2}\times\mathbb{R}^d\setminus\{0\}.
\end{cases}
\end{equation*}
Notice that $D^\lambda\cup D_{\lambda^2}=D$, and
\begin{equation*}
 \begin{cases}
 B_1^{1\over 2}(x,\xi_0)\leq \ds N\left(|\xi_0|^2+\frac{|\psi_{(\xi_0)}(x)|^2}{\psi^{\frac{3}{2}}(x)}\right), \quad & \forall(x,\xi_0)\in D^\lambda\times\mathbb{R}^d,\\
 B_2^{1\over 2}(\xi_0)\leq N|\xi_0|^2, \quad & \forall\xi_0\in \mathbb{R}^d.
 \end{cases}
\end{equation*}
We then have the desired last assertion.
\end{proof}

Now, we are ready to complete the proof of Theorem \ref{thm:firstorder}.
\begin{proof}[\textbf{Proof of Theorem \ref{thm:firstorder}}]

\textbf{Step 1.} By It\^{o}'s formula, we have
\begin{eqnarray}
 &&\ds d\left(|Y_t^\delta-Y_t|^2e^{2\beta t}\right)\nonumber\\
 &=&\displaystyle e^{2\beta t}2\beta|Y_t^\delta-Y_t|^2dt\nonumber\\
 &&\displaystyle +e^{2\beta t}2\left\langle(Y_t^\delta-Y_t),\left[-f(X_t^\delta,Y_t^\delta,Z_t^\delta)
 (1+2\delta r_t)-\tilde Z_t^\delta\delta\pi_t+f(X_t,Y_t,Z_t)\right]\right\rangle dt\nonumber\\
 &&\displaystyle +e^{2\beta t}\left\|\tilde Z_t^\delta-Z_t\right\|^2dt+e^{2\beta t} 2\llan\left(Y_t^\delta-Y_t\right),\left(\tilde Z_t^\delta-Z_t\right)dW_t\rran, \quad t\in [0,\gamma].\label{eq:estin25}
\end{eqnarray}

Then we calculate by parts:

First, in view of assumptions $(H4)$ and $(H5)$, and Cauchy inequality, we have for $\epsilon\in (0,1)$
 \begin{eqnarray*}
   &&\displaystyle e^{2\beta t}2\llan\left(Y_t^\delta-Y_t\right),\left[-f(X_t^\delta,Y_t^\delta,Z_t^\delta)\left(1+2\delta r_t\right)\right]\rran\\
   &\geq&\displaystyle e^{2\beta t}\left(1+2\delta r_t\right)\left[2\mu|Y_t^\delta-Y_t|^2-2L_0|Y_t^\delta-Y_t|\|Z_t^\delta-Z_t\|-2\llan(Y_t^\delta-Y_t),f(X_t^\delta,Y_t,Z_t)\rran\right]\\
   &\geq&\displaystyle e^{2\beta t}(1+2\delta r_t)|Y_t^\delta-Y_t|^2\left(2\mu-2(1+\epsilon)L_0^2\right)-\frac{1}{2(1+\epsilon)}e^{2\beta t}(1+2\delta r_t)\|Z_t^\delta-Z_t\|^2\\
   &&\displaystyle -e^{2\beta t}2(1+2\delta r_t)\llan(Y_t^\delta-Y_t),f(X_t^\delta,Y_t,Z_t)\rran.
 \end{eqnarray*}
By Cauchy inequality, we have for $\epsilon_1\in (0,1)$
 \begin{displaymath}
   -e^{2\beta t}2\llan(Y_t^\delta-Y_t),\tilde Z_t^\delta\delta\pi_t\rran \geq -e^{2\beta t}\epsilon_1|Y_t^\delta-Y_t|^2-\frac{1}{\epsilon_1}e^{2\beta t}|\delta\pi_t|^2\|\tilde Z_t^\delta\|^2.
 \end{displaymath}
Let $f$ be continuously differentiable and have bounded partial derivatives with respect to $(x,y,z)$ (see Remark \ref{re:sec4}). By Cauchy inequality, we have for $\epsilon_2, \epsilon_3\in (0,1)$
 \begin{eqnarray*}
   &&\displaystyle -e^{2\beta t}2(1+2\delta r_t)\llan(Y_t^\delta-Y_t),f(X_t^\delta,Y_t,Z_t)\rran+e^{2\beta t}2\llan(Y_t^\delta-Y_t),f(X_t,Y_t,Z_t)\rran\\
   &=&\displaystyle -e^{2\beta t}4\delta r_t\llan(Y_t^\delta-Y_t),f(X_t^\delta,Y_t,Z_t)\rran\\
   && -e^{2\beta t}2\llan(Y_t^\delta-Y_t),\left[f(X_t^\delta,Y_t,Z_t)-f(X_t,Y_t,Z_t)\right]\rran\\
   &\geq&\displaystyle -e^{2\beta t}4|\delta r_t|\cdot|Y_t^\delta-Y_t|\left[|f(X_t^\delta,0,0)|+L|Y_t|+L_0\|Z_t\|\right]\\
   && -e^{2\beta t}2\|f(\cdot)\|_{0,1}\cdot|Y_t^\delta-Y_t|\cdot|X_t^\delta-X_t|\\
   &\geq&\ds -e^{2\beta t}(\epsilon_2+\epsilon_3)|Y_t^\delta-Y_t|^2-\frac{16}{\epsilon_2}e^{2\beta t}\delta^2r_t^2\left(L^2|Y_t|^2+L_0^2\|Z_t\|^2\right)-\frac{8}{\epsilon_2}e^{2\beta t}\delta^2r_t^2|f(\cdot,0,0)|^2_0\\
   &&\ds -\frac{\|f(\cdot)\|_{0,1}^2}{\epsilon_3}e^{2\beta t}|X_t^\delta-X_t|^2.
  \end{eqnarray*}

 Second, by Cauchy inequality, we have
 \begin{small}\begin{eqnarray*}
   &&\ds -\frac{1}{2(1+\epsilon)}e^{2\beta t}(1+2\delta r_t)\left\|Z_t^\delta-Z_t\right\|^2\\
   &=&\ds -\frac{1}{2(1+\epsilon)}e^{2\beta t}(1+2\delta r_t)\left\|\tilde Z_t^\delta-Z_t-\tilde Z_t^\delta+Z_t^\delta\right\|^2\\
   &\geq&\ds -\frac{(1+2\delta r_t)}{1+\epsilon}e^{2\beta t}\left\|\tilde Z_t^\delta-Z_t\right\|^2-\frac{(1+2\delta r_t)}{1+\epsilon}e^{2\beta t}\left\|Z_t^\delta\right\|^2\left|\sqrt{1+2\delta r_t}e^{\delta P_t}-1\right|^2,
 \end{eqnarray*}\end{small}
 where
 \begin{eqnarray*}
   &&\ds (1+2\delta r_t)\| Z_t^\delta\|^2\left|\sqrt{1+2\delta r_t}e^{\delta P_t}-1\right|^2\leq 2\|\tilde Z_t^\delta\|^2\left[(\sqrt{1+2\delta r_t}-1)^2+(e^{\delta P_t}-1)^2\right]\\
   &\leq& \ds N\|\tilde Z_t^\delta\|^2\left(\delta^2 r_t^2+\delta^2\|P_t\|^2\right)+o(\delta^3).
 \end{eqnarray*}

Third, condition (\ref{ass:co1}) is satisfied for $(x, t)\in D_{\delta_1}^\lambda\times[0,\gamma_1]$. For fixed  $(x, \xi_0)\in D_{\delta_1}^\lambda \times \mathbb{R}^d$, we can always choose a small $\delta$ such that $x+\delta\xi_0\in D_{\delta_1}^\lambda$. In view of Lemma \ref{le:BSDE},  Assertion \rm{(iii)} in Lemma \ref{le:quasiestimate1}, Proposition \ref{pro:estin9} and H\"{o}lder inequality, we have
\begin{small}\begin{eqnarray*}
  &&\ds E_{x,\xi_0}\int_0^{\gamma_1} |(r_s,\pi_s,P_s)|^2|(Y_s,Z_s,Y_s^\delta,\tilde Z_s^\delta)|^2 e^{2\beta s}ds\\
  &\leq&\ds E_{x,\xi_0}\left[\sup\limits_{0\le t\leq\gamma_1} |(r_t,\pi_t,P_t)|^2\int_0^{\gamma_1}|(Y_s,Z_s,Y_s^\delta,\tilde Z_s^\delta)|^2ds\right]\\
  &\leq&\ds E_{x,\xi_0}\left[\sup\limits_{0\le t\leq\gamma_1} |(r_t,\pi_t,P_t)|^4\right]^{\frac{1}{2}}\cdot E_{x,\xi_0}\left[\left(\int_0^{\gamma_1}|(Y_s,Z_s,Y_s^\delta,\tilde Z_s^\delta)|^2 ds\right)^2\right]^{\frac{1}{2}}\\
  &\leq&\ds NB_1^{1\over 2}(x,\xi_0)(|g|_0^2+|f(\cdot,0,0)|_0^2), \quad (x,\xi_0)\in D_{\delta_1}^\lambda\times\mathbb{R}^d.
\end{eqnarray*}\end{small}

\textbf{Step 2.} Choosing $\delta$ small enough, by $(H7)$, we have for $t\in[0,\gamma_1]$
 \begin{equation*}
 2\beta-(\epsilon_1+\epsilon_2+\epsilon_3)+(1+2\delta r_t)\left[2\mu-2(1+\epsilon)L_0^2\right]>0, \quad and \quad \left|\frac{1+2\delta r_t}{1+\epsilon}\right|<1.
 \end{equation*}
 Combining all the above estimates, we have for $(x,\xi_0,t)\in D_{\delta_1}^\lambda\times\mathbb{R}^d\times[0,\gamma_1]$,
  \begin{small}\begin{eqnarray}
   &&\ds E_{x,\xi_0}\left[|Y_t^\delta-Y_t|^2e^{2\beta t}\right]+E_{x,\xi_0}\int_t^{\gamma_1}e^{2\beta s}|Y_s^\delta-Y_s|^2ds+E_{x,\xi_0}\int_t^{\gamma_1}e^{2\beta s}\|\tilde Z_s^\delta-Z_s\|^2ds\nonumber\\
   &\leq&\ds E_{x,\xi_0}\left[|Y_{\gamma_1}^\delta-Y_{\gamma_1}|^2e^{2\beta\gamma_1}\right]+E_{x,\xi_0}\int_t^{\gamma_1}[\frac{1}{\epsilon_1}e^{2\beta s}\delta^2|\pi_s|^2\|\tilde Z_s^\delta\|^2+\frac{16}{\epsilon_2}e^{2\beta s}\delta^2r_s^2(L^2|Y_s|^2+L_0^2\|Z_s\|^2)\nonumber\\
   &&\ds +\frac{N}{1+\epsilon}e^{2\beta s}\delta^2\|\tilde Z_s^\delta\|^2(r_s^2+\|P_s\|^2)
   +\frac{\|f(\cdot)\|_{0,1}^2}{\epsilon_3}e^{2\beta s}|X_s^\delta-X_s|^2+\frac{8}{\epsilon_2}e^{2\beta s}\delta^2r_s^2|f(\cdot,0,0)|^2_0]ds+o(\delta^3)\nonumber\\
  &\leq&\ds E_{x,\xi_0}\left[|Y_{\gamma_1}^\delta-Y_{\gamma_1}|^2e^{2\beta\gamma_1}\right]
  +N\delta^2E_{x,\xi_0}\int_t^{\gamma_1}|(r_s,\pi_s,P_s)|^2 |(Y_s,Z_s,\tilde Z_s^\delta)|^2 ds\nonumber\\
  &&\ds+E_{x,\xi_0}\int_t^{\gamma_1}\frac{\|f(\cdot)\|_{0,1}^2}{\epsilon_3}e^{2\beta s}|X_s^\delta-X_s|^2ds+N|f(\cdot,0,0)|^2_0\  E\int_t^{\gamma_1}\delta^2\left(|\xi_t|^2+\frac{\psi^2_{(\xi_t)}}{\psi^2}\right)ds+o(\delta^3)\nonumber\\
  &\leq&\ds E_{x,\xi_0}\left[|Y_{\gamma_1}^\delta-Y_{\gamma_1}|^2e^{2\beta\gamma_1}\right]+E_{x,\xi_0}\int_t^{\gamma_1}\frac{\|f(\cdot)\|_{0,1}^2}{\epsilon_3}e^{2\beta s}|X_s^\delta-X_s|^2ds\nonumber\\
  &&\ds+N\delta^2B_1^{1\over 2}(x,\xi_0)(|g|^2_0+|f(\cdot,0,0)|^2_0)+o(\delta^3).\label{eq:result1_in_gradient}
  \end{eqnarray}\end{small}

By Cauchy inequality, we have
\begin{equation*}
  E_{x,\xi_0}\int_t^{\gamma_1}e^{2\beta s}|X_s^\delta-X_s|^2ds\leq E_{x,\xi_0}\int_t^{\gamma_1}2\left(\delta^2\left|\frac{X_s^\delta-X_s}{\delta}-\xi_s\right|^2+\delta^2|\xi_s|^2\right)ds.
\end{equation*}
Dividing $\delta^2$ in both sides and let $\delta\to 0$, due to Dominated convergence theorem, formula (\ref{eq:estin11}) and Assertion (iii) in Lemma \ref{le:quasiestimate1}, we know that
\begin{equation}\label{eq:result2_in_gradient}
 \lim\limits_{\delta\to 0}\frac{1}{\delta^2} E_{x,\xi_0}\int_t^{\gamma_1}\|f(\cdot)\|_{0,1}^2\epsilon_3^{-1}e^{2\beta s}|X_s^\delta-X_s|^2ds\leq NB_1^{1\over 2}(x,\xi_0)\|f(\cdot)\|_{0,1}^2.
\end{equation}

For $t=0$, we have
\begin{equation}\label{eq:result3_in_gradient}
 \lim\limits_{\delta\to 0}\left|\frac{Y_0^\delta-Y_0}{\delta}\right|=|u_{(\xi_0)}(x)|.
\end{equation}
In view of formulas (\ref{eq:estin1_in_pro1}), \eqref{eq:result1_in_gradient}, \eqref{eq:result2_in_gradient} and \eqref{eq:result3_in_gradient}, we have
\begin{equation}\label{eq:estin1}
 \frac{|u_{(\xi_0)}(x)|^2}{B_1^{1\over 2}(x,\xi_0)}\leq \sup\limits_{(y,\zeta)\in\partial D_{\delta_1}^\lambda\times\mathbb{R}^d\setminus \{0\}}\frac{|u_{(\zeta)}(y)|^2}{\sqrt{B_1(y,\zeta)}}+N(|g|^2_0+\|f(\cdot)\|_{0,1}^2),\quad \forall (x,\xi_0)\in D_{\delta_1}^\lambda\times\mathbb{R}^d\setminus\{0\}.
\end{equation}

\textbf{Step 3.} In view of Assertion (iii) in Lemma \ref{le:quasiestimate2}, we have for $(x,\xi_0,t)\in D_{\lambda^2}\times\mathbb{R}^d\times[0,\gamma_2]$
  \begin{eqnarray*}
   &&\ds E_{x,\xi_0}\left[|Y_t^\delta-Y_t|^2e^{2\beta t}\right]+E_{x,\xi_0}\int_t^{\gamma_2}e^{2\beta s}|Y_s^\delta-Y_s|^2ds+E_{x,\xi_0}\int_t^{\gamma_2}e^{2\beta s}\|\tilde Z_s^\delta-Z_s\|^2ds\\
   &\leq&\ds E_{x,\xi_0}\left[|Y_{\gamma_2}^\delta-Y_{\gamma_2}|^2e^{2\beta\gamma_2}\right]+E_{x,\xi_0}\int_t^{\gamma_2}[\frac{1}{\epsilon_1}e^{2\beta s}\delta^2|\pi_s|^2\|\tilde Z_s^\delta\|^2+\frac{16}{\epsilon_2}e^{2\beta s}\delta^2r_s^2(L^2|Y_s|^2+L_0^2\|Z_s\|^2)\\
   &&\ds +\frac{N}{1+\epsilon}e^{2\beta s}\delta^2\|\tilde Z_s^\delta\|^2(r_s^2+\|P_s\|^2)
   +\frac{\|f(\cdot)\|_{0,1}^2}{\epsilon_3}e^{2\beta s}|X_s^\delta-X_s|^2\\
   &&\ds+\frac{8}{\epsilon_2}e^{2\beta s}\delta^2r_s^2|f(X_s^\delta,0,0)|^2]ds+o(\delta^3)\\
  &\leq&\ds E_{x,\xi_0}\left[|Y_{\gamma_2}^\delta-Y_{\gamma_2}|^2e^{2\beta\gamma_2}\right]+E_{x,\xi_0}\int_t^{\gamma_2}\frac{\|f(\cdot)\|_{0,1}^2}{\epsilon_3}e^{2\beta s}|X_s^\delta-X_s|^2ds\\
  &&\ds +N\delta^2E_{x,\xi_0}\int_t^{\gamma_2}e^{2\beta s}|(r_s,\pi_s,P_s)|^2 |(Y_s,Z_s,\tilde Z_s^\delta)|^2ds\\
  &&\ds+|f(\cdot,0,0)|^2_0E_{x,\xi_0}\int_t^{\gamma_2}\delta^2e^{2\beta s}|\xi_t|^2ds+o(\delta^3)\\
  &\leq&\ds E_{x,\xi_0}\left[|Y_{\gamma_2}^\delta-Y_{\gamma_2}|^2e^{2\beta\gamma_2}\right]+E_{x,\xi_0}\int_t^{\gamma_2}\frac{\|f(\cdot)\|_{0,1}^2}{\epsilon_3}e^{2\beta s}|X_s^\delta-X_s|^2ds\\
  & &\ds+N\delta^2B_2^{1\over 2}(\xi_0)(|g|^2_0+|f(\cdot,0,0)|^2_0)+o(\delta^3).
  \end{eqnarray*}
Then repeating the arguments in \textbf{Step 2}, and using formula (\ref{eq:estin2_in_pro1}), we conclude that
\begin{equation}\label{eq:estin2}
 \frac{|u_{(\xi_0)}(x)|^2}{B_2^{1\over 2}(\xi_0)}\leq \sup\limits_{(y,\zeta)\in\partial D_{\lambda^2}\times\mathbb{R}^d\setminus \{0\}}\frac{|u_{(\zeta)}(y)|^2}{B_2^\sq(\zeta)}+N(|g|^2_0+\|f(\cdot)\|_{0,1}^2), \quad \forall (x,\xi_0)\in D_{\lambda^2}\times\mathbb{R}^d\setminus\{0\}.
\end{equation}

\textbf{Step 4.} Finally, due to Lemma \ref{le:estin2_in_sec4} and Remark \ref{re:sec4}, we have for $\xi_0\in\mathbb{R}^d$,
\begin{equation}
 |u_{(\xi_0)}(x)|\leq N\left(|\xi_0|+\frac{|\psi_{(\xi_0)}(x)|}{\psi^{\frac{3}{4}}(x)}\right)(|g|_{0,1}+\|f(\cdot)\|_{0,1}),\quad a.e.\ in\ D.
\end{equation}

We can prove (\ref{eq:estin13})  using the same arguments in \textbf{Steps 2} and \textbf{3} and  the BDG inequality.
The proof is complete.
\end{proof}

Before ending this section, we provide some estimates as follows which play an important role in the next subsection. We emphasize that when $\tilde r$ is not vanishing, Theorem \ref{thm:firstorder} still holds.
\begin{Corollary}
Let $(Y,Z)$ be the unique solution of BSDE (\ref{eq:BSDE1}) and $(Y^\delta, Z^\delta)$ be the unique solution of FBSDE (\ref{eq:FBSDE1}), with $(\tilde\pi_t,\tilde P_t)$ vanishing. Let assumptions $(H1)$-$(H5)$ and $(H7)$ be satisfied. If $f\in C^1(\overline{D}\times\mathbb{R}^k\times\mathbb{R}^{k\times d_1})$ with bounded partial derivatives and $g, u\in C^1(\overline{D})$, we have for sufficiently small $\delta>0$, there exists a constant $N=N(K_0, d, d_1, k, D, \mu, L, L_0)$ such that

\begin{eqnarray}
 & &\ds E_{x,\xi_0}\left[\left(\int_0^{\gamma_1}e^{2\beta t}|Y_t^\delta-Y_t|^2dt\right)^2+\left(\int_0^{\gamma_1}e^{2\beta t}\|\tilde Z_t^{\delta}-Z_t\|^2dt\right)^2+\sup\limits_{0\leq t\leq\gamma_1}(e^{4\beta t}|Y_t^\delta-Y_t|^4)\right]\nonumber\\
 &\leq&\ds N\delta^4 B_3^{\frac{1}{2}}(x,\xi_0)(|g|^4_1+\|f(\cdot)\|_{0,1}^4)+\delta^4 o(\delta, n,T),\quad \forall (x,\xi_0)\in D_{\delta_1}^\lambda\times\mathbb{R}^d, \label{eq:estin14}
 \end{eqnarray}
 and
 \begin{eqnarray}
 & &\ds E_{x,\xi_0}\left[\left(\int_0^{\gamma_2}e^{2\beta t}|Y_t^\delta-Y_t|^2dt\right)^2+\left(\int_0^{\gamma_2}e^{2\beta t}\|\tilde Z_t^{\delta}-Z_t\|^2dt\right)^2+\sup\limits_{0\leq t\leq\gamma_2}(e^{4\beta t}|Y_t^\delta-Y_t|^4)\right]\nonumber\\
 &\leq&\ds N\delta^4 B_4^{\frac{1}{2}}(\xi_0)(|g|^4_1+\|f(\cdot)\|_{0,1}^4)+\delta^4 o(\delta, n,T), \quad \forall (x,\xi_0)\in D_{\lambda^2}\times\mathbb{R}^d,\label{eq:estin15}
\end{eqnarray}
where $o(\delta, n,T)$ is an infinitesimal  as first $\delta\to 0$ and then $T, n\to \infty$.
\end{Corollary}

\subsection{Interior Hessian Estimate}\label{sec:second}
In this subsection, let $(X, Y, Z)$ be the unique solution of \eqref{eq:SDE2} and \eqref{eq:BSDE2}, $(\xi, \eta)$ be the unique solution of \eqref{eq:Qua1} and \eqref{eq:Qua3}, and $(X^\delta, Y^\delta, Z^\delta)$ be the unique solution of FBSDE \eqref{eq:FBSDE1}. Define $u$ by (\ref{eq:Sol1}). Choose $(\eta_0,\tilde\pi_t,\tilde P_t)$ to be vanishing for simplicity.

\begin{Theorem}\label{thm:secondorder}
 Let assumptions $(H1)$-$(H5)$, $(H6)_1$, $(H7)$-$(H9)$ and $(H10)_2$ be satisfied. Then $u\in C^{1,1}_{loc}(D)\cap C^{0,1}(\overline{D})$, such that for $\xi_0\in \mathbb{R}^d$ and $a.e. \ x\in D$,
\begin{equation*}
 |u_{(\xi_0)(\xi_0)}(x)|\leq N\left(|\xi_0|^2+\frac{\psi_{(\xi_0)}^2(x)}{\psi^{\frac{7}{4}}(x)}\right)[|g|_{1,1}+\|f(\cdot)\|_{0,1}+[f]_{1,1}(1+|g|_1^2+\|f(\cdot)\|_{0,1}^2)]
\end{equation*}
where $N=N(K_0, d, d_1, k, D, L_0, L, \mu)$.
\end{Theorem}

The following lemma is analogous to \cite[Lemma 3.2]{WEI1}.

\begin{Lemma}\label{le:normalestimate}
Let $f\in C^1(\overline{D}\times\mathbb{R}^k\times\mathbb{R}^{k\times d_1})$ with bounded partial derivatives, $g\in C^2(\overline{D})$, $u\in C^1(\overline{D})$, $(H4)$, $(H5)$, $(H7)$ and $E\tau(x)\leq\psi(x),\ x\in D$ be satisfied. Then we have
\begin{equation*}
 |u_{(n)}(y)|\leq N(|g|_2+|f(\cdot,0,0)|_0), \quad \forall y\in\partial D,
\end{equation*}
where $n:=n(y)$ is the unitary inward normal on $\partial D$ and the positive constant $N$ depends on the quadruple $(K_0, L_0, L, \mu)$. Note that $N$ does not depend on $|u_x|_0$.
\end{Lemma}
\begin{proof}
 Fix a $y\in\partial D$, and choose $\epsilon_0>0$ so that $x:=y+\epsilon n\in D$ for $0<\epsilon\leq \epsilon_0$.
 Set $\tilde Y_t:=Y_t-g(X_t)$ and $\tilde Z_t:=Z_t-\nabla g(X_t)\sigma(X_t)$, for $t\in[0,\tau]$, where $X$ is the unique solution to \eqref{eq:SDE2} and $(Y,Z)$ is the unique solution to \eqref{eq:BSDE2}. As $g\in C^2(\overline{D})$, $(\tilde Y, \tilde Z)$ is the unique solution to the BSDE
 \begin{equation}\label{eq:BSDE3}
  \begin{cases}
   &d\tilde Y_t=[-f(X_t,\tilde Y_t+g(X_t),\tilde Z_t+\nabla g(X_t)\sigma(X_t))-\cL g(X_t)]dt+\tilde Z_tdW_t, \quad t\in[0,\tau),\\
   &\tilde Y_\tau=0.
  \end{cases}
 \end{equation}
 With the analogue of \eqref{eq:estin10}, we have
 \begin{equation}\label{eq:estin1_in_normal}
   E\left[\sup\limits_{0\leq t\leq\tau}|\tilde Y_t|^2\right]+E\int_0^\tau \|\tilde Z_t\|^2dt\leq N(|g|_2^2+|f(\cdot,0,0)|^2_0)\psi(x).
 \end{equation}
 Then, by formulas (\ref{eq:Sol1}), (\ref{eq:BSDE3}) and \eqref{eq:estin1_in_normal}, we have
 \begin{eqnarray*}
   u(x)&=&\ds g(x)+E_x\int_0^\tau \cL g(X_t)+f(X_t,\tilde Y_t+g(X_t),\tilde Z_t+\nabla g(X_t)\sigma(X_t))dt\\
    &\leq&\ds g(x)+N(|g|_2+|f(\cdot,0,0)|_0)\psi(x).
 \end{eqnarray*}
 Since $u(y)=g(y)$ and $\psi(y)=0$, we have
 \begin{equation*}
   \frac{u(y+\epsilon n)-u(y)}{\epsilon}\leq \frac{g(y+\epsilon n)-g(y)}{\epsilon}+N(|g|_2+|f(\cdot,0,0)|_0)\left[\frac{\psi(y+\epsilon n)-\psi(y)}{\epsilon}\right].
 \end{equation*}
 Letting $\epsilon\to 0$, we get
 \begin{equation*}
  u_{(n)}(y)\leq N(|g|_2+|f(\cdot,0,0)|_0).
 \end{equation*}
 Replacing $u$ with $-u$ yields an estimate of $u_{(n)}$ from below. The proof is complete.
\end{proof}

Now we are ready to prove Theorem \ref{thm:secondorder}.
\begin{proof}[\textbf{Proof of Theorem \ref{thm:secondorder}}]

Set $\nabla X_t^\delta:=X_t^\delta-X_t$ and $\nabla X_t^{-\delta}:=X_t-X_t^{-\delta}$. Define $\nabla Y_t^\delta, \nabla Y_t^{-\delta}$, $\nabla Z_t^\delta$ and $\nabla Z_t^{-\delta}$ in the same way. Set $\bar{\gamma}:=\tau\wedge\tau^\delta\wedge\tau^{-\delta}\wedge k_n\wedge k_n'\wedge T$, where $k_n:=\inf\{t\geq 0;|\xi_t|\geq n\}$, $k_n':=\inf\{t\geq 0;|\eta_t|\geq n\}$ and $T\in[1,\infty)$.

\textbf{Step 1.} Set $\bar{\gamma}_1:=\bar{\gamma}\wedge\tau_1$.
Write
\begin{eqnarray*}
&&{\rm II_1}:=2\langle(\nabla Y_t^\delta-\nabla Y_t^{-\delta}),(2\delta r_t)[f(X_t^\delta,Y_t^\delta,Z_t^\delta)- f(X_t^{-\delta},Y_t^{-\delta},Z_t^{-\delta})],\\
 &&{\rm II_2}:=2\langle(\nabla Y_t^\delta-\nabla Y_t^{-\delta}),(\delta^2 \tilde r_t)[f(X_t^\delta,Y_t^\delta,Z_t^\delta)+ f(X_t^{-\delta},Y_t^{-\delta},Z_t^{-\delta})]\rangle,\\
  &&{\rm II_3}:=2\langle(\nabla Y_t^\delta-\nabla Y_t^{-\delta}),(\tilde Z_t^\delta-\tilde Z_t^{-\delta})\delta\pi_t\rangle,\\
  &&{\rm II_4}:=2\langle(\nabla Y_t^\delta-\nabla Y_t^{-\delta}),[f(X_t^\delta,Y_t^\delta,Z_t^\delta)
 -2f(X_t,Y_t,Z_t)+f(X_t^{-\delta},Y_t^{-\delta},Z_t^{-\delta})]\rangle.
 \end{eqnarray*}
By It\^{o}'s formula, we have for $(x,\xi_0,t)\in D_{\delta_1}^\lambda\times\mathbb{R}^d\times[0,\bar{\gamma}_1]$
\begin{eqnarray*}
 &&d\left(\left|\nabla Y_t^\delta-\nabla Y_t^{-\delta}\right|^2e^{4\beta t}\right)\\
 &=&\ds -({\rm II_1}+{\rm II_2}+{\rm II_3}+{\rm II_4})e^{4\beta t}dt+4\beta\left|\nabla Y_t^\delta-\nabla Y_t^{-\delta}\right|^2e^{4\beta t}dt\\
 &&\displaystyle +\left\|\tilde Z_t^\delta-2Z_t+\tilde Z_t^{-\delta}\right\|^2e^{4\beta t}dt+2\llan\left(\nabla Y_t^\delta-\nabla Y_t^{-\delta}\right),\left(\tilde Z_t^\delta-2Z_t+\tilde Z_t^{-\delta}\right)e^{4\beta t}dW_t\rran.
\end{eqnarray*}

First, we estimate the term ${\rm II_1}$. Let $f\in C^1(\overline{D}\times\mathbb{R}^k\times\mathbb{R}^{k\times d_1})$ with bounded partial derivatives and $(H4)$ be satisfied. By Cauchy inequality, we have for $\epsilon\in(0,1)$
\begin{eqnarray*}
 {\rm II_1}&\leq&\displaystyle \epsilon |\nabla Y_t^\delta-\nabla Y_t^{-\delta}|^2
  +4\delta^2r_t^2\epsilon^{-1}|f(X_t^\delta,Y_t^\delta,Z_t^\delta)-f(X_t^{-\delta},Y_t^{-\delta},Z_t^{-\delta})|^2\\
  &\leq&\displaystyle \epsilon |\nabla Y_t^\delta-\nabla Y_t^{-\delta}|^2+N\delta^2r_t^2\|f(\cdot)\|_{0,1}^2\epsilon^{-1}|X_t^\delta-X_t^{-\delta}|^2\\
  & &\ds+N\delta^2r_t^2\epsilon^{-1}(|Y_t^\delta-Y_t^{-\delta}|^2+\|Z_t^\delta-Z_t^{-\delta}\|^2).
\end{eqnarray*}
In view of the estimates in Lemmas \ref{le:stopping} and \ref{le:quasiestimate3} and formula (\ref{eq:estin11}), we have for $(x,\xi_0)\in D_{\delta_1}^\lambda\times\mathbb{R}^d$,
\begin{eqnarray}
  &&\ds E_{x,\xi_0}\left[\left(\int_0^{\bar{\gamma}_1}|X_t^\delta-X_t^{-\delta}|^2dt\right)^2\right]\nonumber\\
  &\leq&\displaystyle NE_{x,\xi_0}\left[\left( \int_0^{\bar{\gamma}_1}(|X_t^\delta-X_t|^2+|X_t-X_t^{-\delta}|^2)dt \right)^2\right]\nonumber\\
  &\leq&\ds NE_{x,\xi_0} \left[\sup\limits_{0\leq t\leq\bar{\gamma}_1}|X_t^\delta-X_t|^4\cdot \bar{\gamma}_1^2 \right]\nonumber\\
  &\leq&\ds N\delta^4  E_{x,\xi_0}\left[\sup\limits_{0\leq t\leq\bar{\gamma}_1}\left|\frac{X_t^\delta-X_t}{\delta}-\xi_t\right|^8\right]^{\frac{1}{2}}
  +N\delta^4 E_{\xi_0}\left[\sup\limits_{0\leq t\leq\bar{\gamma}_1}|\xi_t|^8\right]^{\frac{1}{2}}\nonumber\\
  &\leq&\ds N\delta^4B_3^{\frac{1}{2}}(x,\xi_0)+o(\delta^5).\label{eq:estin1_in_hessian}
\end{eqnarray}
In view of formula (\ref{eq:estin14}), we have for $(x,\xi_0)\in D_{\delta_1}^\lambda\times\mathbb{R}^d$,
\begin{eqnarray}
 \ds E_{x,\xi_0}\left[\left(\int_0^{\bar{\gamma}_1}|Y_t^\delta-Y_t^{-\delta}|^2 e^{2\beta t}dt\right)^2\right]&\leq&\ds NE_{x,\xi_0}\left[\left(\int_0^{\bar{\gamma}_1}|Y_t^\delta-Y_t|^2 e^{2\beta t}dt\right)^2\right]\nonumber\\
 &\leq&\ds N\delta^4 B_3^{\frac{1}{2}}(x,\xi_0)(|g|_1^4+\|f(\cdot)\|_{0,1}^4)+\delta^4o(\delta, n, T),
 \label{eq:estin2_in_hessian}
 \end{eqnarray}
 and
\begin{eqnarray}
&&\ds E_{x,\xi_0} \left[\left( \int_0^{\bar{\gamma}_1}\|Z_t^\delta-Z_t^{-\delta}\|^2e^{2\beta t}dt \right)^2\right]\nonumber\\
 &\leq&\displaystyle NE_{x,\xi_0} \left[\left( \int_0^{\bar{\gamma}_1}(\|Z_t^\delta-\tilde Z_t^\delta\|^2
 +\|\tilde Z_t^\delta-Z_t\|^2e^{2\beta t})dt \right)^2\right]\nonumber\\
 &\leq&\displaystyle NE_{x,\xi_0} \left[\left( \int_0^{\bar{\gamma}_1}\|Z_t^\delta-\tilde Z_t^\delta\|^2dt \right)^2\right]+N\delta^4B_3^{\frac{1}{2}}(x,\xi_0)(|g|_1^4+\|f(\cdot)\|_{0,1}^4)+\delta^4o(\delta, n, T). \label{eq:estin3_in_hessian}
\end{eqnarray}
Using Taylor expansion to deal with the first term in the preceding inequality, we have
\begin{eqnarray}
 &&\displaystyle 1-(1+2\delta r_t+\delta^2\tilde r_t)^\frac{1}{2}e^{\delta P_t}\nonumber\\
 &=&\displaystyle \left[1-(1+2\delta r_t+\delta^2\tilde r_t)^\frac{1}{2}\right]+(1+2\delta r_t+\delta^2\tilde r_t)^\frac{1}{2}\left[1-e^{\delta P_t}\right]\nonumber\\
 &=&\displaystyle \left[-\frac{1}{2}(2\delta r_t+\delta^2\tilde r_t)+\frac{1}{8}(2\delta r_t+\delta^2\tilde r_t)^2+o(\delta^3)\right]\nonumber\\
 &&\displaystyle +(1+2\delta r_t+\delta^2\tilde r_t)^\frac{1}{2}\left[-\delta P_t-\frac{1}{2}\delta^2P_t^2+o(\delta^3)\right].\label{eq:taylor_in_hessian}
\end{eqnarray}
 In view of Assertion (ii) in Lemma \ref{le:quasiestimate3} and Proposition \ref{pro:estin9}, for sufficiently small $\delta$ and $t\in[0,\bar{\gamma}_1]$, we have

\begin{eqnarray}
 &&\displaystyle E_{x,\xi_0} \left[\left( \int_0^{\bar{\gamma}_1}\|Z_t^\delta-\tilde Z_t^\delta\|^2dt \right)^2\right]\nonumber\\
 &=&\ds E_{x,\xi_0} \left[\left( \int_0^{\bar{\gamma}_1}\|Z_t^\delta\|^2( 1-(1+2\delta r_t+\delta^2\tilde r_t)^\frac{1}{2}e^{\delta P_t})^2dt \right)^2\right]\nonumber\\
 &=&\ds E_{x,\xi_0} \left[\left( \int_0^{\bar{\gamma}_1}\|Z_t^\delta\|^2(\delta^2|(r_t,P_t)|^2+o(\delta^3))dt \right)^2\right]\nonumber\\
 &\leq&\ds N\delta^4 E_{x,\xi_0} \left[\sup\limits_{0\leq t\leq\bar{\gamma}_1}|(r_t,P_t)|^4\cdot\left(\int_0^{\bar{\gamma}_1}\|\tilde Z_t^\delta\|^2dt\right)^2 \right]+o(\delta^6)\nonumber\\
 &\leq&\displaystyle N\delta^4B_3^{\frac{1}{2}}(x,\xi_0)(|g|_0^4+|f(\cdot,0,0)|^4_0)+o(\delta^6).\label{eq:estin4_in_hessian}
\end{eqnarray}
Collecting the above estimates (\ref{eq:estin1_in_hessian}), (\ref{eq:estin2_in_hessian}), (\ref{eq:estin3_in_hessian}) and (\ref{eq:estin4_in_hessian}), we have
\begin{eqnarray}
&&\ds E_{x,\xi_0} \left[\left(\int_0^{\bar{\gamma}_1}e^{2\beta t}|(X_t^\delta,Y_t^\delta,Z_t^\delta)-(X_t^{-\delta},Y_t^{-\delta},Z_t^{-\delta})|^2dt\right)^2\right]\nonumber\\
 &\leq&\ds N\delta^4 B_3^{\frac{1}{2}}(x,\xi_0)(1+|g|_1^4+\|f(\cdot)\|_{0,1}^4)+\delta^4 o(\delta, n,T).\label{eq:estin28}
\end{eqnarray}

Thus, in view of Assertion \rm{(iii)} in Lemma \ref{le:quasiestimate1} and formula \eqref{eq:estin28}, we have for $(x,\xi_0)\in D_{\delta_1}^\lambda\times\mathbb{R}^d$
\begin{eqnarray}
 &&\ds E_{x,\xi_0}\int_t^{\bar{\gamma}_1}{\rm II_1}e^{4\beta s}ds\nonumber\\
 &\leq &\ds E_{x,\xi_0}\int_t^{\bar{\gamma}_1}\epsilon |\nabla Y_s^\delta-\nabla Y_s^{-\delta}|^2e^{4\beta s}ds\nonumber\\
 & &\ds +E_{x,\xi_0}\int_0^{\bar{\gamma}_1}N\delta^2r_t^2\epsilon^{-1}[\|f(\cdot)\|_{0,1}^2| X_t^\delta-X_t^{-\delta}|^2+|(Y_t^\delta,Z_t^\delta)-(Y_t^{-\delta},Z_t^{-\delta})|^2]e^{4\beta t}dt\nonumber\\
 &\leq&\ds E_{x,\xi_0}\int_t^{\bar{\gamma}_1}\epsilon |\nabla Y_s^\delta-\nabla Y_s^{-\delta}|^2e^{4\beta s}ds+N\delta^4 B_3^\frac{1}{2}(x,\xi_0)(|g|_1^2+\|f(\cdot)\|_{0,1}^2)+\delta^4 o(\delta, n,T).\nonumber\\
 \label{eq:estin21}
\end{eqnarray}
Second, we estimate the term ${\rm II_2}$. By $(H4)$ and Cauchy inequality, we have for $\epsilon\in(0,1)$
\begin{eqnarray*}
  {\rm II_2}&\leq&\displaystyle \epsilon \left|\nabla Y_t^\delta-\nabla Y_t^{-\delta}\right|^2+\delta^4\tilde r_t^2\epsilon^{-1}\left|f(X_t^\delta,Y_t^\delta,Z_t^\delta)+f(X_t^{-\delta},Y_t^{-\delta},Z_t^{-\delta})\right|^2\\
  &\leq&\displaystyle \epsilon \left|\nabla Y_t^\delta-\nabla Y_t^{-\delta}\right|^2+\delta^4\tilde r_t^2\epsilon^{-1}\left[|f(\cdot,0,0)|_0^2+L_0^2|Y_t^\delta|^2+L_0^2\|Z_t^{\delta}\|^2\right].
 \end{eqnarray*}
In view of Assertion \rm{(ii)} in Lemma \ref{le:quasiestimate3} and Proposition \ref{pro:estin9}, we have for $(x,\xi_0)\in D_{\delta_1}^\lambda\times\mathbb{R}^d$
\begin{eqnarray}
 &&\ds E_{x,\xi_0}\int_t^{\bar{\gamma}_1}{\rm II_2}e^{4\beta s}ds\nonumber\\
 &\leq &\ds E_{x,\xi_0}\int_t^{\bar{\gamma}_1}\epsilon |\nabla Y_s^\delta-\nabla Y_s^{-\delta}|^2e^{4\beta s}ds\nonumber\\
 &&\displaystyle +E_{x,\xi_0}\int_0^{\bar{\gamma}_1}\delta^4\tilde r_t^2\epsilon^{-1}\left[|f(\cdot,0,0)|_0^2+L_0^2|Y_t^\delta|^2+L_0^2\|Z_t^{\delta}\|^2\right]e^{4\beta t}dt\nonumber\\
 &\leq&\displaystyle E_{x,\xi_0}\int_t^{\bar{\gamma}_1}\epsilon |\nabla Y_s^\delta-\nabla Y_s^{-\delta}|^2e^{4\beta s}ds+N\delta^4B_3^{\frac{1}{2}}(x,\xi_0)(|g|_0^2+|f(\cdot,0,0)|_0^2).\label{eq:estin22}
\end{eqnarray}
Third, we estimate the term ${\rm II_3}$. Using Cauchy inequality, we have for $\epsilon\in(0,1)$
\begin{equation*}
  {\rm II_3}
  \leq \epsilon \left|\nabla Y_t^\delta-\nabla Y_t^{-\delta}\right|^2+4\delta^2|\pi_t|^2\epsilon^{-1}\left\|\tilde Z_t^\delta-Z_t\right\|^2.
\end{equation*}
In view of Assertion \rm{(ii)} in Lemma \ref{le:quasiestimate3} and formula (\ref{eq:estin14}), we have
 \begin{eqnarray*}
 && \ds E_{x,\xi_0}\int_0^{\bar{\gamma}_1}|\pi_t|^2\left\|\tilde Z_t^\delta-Z_t\right\|^2e^{4\beta t}dt \\
 &\leq&\ds E_{x,\xi_0}\left[\sup\limits_{0\leq t\leq\bar{\gamma}_1} |\pi_t|^4e^{4\beta t}\right]^{\frac{1}{2}}\cdot E_{x,\xi_0}\left[\left(\int_0^{\bar{\gamma}_1}\left\|\tilde Z_t^\delta-Z_t\right\|^2 e^{2\beta t}dt \right)^2\right]^{\frac{1}{2}}\\
  &\leq&\ds N B_3^\frac{1}{4}(x,\xi_0)\cdot \left[N\delta^4(|g|_1^4+\|f(\cdot)\|_{0,1}^4)B_3^\frac{1}{2}(x,\xi_0)+\delta^4 o(\delta, n,T)\right]^{\frac{1}{2}}\\
  &\leq&\ds N\delta^2 B_3^\frac{1}{2}(x,\xi_0)(|g|_1^2+\|f(\cdot)\|_{0,1}^2)+\delta^2 o^{1\over 2}(\delta, n,T),\quad \forall (x,\xi_0)\in D_{\delta_1}^\lambda\times\mathbb{R}^d.
 \end{eqnarray*}
 Hence, we have for $(x,\xi_0)\in D_{\delta_1}^\lambda\times\mathbb{R}^d$
 \begin{eqnarray}
   &&\ds E_{x,\xi_0}\int_t^{\bar{\gamma}_1}{\rm II_3}e^{4\beta s}ds\nonumber\\
 &\leq &\ds E_{x,\xi_0}\int_t^{\bar{\gamma}_1}\epsilon |\nabla Y_s^\delta-\nabla Y_s^{-\delta}|^2e^{4\beta s}ds+NE_{x,\xi_0}\int_0^{\bar{\gamma}_1}\delta^2|\pi_t|^2\left\|\tilde Z_t^\delta-Z_t\right\|^2e^{4\beta t}dt \nonumber\\
 &\leq&\ds E_{x,\xi_0}\int_t^{\bar{\gamma}_1}\epsilon |\nabla Y_s^\delta-\nabla Y_s^{-\delta}|^2e^{4\beta s}ds+N\delta^4 B_3^\frac{1}{2}(x,\xi_0)(|g|_1^2+\|f(\cdot)\|_{0,1}^2)+\delta^4 o^{1\over 2}(\delta, n,T).\nonumber\\
 \label{eq:estin23}
 \end{eqnarray}

Fourth, we estimate the term ${\rm II_4}$. Set
\begin{equation*}
 \Xi_{\lambda,t}^\delta:=(X_t+\lambda\nabla X_t^\delta,Y_t+\lambda\nabla Y_t^\delta, Z_t+\lambda\nabla Z_t^\delta).
\end{equation*}
We have
\begin{eqnarray*}
  &&\displaystyle f(X_t^\delta,Y_t^\delta,Z_t^\delta)-2f(X_t,Y_t,Z_t)+f(X_t^{-\delta},Y_t^{-\delta},Z_t^{-\delta})\\
  &=&\displaystyle \int_0^1 \left[f_x(\Xi_{\lambda,t}^\delta)-f_x(\Xi_{-\lambda,t}^{-\delta}) \right]d\lambda \nabla X_t^\delta+\int_0^1f_x(\Xi_{-\lambda,t}^{-\delta})d\lambda (\nabla X_t^\delta-\nabla X_t^{-\delta})\\
  &&\displaystyle +\int_0^1 \left[ f_y(\Xi_{\lambda,t}^\delta)-f_y(\Xi_{-\lambda,t}^{-\delta})\right]d\lambda \nabla Y_t^\delta+\int_0^1f_y(\Xi_{-\lambda,t}^{-\delta})d\lambda (\nabla Y_t^\delta-\nabla Y_t^{-\delta})\\
  &&\displaystyle +\int_0^1 \left[ f_z(\Xi_{\lambda,t}^\delta)-f_z(\Xi_{-\lambda,t}^{-\delta})\right]d\lambda \nabla Z_t^\delta+\int_0^1f_z(\Xi_{-\lambda,t}^{-\delta})d\lambda (\nabla Z_t^\delta-\nabla Z_t^{-\delta}).
\end{eqnarray*}
Set
\begin{eqnarray*}
&&{\rm II_{4.1}}:=|\Xi_{\lambda,t}^\delta-\Xi_{-\lambda,t}^{-\delta}|^2(|\nabla X_t^\delta|^2+|\nabla Y_t^\delta|^2)\\
&&{\rm II_{4.2}}:=|\nabla X_t^\delta-\nabla X_t^{-\delta}|^2, \\
&&{\rm II_{4.3}}:=2\langle(\nabla Y_t^\delta-\nabla Y_t^{-\delta}), \int_0^1[ f_z(\Xi_{\lambda,t}^\delta)-f_z(\Xi_{-\lambda,t}^{-\delta})]d\lambda \nabla Z_t^\delta+\int_0^1f_z(\Xi_{-\lambda,t}^{-\delta})d\lambda (\nabla Z_t^\delta-\nabla Z_t^{-\delta})\rangle.
\end{eqnarray*}
When $f\in C^2$ with bounded first and second order partial derivatives (see Remark \ref{re:sec4}), we have for $\epsilon\in(0,1)$
\begin{equation*}
 {\rm II_4}\leq(3\epsilon-2\mu)\left|\nabla Y_t^\delta-\nabla Y_t^{-\delta}\right|^2+ N[f]_{1,1}^2 {\rm II_{4,1}}+\|f(\cdot)\|_{0,1}^2\epsilon^{-1}{\rm II_{4.2}}+{\rm II_{4.3}}.
\end{equation*}
By H\"{o}lder's inequality, Assertion \rm{(iii)} in Lemma \ref{le:quasiestimate1}, formulas (\ref{eq:estin11}) and (\ref{eq:estin28}), we have
 \begin{small}\begin{eqnarray}
  &&\ds E_{x,\xi_0}\int_0^{\bar{\gamma}_1}{\rm II_{4.1}}e^{4\beta t}dt\nonumber\\
  &\leq&\ds E_{x,\xi_0}\left[\sup\limits_{0\leq t\leq\bar{\gamma}_1}(|\nabla X_t^\delta|^4+|\nabla Y_t^\delta|^4e^{4\beta t})\right]^\frac{1}{2}\cdot E_{x,\xi_0}\left[\left(\int_0^{\bar{\gamma}_1}
  \left|\Xi_{\lambda,t}^\delta-\Xi_{-\lambda,t}^{-\delta}\right|^2e^{2\beta t} dt\right)^2\right]^\frac{1}{2}\nonumber\\
  &\leq&\ds \left\{\delta^2E_{x,\xi_0}\left[\sup\limits_{0\leq t\leq\bar{\gamma}_1}\left|\frac{X_t^\delta-X_t}{\delta}-\xi_t\right|^4+\left|\xi_t\right|^4\right]^\frac{1}{2}
  +N\delta^2B_3^{\frac{1}{4}}(x,\xi_0)(|g|_1^2+\|f(\cdot)\|_{0,1}^2)+\delta^2 o^{1\over 2}(\delta, n,T)\right\}\nonumber\\
  & &\ds \cdot \left\{N\delta^4 B_3^{\frac{1}{2}}(x,\xi_0) [(1+|g|_1^4+\|f(\cdot)\|_{0,1}^4)+\delta^4 o(\delta, n,T)] \right\}^\frac{1}{2}\nonumber\\
  &\leq&\displaystyle N\delta^4B_3^{\frac{1}{2}}(x,\xi_0)(1+|g|_1^4+\|f(\cdot)\|_{0,1}^4)+\delta^4 o(\delta, n,T),\quad \forall (x,\xi_0)\in D_{\delta_1}^\lambda\times\mathbb{R}^d.\label{eq:estin262}
 \end{eqnarray}\end{small}
 Hence, we have for $(x,\xi_0)\in D_{\delta_1}^\lambda\times\mathbb{R}^d$
 \begin{equation}\label{eq:estin26}
   N[f]_{1,1}^2 E_{x,\xi_0}\int_0^{\bar{\gamma}_1}{\rm II_{4.1}}e^{4\beta t}dt
  \leq N\delta^4B_3^{\frac{1}{2}}(x,\xi_0)(1+|g|_1^4+\|f(\cdot)\|_{0,1}^4)[f]_{1,1}^2+\delta^4 o(\delta, n,T).
 \end{equation}
In view of Assertion \rm{(iii)} in Lemma \ref{le:quasiestimate3} and formula (\ref{eq:estin12}), we have for $\eta_0=0$
\begin{eqnarray}
 \ds E_{x,\xi_0}\int_0^{\bar{\gamma}_1}{\rm II_{4.2}}dt&\leq&\ds E_{x,\xi_0}\left[\sup\limits_{0\leq t\leq\bar{\gamma}_1}\left|\nabla X_t^\delta-\nabla X_t^{-\delta}\right|^4\right]^\frac{1}{2}\cdot E_{x,\xi_0}[\bar{\gamma}_1^2]^\frac{1}{2}\nonumber\\
 &\leq&\ds N\delta^4 E_{x,\xi_0,0}\left[\sup\limits_{0\leq t\leq\bar{\gamma}_1}\left|\frac{\nabla X_t^\delta-\nabla X_t^{-\delta}}{\delta^2}-\eta_t\right|^4+\left|\eta_t\right|^4\right]^\frac{1}{2}\nonumber\\
 &\leq&\displaystyle N\delta^4B_3^\frac{1}{2}(x,\xi_0)+o(\delta^5),\quad \forall (x,\xi_0)\in D_{\delta_1}^\lambda\times\mathbb{R}^d.\label{eq:estin27}
\end{eqnarray}

Fifth, we need to treat the hardest term $\rm II_{4.3}$ consisting of $\nabla Z_t^\delta$. From the estimates \eqref{eq:estin21}-\eqref{eq:estin27} and by $(H7)$, we choose positive constants $\epsilon_1,\ \epsilon_2$ and $\epsilon$ such that $$0<\epsilon_1<3[8(1+2c_p^2)]^{-1}, \quad\epsilon_2=c_p^2/4, \quad \hbox{ \rm and } \quad -4\beta+6\epsilon-2\mu+4L_0^2<0,$$
 where $c_p$ is the constant in the BDG inequality. We have for $(x,\xi_0)\in D_{\delta_1}^\lambda\times\mathbb{R}^d$,
\begin{small}\begin{eqnarray*}
 &&\ds E_{x,\xi_0}\left[\left|\nabla Y_t^\delta-\nabla Y_t^{-\delta}\right|^2e^{4\beta t}\right]+E_{x,\xi_0} \int_t^{\bar{\gamma}_1} \left\|\tilde Z_s^\delta-2 Z_s+\tilde Z_s^{-\delta}\right\|^2e^{4\beta s}ds\\
 &\leq&\ds E_{x,\xi_0}\left[\left|\nabla Y_{\bar{\gamma}_1}^\delta-\nabla Y_{\bar{\gamma}_1}^{-\delta}\right|^2e^{4\beta\bar{\gamma}_1}\right]+N_1\delta^4B_3^{\frac{1}{2}}(x,\xi_0)+\delta^4 o(\delta, n,T)\\
 &&\ds +E_{x,\xi_0}\int_t^{\bar{\gamma}_1}(-4\beta+6\epsilon-2\mu)\left|\nabla Y_s^{\delta}-\nabla Y_s^{-\delta}\right|^2e^{4\beta s}ds+E_{x,\xi_0}\int_t^{\bar{\gamma}_1}e^{4\beta s} {\rm II_{4.3}}ds\\
 &\leq&\ds E_{x,\xi_0}\left[\left|\nabla Y_{\bar{\gamma}_1}^\delta-\nabla Y_{\bar{\gamma}_1}^{-\delta}\right|^2e^{4\beta\bar{\gamma}_1}\right]+N_1\delta^4B_3^{\frac{1}{2}}(x,\xi_0)+\delta^4 o(\delta, n,T)\\
 &&\ds +\epsilon_1E_{x,\xi_0}\left[\sup\limits_{t\in [0,\bar{\gamma}_1]}\left|\nabla Y_t^\delta-\nabla Y_t^{-\delta}\right|^2e^{4\beta t}\right]+E\int_t^{\bar{\gamma}_1}\frac{\epsilon_2}{c_p^2}\left\|\nabla Z_s^\delta-\nabla Z_s^{-\delta}\right\|^2 e^{4\beta s}ds\\
  &&\ds+N[f]_{1,1}^2E_{x,\xi_0}\left[\left(\int_t^{\bar{\gamma}_1}\left|\Xi_{\lambda,s}^{\delta}-\Xi_{-\lambda,s}^{-\delta}\right|^2e^{2\beta s}ds\right)^2\right]^{\frac{1}{2}}
 \cdot E_{x,\xi_0}\left[\left(\int_t^{\bar{\gamma}_1}\left\|\nabla Z_s^{\delta}\right\|^2e^{2\beta s}ds\right)^2\right]^{\frac{1}{2}}\\
 &\leq&\ds E_{x,\xi_0}\left[\left|\nabla Y_{\bar{\gamma}_1}^\delta-\nabla Y_{\bar{\gamma}_1}^{-\delta}\right|^2e^{4\beta\bar{\gamma}_1}\right]+N_1\delta^4B_3^{\frac{1}{2}}(x,\xi_0)+\delta^4 o(\delta, n,T)\\
 &&\ds +\epsilon_1E_{x,\xi_0}\left[\sup\limits_{t\in [0,\bar{\gamma}_1]}\left|\nabla Y_t^\delta-\nabla Y_t^{-\delta}\right|^2e^{4\beta t}\right]+E_{x,\xi_0}\int_t^{\bar{\gamma}_1}\frac{\epsilon_2}{c_p^2}\left\|\nabla Z_s^\delta-\nabla Z_s^{-\delta}\right\|^2 e^{4\beta s}ds,
\end{eqnarray*}\end{small}
where
\begin{equation} \label{eq:N_1}
N_1:=N[|g|_1^2+\|f(\cdot)\|_{0,1}^2+[f]_{1,1}^2(1+|g|_1^4+\|f(\cdot)\|_{0,1}^4)].
\end{equation}
Meanwhile, using the BDG inequality, we have
\begin{small}\begin{eqnarray*}
 &&\ds E_{x,\xi_0}\left[\sup\limits_{t\in[0,\bar{\gamma}_1]}\left|\nabla Y_t^\delta-\nabla Y_t^{-\delta}\right|^2e^{4\beta t}\right]\\
 &\leq&\ds E_{x,\xi_0}\left[\left|\nabla Y_{\bar{\gamma}_1}^\delta-\nabla Y_{\bar{\gamma}_1}^{-\delta}\right|^2e^{4\beta\bar{\gamma}_1}\right]+N_1\delta^4B_3^{\frac{1}{2}}(x,\xi_0)+\delta^4 o(\delta, n,T)\\
 &&\ds +\epsilon_1E_{x,\xi_0}\left[\sup\limits_{t\in [0,\bar{\gamma}_1]}\left|\nabla Y_t^\delta-\nabla Y_t^{-\delta}\right|^2e^{4\beta t}\right]+E_{x,\xi_0}\int_0^{\bar{\gamma}_1}\frac{\epsilon_2}{c_p^2}\left\|\nabla Z_s^\delta-\nabla Z_s^{-\delta}\right\|^2e^{4\beta s}ds\\
 &&\ds +\frac{1}{2}E_{x,\xi_0}\left[\sup\limits_{t\in[0,\bar{\gamma}_1]}\left|\nabla Y_t^\delta-\nabla Y_t^{-\delta}\right|^2e^{4\beta t}\right]+2c_p^2E_{x,\xi_0}\int_0^{\bar{\gamma}_1} \left\|\tilde Z_s^\delta-2 Z_s+\tilde Z_s^{-\delta}\right\|^2e^{4\beta s}ds\\
 &\leq&\ds NE_{x,\xi_0}\left[\left|\nabla Y_{\bar{\gamma}_1}^\delta-\nabla Y_{\bar{\gamma}_1}^{-\delta}\right|^2e^{4\beta\bar{\gamma}_1}\right]+N_1\delta^4B_3^{\frac{1}{2}}(x,\xi_0)+\delta^4 o(\delta, n,T)\\
 &&\ds +(\epsilon_1+2c_p^2\epsilon_1)E_{x,\xi_0}\left[\sup\limits_{t\in [0,\bar{\gamma}_1]}\left|\nabla Y_t^\delta-\nabla Y_t^{-\delta}\right|^2e^{4\beta t}\right]\\
 &&+\left(\frac{\epsilon_2}{c_p^2}+2\epsilon_2\right)E_{x,\xi_0}\int_0^{\bar{\gamma}_1}\left\|\nabla Z_s^\delta-\nabla Z_s^{-\delta}\right\|^2e^{4\beta s}ds,\quad \forall (x,\xi_0)\in D_{\delta_1}^\lambda\times\mathbb{R}^d.\\
\end{eqnarray*}\end{small}
So, we have for $(x,\xi_0)\in D_{\delta_1}^\lambda\times\mathbb{R}^d$,
\begin{eqnarray}
 \displaystyle E_{x,\xi_0}\left[\left|\nabla Y_t^\delta-\nabla Y_t^{-\delta}\right|^2e^{4\beta t}\right]&\leq& NE_{x,\xi_0}\left[\left|\nabla Y_{\bar{\gamma}_1}^\delta-\nabla Y_{\bar{\gamma}_1}^{-\delta}\right|^2e^{4\beta\bar{\gamma}_1}\right]+N_1\delta^4B_3^{\frac{1}{2}}(x,\xi_0)\nonumber\\
 &&\displaystyle +\epsilon_3E_{x,\xi_0}\int_0^{\bar{\gamma}_1}\left\|\nabla Z_s^\delta-\nabla Z_s^{-\delta}\right\|^2e^{4\beta s}ds+\delta^4 o(\delta, n,T),\nonumber\\
 \label{eq:result1_in_Hessian}
\end{eqnarray}
where $\epsilon_3:=\epsilon_2c_p^{-2}+2\epsilon_1\epsilon_2(c_p^{-2}+2)[1-2\epsilon_1(1+2c_p^2)]^{-1}<1$.
Then, we estimate the term $\|\nabla Z_s^\delta-\nabla Z_s^{-\delta}\|^2$.
Set $${\rm II_5}:=\|Z_t^\delta-\tilde Z_t^\delta-\tilde Z_t^{-\delta}+Z_t^{-\delta}\|^2.$$
Through simple calculation, for sufficiently small $\epsilon_4\in(0,1)$ satisfying $(1+\epsilon_4)\epsilon_3<1$, we have
\begin{equation*}
\left\|\nabla Z_t^\delta-\nabla Z_t^{-\delta}\right\|^2\leq (1+\epsilon_4)\left\|\tilde Z_t^\delta-2Z_t+\tilde Z_t^{-\delta}\right\|^2+(1+\frac{1}{\epsilon_4}){\rm II_5}.
\end{equation*}
We continue to estimate the term ${\rm II_5}$.
Write
\begin{eqnarray*}
&&{\rm II_{5.1}}:=\delta^2r_t^2\|Z_t^\delta-Z_t^{-\delta}\|^2,\\
&&{\rm II_{5.2}}:=\delta^4(r_t^4+\tilde r_t^2)[\|Z_t^\delta\|^2+\|Z_t^{-\delta}\|^2],\\
&&{\rm II_{5.3}}:=\delta^2P_t^2\|Z_t^\delta(1+2\delta r_t+\delta^2\tilde r_t)^\frac{1}{2}-Z_t^{-\delta}(1-2\delta r_t+\delta^2\tilde r_t)^\frac{1}{2}\|^2,\\
&&{\rm II_{5.4}}:=\delta^4P_t^4[\|Z_t^\delta\|^2(1+2\delta r_t+\delta^2\tilde r_t)+\|Z_t^{-\delta}\|^2(1-2\delta r_t+\delta^2\tilde r_t)].
\end{eqnarray*}
Using Taylor expansion as formula \eqref{eq:taylor_in_hessian} to the terms $1-(1+2\delta r_t+\delta^2\tilde r_t)^\frac{1}{2}e^{\delta P_t}$ and $1-(1-2\delta r_t+\delta^2\tilde r_t)^\frac{1}{2}e^{-\delta P_t}$, we have
\begin{eqnarray*}
 {\rm II_5}&=&\displaystyle \|\delta r_t(-Z_t^\delta+Z_t^{-\delta})+(-\frac{1}{2}\delta^2\tilde r_t+\frac{1}{2}\delta^2r_t^2)(Z_t^\delta+Z_t^{-\delta})\\
&&\displaystyle +\delta P_t(-Z_t^\delta(1+2\delta r_t+\delta^2\tilde r_t)^\frac{1}{2}+Z_t^{-\delta}(1-2\delta r_t+\delta^2\tilde r_t)^\frac{1}{2})\\
 &&\displaystyle -\frac{1}{2}\delta^2P_t^2(Z_t^\delta(1+2\delta r_t+\delta^2\tilde r_t)^\frac{1}{2}+Z_t^{-\delta}(1-2\delta r_t+\delta^2\tilde r_t)^\frac{1}{2})\|^2+o(\delta^5)\\
 &\leq& N({\rm II_{5.1}}+{\rm II_{5.2}}+{\rm II_{5.3}}+{\rm II_{5.4}})+o(\delta^{\, 5}).
\end{eqnarray*}
Then, in view of Assertion \rm{(iii)} in Lemma \ref{le:quasiestimate1} and formula (\ref{eq:estin28}), we have
\begin{eqnarray*}
 &&\ds E_{x,\xi_0}\int_0^{\bar{\gamma}_1}e^{4\beta t}{\rm II_{5.1}}dt\\
 &\leq&\displaystyle \delta^2E_{x,\xi_0}\left[\sup\limits_{ 0\leq t\leq\bar{\gamma}_1} r_t^4e^{4\beta t}\right]^{\frac{1}{2}}\cdot E_{x,\xi_0}\left[\left(\int_0^{\bar{\gamma}_1}\left\|Z_t^\delta-Z_t^{-\delta}\right\|^2e^{2\beta t}dt\right)^2\right]^\frac{1}{2}\\
 &\leq&\displaystyle N\delta^4 B_3^\frac{1}{2}(x,\xi_0)(|g|^2_1+\|f(\cdot)\|_{0,1}^2)+\delta^4 o(\delta, n, T),\quad \forall (x,\xi_0)\in D_{\delta_1}^\lambda\times\mathbb{R}^d.
\end{eqnarray*}
In view of Assertion \rm{(ii)} in Lemma \ref{le:quasiestimate3} and Proposition \ref{pro:estin9}, we have
\begin{eqnarray*}
 &&\ds E_{x,\xi_0}\int_0^{\bar{\gamma}_1}e^{4\beta t}{\rm II_{5.2}}dt\\
 &\leq &\displaystyle \delta^4E_{x,\xi_0}\left\{\sup\limits_{0\leq t\leq\bar{\gamma}_1}\left[(r_t^4+\tilde r_t^2)e^{4\beta t}\right]\cdot\int_0^{\bar{\gamma}_1}(\left\|Z_t^{\delta}\right\|^2+\left\|Z_t^{-\delta}\right\|^2)dt\right\}\\
 &\leq&\displaystyle N\delta^4 B_3^\frac{1}{2}(x,\xi_0)(|g|_0^2+|f(\cdot,0,0)|_0^2),\quad \forall (x,\xi_0)\in D_{\delta_1}^\lambda\times\mathbb{R}^d.
\end{eqnarray*}
Terms ${\rm II_{5.3}}$ and ${\rm II_{5.4}}$ can be estimated in a similar way.
In summary, we have
\begin{equation}\label{eq:result2_in_Hessian}
 E_{x,\xi_0}\int_0^{\bar{\gamma}_1} e^{4\beta t}{\rm II_5} dt\leq N\delta^4 B_3^\frac{1}{2}(x,\xi_0)(|g|^2_1+\|f(\cdot)\|_{0,1}^2)+\delta^4 o(\delta, n, T),\quad \forall (x,\xi_0)\in D_{\delta_1}^\lambda\times\mathbb{R}^d.
\end{equation}
So, by formulas \eqref{eq:result1_in_Hessian} and \eqref{eq:result2_in_Hessian}, we have for $(x,\xi_0)\in D_{\delta_1}^\lambda\times\mathbb{R}^d$
\begin{equation}\label{eq:result3_in_Hessian}
E_{x,\xi_0}\left[\left|\nabla Y_t^\delta-\nabla Y_t^{-\delta}\right|^2e^{4\beta t}\right]\leq NE\left[\left|\nabla Y_{\bar{\gamma}_1}^\delta-\nabla Y_{\bar{\gamma}_1}^{-\delta}\right|^2e^{4\beta \bar{\gamma}_1}\right]+N_1\delta^4B_3^{\frac{1}{2}}(x,\xi_0)+\delta^4 o(\delta, n, T).
\end{equation}

Sixth, repeating \textbf{Step 2} in the proof of Theorem \ref{thm:firstorder} and using Assertion \rm{(iii)} in Lemma \ref{le:quasiestimate3}, we have for $(x,\xi_0)\in D_{\delta_1}^\lambda\times\mathbb{R}^d$ and $\eta_0=0$
 \begin{small}\begin{eqnarray}
  &&\ds \lim\limits_{\delta\to 0}\frac{1}{\delta^4}E_{x,\xi_0}\left[\left|Y_{\bar{\gamma}_1}^\delta-2Y_{\bar{\gamma}_1}+Y_{\bar{\gamma}_1}^{-\delta}\right|^2 e^{4\beta\bar{\gamma}_1}\right]\nonumber\\
  &\leq&\ds \lim\limits_{\delta\to 0}E_{x,\xi_0,0}\left[\left|\frac{u(X_{\bar{\gamma}_1}^\delta)-2u(X_{\bar{\gamma}_1})+u(X_{\bar{\gamma}_1}^{-\delta})}{\delta^2}-u_{(\eta_{\bar{\gamma}_1})}(X_{\bar{\gamma}_1})
  -u_{(\xi_{\bar{\gamma}_1})(\xi_{\bar{\gamma}_1})}(X_{\bar{\gamma}_1})\right|^2e^{4\beta\bar{\gamma}_1}\right]\nonumber\\
  &&\ds+\lim\limits_{\delta\to 0}E_{x,\xi_0,0}\left[\left|u_{(\eta_{\bar{\gamma}_1})}(X_{\bar{\gamma}_1})+u_{(\xi_{\bar{\gamma}_1})(\xi_{\bar{\gamma}_1})}(X_{\bar{\gamma}_1})\right|^2e^{4\beta\bar{\gamma}_1}\right]\nonumber\\
  &\leq&\ds \lim\limits_{\delta\to 0}\left[E_{x,\xi_0}\left|u_{(\xi_{\bar{\gamma}_1})(\xi_{\bar{\gamma}_1})}(X_{\bar{\gamma}_1})\right|^2+\mathop{\sup\limits_{y\in\partial D_{\delta_1}^\lambda,}}_{ |\zeta|=1} \left|u_{(\zeta)}(y)\right|^2\cdot E_0\left|\eta_{\bar{\gamma}_1}\right|^2\right]\nonumber\\
  &\leq&\displaystyle \mathop{\sup\limits_{y\in\partial D_{\delta_1}^\lambda,}}_{ 0\neq\zeta\in\mathbb{R}^d}\frac{\left|u_{(\zeta)(\zeta)}(y)\right|^2}{B_3^{1\over 2}(y,\zeta)}B_3^\sq(x,\xi_0)+\left[\mathop{\sup\limits_{y\in\partial D_{\delta_1}^\lambda,}}_{|\zeta|=1} |u_{(\zeta)}(y)|^2\right]NB_3^\sq(x,\xi_0).\label{eq:result4_in_Hessian}
 \end{eqnarray}\end{small}

For $t=0$, divide $\delta^4$ in both sides of \eqref{eq:result3_in_Hessian}. Let $\delta\to 0$, then $(T,n)\to (+\infty,+\infty)$. By \eqref{eq:result4_in_Hessian}, we have for $(x,\xi_0)\in D_{\delta_1}^\lambda\times\mathbb{R}^d\setminus\{0\}$ and $\eta_0=0$
\begin{equation*}
 \frac{\left|u_{(\xi_0)(\xi_0)}(x)\right|^2}{B_3^\sq(x,\xi_0)}\leq \mathop{\sup\limits_{y\in\partial D_{\delta_1}^\lambda,}}_{0\neq\zeta\in\mathbb{R}^d}\frac{\left|u_{(\zeta)(\zeta)}(y)\right|^2}{B_3^\sq(y,\zeta)}+N\mathop{\sup\limits_{y\in\partial D_{\delta_1}^\lambda,}}_{|\zeta|=1} \left|u_{(\zeta)}(y)\right|^2+N_1.
\end{equation*}

\textbf{Step 2.} Repeating the arguments in \textbf{Step 1} for $x\in D_{\lambda^2}$, we have for $(x, \xi_0) \in D_{\lambda^2}\times \mathbb{R}^d\setminus\{0\}$ and $\eta_0=0$,

\begin{equation*}
 \frac{\left|u_{(\xi_0)(\xi_0)}(x)\right|^2}{B_4^\sq(\xi_0)}\leq \mathop{\sup\limits_{y\in\partial D_{\lambda^2},}}_{ 0\neq\zeta\in\mathbb{R}^d}\frac{\left|u_{(\zeta)(\zeta)}(y)\right|^2}{B_4^\sq(\zeta)}+N_1.
\end{equation*}

\textbf{Step 3.} In view of Lemma \ref{le:normalestimate} and Theorem \ref{thm:firstorder}, we have
\begin{eqnarray*}
  &&\displaystyle \lim\limits_{\delta_1\to 0}\mathop{\sup\limits_{x\in\partial D_{\delta_1}^\lambda,}}_{|\zeta|=1}\left|u_{(\zeta)}(x)\right|^2\\
  &\leq &\displaystyle \mathop{\sup\limits_{x\in\partial D,}}_{|l|=1,\ l\|\partial D}\left|u_{(l)}(x)\right|^2+\mathop{\sup\limits_{x\in\partial D,}}_{|n|=1,\ n\bot\partial D}\left|u_{(n)}(x)\right|^2+\mathop{\sup\limits_{\psi(x)=\lambda,}}_{|\zeta|=1}\left|u_{(\zeta)}(x)\right|^2\\
  &\leq&\displaystyle \mathop{\sup\limits_{x\in\partial D,}}_{|l|=1,\ l\|\partial D}\left|g_{(l)}(x)\right|^2+N(|g|_2^2+|f(\cdot,0,0)|_0^2)+N\left(1+\frac{|\psi|_1}{\lambda^{\frac{3}{2}}}\right)(|g|_1^2+\|f(\cdot)\|_{0,1}^2).
 \end{eqnarray*}

With the analogues of Lemmas \ref{le:barriers} and \ref{le:estin2_in_sec4}, we obtain
\begin{equation*}
 \left|u_{(\xi_0)(\xi_0)}(x)\right|\leq N_2\left(|\xi_0|^2+\frac{\psi_{(\xi_0)}^2(x)}{\psi^{\frac{7}{4}}(x)}\right), \quad a.e.\ x\in D,\ \forall \xi_0\in\mathbb{R}^d,
\end{equation*}
where
\begin{equation} \label{eq:N_2}
N_2=N[|g|_{1,1}+\|f(\cdot)\|_{0,1}+[f]_{1,1}(1+|g|_1^2+\|f(\cdot)\|_{0,1}^2)].
\end{equation}

The proof is complete.
\end{proof}

The proof of Theorem \ref{thm:estimate} remains to prove the existence and uniqueness of \eqref{eq:Par2}. Before that, we need the existence, uniqueness and probabilistic interpretation of weak solutions (in Sobolev sense) for the Dirichlet problem of systems of semi-linear degenerate elliptic PDEs, which was inspired by the work of Bally and Matoussi \cite{BAL}. However, due to the length of the paper, we will not include all the arguments here. The outline of the proof is as follows: Firstly, since the coefficients $b, \sigma$ are defined on the whole space, applying \cite[Proposition 5.1]{BAL} and choosing $\varphi:=\varphi\cdot 1_D$, we get a norm equivalence result in a bounded domain. Secondly, according to the proof of \cite[Theorem 2.1]{BAL}, using the approximation procedure (see \cite[page 138]{BAL}), we get a probabilistic interpretation for the solution of the linear elliptic system. Finally, for the semi-linear system, using the norm equivalence result and the probabilistic interpretation for the linear system mentioned above, following the proof of \cite[Theorem 3.1]{BAL}, we have: under the assumptions $(H1)$, $(H4)-(H7)$, with the well-posedness of solutions to random horizon BSDEs (see Lemma \ref{le:BSDE}), there is a unique weak solution $u\in L_\rho^2(\overline{D})$ of \eqref{eq:Par2}.

\begin{proof}[Proof of the existence and uniqueness of \eqref{eq:Par2}]
Let $u$ is given by \eqref{eq:Sol1}. The Lipschitz continuity of the solution $u$ up to the boundary can be proved by Lemma \ref{le:normalestimate} for boundary normal derivative estimates and by $(H6)_1$ for boundary tangential derivative estimates.
Moreover, since $u$ is a weak solution of \eqref{eq:Par2} in $L_\rho^2$, and has second derivatives almost everywhere by \ref{thm:secondorder}, then $u$ given by \eqref{eq:Sol1} satisfies \eqref{eq:Par2} almost everywhere in the space $C_{loc}^{1,1}(D)\cap C^{0,1}(\overline{D})$.

For the uniqueness, if PDE \eqref{eq:Par2} has a weak solution $u$, then we have $u(X_t)=Y_t$ and $(\nabla u\sigma)(X_t)=Z_t$, where $(Y_t,Z_t)$ is the solution of BSDE \eqref{eq:BSDE1}. Then the uniqueness of the weak solutions of PDEs follows from the uniqueness of solutions of FBSDEs.
The proof is complete.
\end{proof}

\begin{Remark} Consider the case of $k=1$. Let $U$ be a separable metric space. By $\mathcal{U}$, we denote the set of progressively measurable processes $\alpha_t$ taking values in $U$. In the above proof,  if we replace $(\sigma(x)$, $b(x)$, $f(x,y,z)$ and $g(x)$ in (\ref{eq:BSDE1}) and (\ref{eq:SDE1}) with $\sigma(\alpha,x)$, $b(\alpha,x)$, $f(\alpha,x,y,z)$ and $g(\alpha,x)$, where $\alpha\in\mathcal{U}$ is the control variable, under appropriate measurable assumptions, the gradient and Hessian estimates (\ref{eq:estimate1}) and (\ref{eq:estimate2}) are still true. In this way, we can get the interior regularity estimates for the solution of the so-called HJB equations, which is nonlinear and degenerate elliptic PDEs in a domain.
\end{Remark}


\end{document}